\documentclass{siamltex}
\usepackage{amssymb,amsfonts,amsmath,subfigure}
\usepackage{color}
\usepackage[left=1in,top=1in,right=1in,bottom=1in]{geometry}
\usepackage{graphicx}
\usepackage{float}

\usepackage{palatino} 
\numberwithin{equation}{section} 

\newtheorem{thm}{Theorem}[section]
\newtheorem{defn}[thm]{Definition}
\newtheorem{cor}[thm]{Corollary}

\newtheorem{rmrk}[thm]{Remark}

\newcommand{\e}{\varepsilon}
\newcommand{\R}{\mathbb{R}}

\newtheorem{prop}[thm]{Proposition}

\newcommand{\abs}[1]{\left\vert{#1}\right\vert}

\newcommand{\ba}{\begin{array}}
\newcommand{\ea}{\end{array}}

\newcommand{\bthm}{\begin{thm}}
\newcommand{\ethm}{\end{thm}}
\newcommand{\bstp}{\begin{stp}}
\newcommand{\estp}{\end{stp}}
\newcommand{\blemma}{\begin{lemma}}
\newcommand{\elemma}{\end{lemma}}
\newcommand{\bprop}{\begin{prop}}
\newcommand{\eprop}{\end{prop}}
\newcommand{\bpf}{\begin{pf}}
\newcommand{\epf}{\end{pf}}
\newcommand{\bdefn}{\begin{defn}}

\newcommand{\edefn}{\end{defn}}
\newcommand{\brk}{\begin{rmrk}}
\newcommand{\erk}{\end{rmrk}}
\newcommand{\bcrl}{\begin{crl}}
\newcommand{\ecrl}{\end{crl}}
\newcommand{\norm}[1]{\left\|#1\right\|}

\newcommand{\beqn}{\begin{equation}}
\newcommand{\eeqn}{\end{equation}}

\renewcommand{\leq}{\leqslant}
\renewcommand{\geq}{\geqslant}

\newcommand{\beq}{\begin{equation}}
\newcommand{\eeq}{\end{equation}}
\newcommand{\bea}{\begin{eqnarray}}

\newcommand{\eea}{\end{eqnarray}}

\newcommand{\dive}{\mathrm{div}\,}
\newcommand{\hdiv}{H_{\dive}( \Omega;\R^2)}
\newcommand{\hdivS}{H_{\dive}( \Omega;\mathbb{S}^1)}
\newcommand{\sh}{\mathcal{H}}
\newcommand{\divcon}{\stackrel{\wedge}{\rightharpoonup} }
\newcommand{\sA}{\mathcal{A}}
\newcommand{\oned}{E^{1D}_\e}
\newcommand{\onedlim}{E_0^{1D}}
\newcommand{\bD}{\mathbb{D}}
\newcommand{\bZ}{\mathbb{Z}}
\newcommand{\nbdsig}{J_u^{\delta}}
\newcommand{\dist}{\mathrm{dist}}
\newcommand{\rot}{\mathrm{curl}}
\newcommand{\trace}{\mathrm{trace}}
\newcommand{\spt}{\mathrm{spt}}
\newcommand{\biglim}{\mathrm{Lim}}

\title{A Ginzburg-Landau type problem for highly anisotropic nematic liquid crystals}

\author{Dmitry Golovaty \thanks{Department of Mathematics, 
University of Akron, Akron, OH 44325, USA} \and Peter Sternberg\thanks{Department of Mathematics, Indiana University, Bloomington, IN 47405}  \and Raghavendra Venkatraman
\footnotemark[2]}

\begin{document}

\maketitle

\begin{abstract}
We carry out an asymptotic analysis of a variational problem relevant in the studies of nematic liquid crystalline films when one elastic constant dominates over the others, namely 
\begin{align*}
\inf E_\e(u)\quad\mbox{where}\quad E_\e(u) := \frac{1}{2}\int_\Omega \left\{\e\,|\nabla u|^2 + \frac{1}{\e} \,(|u|^2 - 1)^2 + L \,(\dive u)^2\right\} \,dx.
\end{align*}
Here $u: \Omega \to \R^2$ is a vector field, $0 < \e \ll 1 $ is a small parameter, and $L > 0$ is a fixed constant, independent of $\e$. 
We identify a candidate for the $\Gamma$-limit $E_0$, which is a sum of a bulk term penalizing divergence and an Aviles-Giga type wall energy involving the cube of the jump in
the tangential component of the $\mathbb{S}^1$-valued nematic director. We establish the lower bound and provide the recovery sequence for this candidate within a restricted class. Then we consider a set of variational problems for $E_0$ arising for a choice of domain geometries and boundary conditions. We demonstrate that the criticality conditions for $E_0$ can be expressed as a pair of scalar conservation laws that share characteristics. We use the method of characteristics to analytically construct critical points of $E_0$ that we observe numerically.
\end{abstract}



\section{Introduction}

Describing the elastic energy in nematic liquid crystal models involves making a choice of the elastic constants appearing as coefficients in front of the various terms penalizing spatial variations. Whether in director theories such as Oseen-Frank, where the unknown is a unit vector $n\in \R^2$ or $\R^3$, or within the Landau-de Gennes $Q$-tensor model where $Q$ is a symmetric, traceless $3\times 3$ matrix \cite{MN,V}, some studies pursue an isotropic, or equal constants, choice where the elastic energy density is given simply by $\abs{\nabla n}^2$ or $\abs{\nabla Q}^2$. Others opt for more generality and consider, for instance, three distinct coefficients multiplying the square of the  divergence and the squares of the components of the curl along and perpendicular to the director, respectively.  However, in response to numerous studies by materials scientists who suggest that interesting morphologies in liquid crystals are related to disparities in the values of the elastic constants  (\cite{DA,ZN}),  here we consider a model variational problem with extreme disparity in elastic constants and explore the implications of this choice of elastic coefficients on the structure of minimizers.

We will focus our study on a problem in two dimensions with a thin film nematic film in mind and so for a bounded, Lipschitz domain $\Omega\subset\R^2$ we consider the following variational problem: 
\begin{align}
\label{energy}
\inf E_\e(u)\quad\mbox{where}\quad E_\e(u) := \frac{1}{2}\int_\Omega \left\{\e\,|\nabla u|^2 + \frac{1}{\e} \,(|u|^2 - 1)^2 + L \,(\dive u)^2\right\} \,dx.
\end{align}
Here $u: \Omega \to \R^2$ is a vector field, $0 < \e \ll 1 $ is a small parameter, and $L > 0$ is a fixed constant, independent of $\e$. 
In general, we will augment \eqref{energy} with Dirichlet boundary conditions $u=g$ on $\partial\Omega$ for given $g:\partial\Omega\to\mathbb{S}^1.$ We point out that in light of the two-dimensional identity
\[
(\dive u)^2+\abs{{\rm curl}\,u}^2=\abs{\nabla u}^2+\;\mbox{null Lagrangian}
\]
it suffices in this study to just penalize the divergence and not to include the curl as well. As $u$ is not a unit vector, \eqref{energy} is not a director model per se but rather bears more resemblance to the Ericksen model with variable degree of orientation \cite{E}. Still it maintains some essential features of both Oseen-Frank and Landau-de Gennes models that we wish to focus on in this investigation. 

In order to orient the reader as to how this energy compares with other more familiar models, we point out that when the positive parameter $L$
is dropped, one is left with precisely the simplified Ginzburg-Landau model
\begin{align}
\label{BBHF}
\frac{1}{2}\int_\Omega \left\{\e\,|\nabla u|^2 + \frac{1}{\e} \,(|u|^2 - 1)^2\right\} \,dx.\tag{BBH}
\end{align}
thoroughly examined in \cite{BBH} under the scaling $\frac{1}{\e}E_\e$. For $\e\ll 1$, minimizers $u_\e$ of that problem are characterized by so-called Ginzburg-Landau vortices with $u_\e\approx f_\e(r)\big(\cos\theta,\pm\sin\theta\big)$ near a zero that carries degree $\pm 1$. On the other hand, formally passing to the limit $L\to\infty$ in \eqref{energy}, one is led to a divergence-free constraint, in which case, at least for simply connected domains $\Omega$, one can introduce a stream function $\psi$ via $\nabla^{\perp}\psi=u$. Then $E_\e$ transforms into
\beq
\frac{1}{2}\int_\Omega \left\{\e| D^2\psi |^2 + \frac{1}{\e} (|\nabla \psi|^2 - 1)^2\right\}\,dx,\tag{AG}\label{AGintro}
\eeq
which is precisely the well-studied Aviles-Giga model, see e.g. \cite{AG,ADM,CD,DO,Ignat,JinKohn,Lorent,NewLorent} and the references therein. Singular structures for that model emerging in the $\e\to 0$ limit take the form of domain walls--generically curves-- across which the normal component of $\nabla \psi$ jumps. Though we do not pursue it in this article, an interesting direction would be to make a rigorous study of the limit $L\to\infty$ in relating our problem to \eqref{AGintro}. We should also mention that there are a multitude of models bearing some resemblance to $E_\e$ coming from the micromagnetics community, including, for instance, the ones studied in \cite{Ignat,ARS,JOP,RS} where the $L^2$-norm of the divergence is replaced by an $H^{-1}$-norm which is then considered with a different scaling.

From this perspective then, our problem rests between the two models \eqref{BBHF} and \eqref{AGintro} and indeed we will find a rich array of singular structures playing a role including Ginzburg-Landau type vortices, which in the scaling of \eqref{energy} are relatively expensive, domain walls which end up contributing $O(1)$ to the energy $E_\e$ and divergence-free vortices of the form $f_\e(r)\,\widehat{e}_\theta$ where $\widehat{e}_\theta:=\big(-\sin\theta,\cos\theta\big)$, whose asymptotic contribution to the energy is zero.

A natural goal is to identify a candidate for the $\Gamma$-limit of the sequence $\{E_\e\}$ as $\e\to 0$
and with this in mind,
a first issue is to determine the appropriate space of competitors for such a limit and to explore what kind of compactness properties hold for sequences of $H^1(\Omega;\R^2)$ functions, say $\{w_\e\}$, satisfying a uniform energy bound $E_\e(w_\e)<C.$ One is naturally led to consider the Hilbert space $\hdiv$ consisting of $L^2$ vector fields having $L^2$-divergence and it is immediate that $\{w_\e\}$ will be weakly compact in this space, with an $\mathbb{S}^1$-valued limit. Such mappings can, in general, have tangential components that jump across curves, though their normal components cannot jump. In Theorem \ref{comp} we note that through a minor modification of the compactness result of \cite{DKMO} one may also show
strong convergence, up to subsequences, in $L^p(\Omega;\R^2)$ for any $p<\infty$; see also \cite{ADM} for an independent proof of compactness in the Aviles-Giga setting.

 From the standpoint of constructing energy efficient sequences, and ultimately recovery sequences for $\Gamma$-convergence, the resolution of a jump in the tangential components of an $\mathbb{S}^1$-valued map, say $w$, across a wall leads one to consider a Modica-Mortola type of heteroclinic connection linking the tangential values $\pm \sqrt{1-(w\cdot \nu)^2}$ across an interface having normal $\nu$. With these heuristics in mind, and denoting the one-sided traces along such a jump set $J_u$ by $u_+$ and $u_-$,
 one is led to a candidate for the $\Gamma$-limit of the form  \begin{equation}
 E_0(u) := \frac{L}{2} \int_{\Omega} (\dive u)^2 \,dx + \frac{1}{6}\int_{J_u\cap\Omega} |u_--u_+|^3 \,d \sh^1 + \frac{1}{6} \int_{\partial \Omega} |u_{\partial\Omega}-g|^3 \,d\sh^1. \label{Ezerointro}
 \end{equation}
 We note that the cubic dependence on the jump across $J_u$ is identical to that found in the asymptotics for \eqref{AGintro}. However,  we also point out the presence of the boundary integral in \eqref{Ezerointro} measuring possible jumps in the tangential component along $\partial\Omega$, a feature of our model not typically found in the Aviles-Giga problem. 
 
 The form of $E_0$ suggests that the space of definition for the $\Gamma$-limit must be a subset of those vector fields in $\hdivS$ having a rectifiable jump set with the cube of the jump in the one-sided traces being integrable. The difficulty lies in the fact that energy-bounded sequences may not have limits lying in the space of functions of bounded variation, an effect first elucidated for (AG) in \cite{ADM}, so identification of a natural space is nontrivial. In \cite{ADM,DO} the authors identify what would appear to be the right space for establishing $\Gamma$-convergence for the Aviles-Giga functional, introducing the notion of  `entropy measures,' but to date the construction of a recovery sequence remains an open problem for (AG). We do not pursue here the interesting question of whether some analog of the results in \cite{DO,LO} on the structure of elements of this new space holds for the energy $E_\e$ in \eqref{energy}. 

Instead we will present arguments for the $\Gamma$-limit lower bound and for the recovery sequence under the {\it assumption} of the limit lying in $\hdivS\cap BV(\Omega;S^1).$ This is the content of Theorem \ref{main}. We note that similar difficulties arise when partial $\Gamma$-convergence results are obtained in micromagnetic models such as \cite{ARS}. Our techniques for proving lower-semicontinuity adapt the Jin-Kohn entropy \cite{JinKohn} and borrow some ideas from \cite{ARS}. For the recovery sequence we adopt the rather ingenious and nontrivial construction of Conti and De Lellis for \eqref{AGintro}, cf. \cite{CD}, with care taken to verify that the divergence term in $E_\e$---not present in \eqref{AGintro}---does not contribute to the energy in a neighborhood of the jump set.

After presenting the arguments for $\Gamma$-convergence within this special class, we turn to the analysis of the behavior of minimizers of the presumed $\Gamma$-limit $E_0$ in various geometries and under various boundary conditions $g$. That is, we want to focus on the question of what kinds of morphologies one should expect to see for very disparate elastic constants and in the process we will develop new tools for carrying out such an investigation.
  
We begin this pursuit by establishing various notions of criticality for $E_0$. In Theorem \ref{critthm} we show that in the bulk, that is, away from the jump set $J_u$, criticality of a vector field $u$ implies that the gradient of divergence lies in the direction of $u$. When $u$ is locally lifted to $u=e^{i\theta}$, this leads to a pair of conservation laws for the phase $\theta$ and the divergence of $u$,
both sharing the same characteristics, cf. Corollary \ref{conservation}. This makes for an interesting comparison with the presumed $\Gamma$-limit of \eqref{AGintro}, where, for example, the authors of \cite{DKMO} exploit the presence of a single conservation law for $\theta$ writing $\nabla^{\perp}u=e^{i\theta}$ where $u$ solves the eikonal equation. We also derive in \eqref{el:4} and \eqref{el:45} a natural boundary condition holding along $J_u$ relating the normal component of $u$ to the jump in the divergence across the wall, and in \eqref{el:13} a criticality condition yielding stationarity of the wall itself that not surprisingly involves its curvature. We use these conditions in the rest of the paper to build critical points for specific examples.

In Section 5 we specialize our study of minimizers of $E_0$  to the case where $\Omega$ is either a disc or an annulus. Depending on choice of the $\mathbb{S}^1$-valued boundary condition $g$, we find that minimizers may or may not develop walls and tend to follow $\widehat{e}_\theta$ as much as possible. In particular, for `hedgehog' boundary conditions $g(\theta)=\widehat{e}_r:=(\cos\theta,\sin\theta)$ in the disc, we can establish an explicit formula for the minimizer as a vector field that behaves like $\widehat{e}_\theta$ near the origin and then unwinds to $\widehat{e}_r$ to accommodate the boundary conditions, cf. Theorem \ref{radialhedge}. This result is reminiscent of a similar observation made in \cite{Helein} for three-dimensional Oseen-Frank model in a ball with hedgehog boundary conditions when divergence is penalized heavily. Perhaps the most interesting case to us is for the disc under the choice $g(\theta)=(\cos\theta,-\sin\theta)$. Here our numerics reveal a rather dramatic dependence of the wall geometry and location on the value of the parameter $L$ and through the three criticality conditions and system of conservation laws derived in Section 4 we are able to build a critical point that appears to capture this complicated morphology, at least in one parameter regime. We conclude this section with an example posed in an annulus where our analysis suggests that in some parameter regime, a minimizer prefers to have a wall that coincides with the boundary.

Finally, in Section 6 we pose the problem of minimizing $E_0$ in a rectangle subject to constant Dirichlet data on the top and bottom of the form $(\pm \sqrt{1 - a^2},a)$ for $a\in [0,1)$ and periodic boundary conditions on the sides. What motivates our choice of periodic boundary conditions is the wish to understand under what conditions the transition from the top to the bottom involves a one-dimensional wall construction as opposed to a more complicated two-dimensional cross-tie type scenario as appears in various micromagnetic studies such as \cite{ARS, DKMO2}. This question was raised and partially addressed for the case of anisotropic
elastic energy, though not `extreme anisotropic' elastic energy in the sense of our present work, in the articles \cite{BF,Go}. 

Our focus at the beginning of this section is to revisit the question of compactness and $\Gamma$-convergence within the one-dimensional context where competitors only vary with $y$. In Theorem \ref{onedcompactness} we show that energy bounded sequences do necessarily have subsequential limits whose third power lie in $BV(-H,H)$, where $2H$ is the height of the rectangle. We then give a complete proof of $\Gamma$-convergence in one-dimension, cf. Theorem \ref{thm1Dmain}. Though, of course, the two-dimensional result given in Theorem \ref{main} applies in particular to the one-dimensional setting, our reason for presenting the one-dimensional proof is that it is considerably simpler while still maintaining many of the essential complications and features of the much more involved two-dimensional lower-semicontinuity argument and recovery sequence construction. 

After then giving a complete characterization of one-dimensional minimizers in Theorem \ref{1dmain} we conclude with a two-dimensional construction of a critical point with cross-ties, again utilizing the criticality conditions and conservation laws. The energy of this critical point is then compared to the minimal one-dimensional energy to reveal in Theorem \ref{not1d} that there exists a finite interval $(L_0,L_1)$ of $L-$values---bounded away from zero---for which the one dimensional minimizers from Theorem \ref{1dmain} do \textit{not} minimize the full two-dimensional $E_0$ energy. Here we use a combination of analysis and simple numerical integration to demonstrate that the $E_0$ energy of our critical point with cross-ties is below the energy of one-dimensional minimizers when $L\in(L_0,L_1)$. Additional numerical simulations of the gradient flow for the energy $E_\varepsilon$ show that the (local) minimizers of $E_\varepsilon$ have similar morphology and energy to our cross-tie construction within the interval of $L$-values where the energy of this construction is lower than that of the one-dimensional minimizers. In fact, these simulations also suggest that a different cross-tie type structure develops as $L$ is increased further and this structure has energy that is still lower than that of the one-dimensional critical point. 

We begin our article with a section introducing notation and recalling key notions regarding the function spaces $\hdiv$ and $BV(\Omega;\R^2).$

\section{Preliminaries} \label{prelim}
Throughout the article, $\Omega \subset \R^2$ will denote a bounded Lipschitz domain. We let $\nu_{\partial \Omega}$ denote the outward pointing unit normal along $\partial \Omega.$ 
\noindent
Two spaces of vector fields that will play a prominent role in our analysis are $BV(\Omega;\R^2)$, the space of vector fields of bounded variation taking values in $\R^2,$ and $\hdiv$, the Hilbert space of $L^2(\Omega;\R^2)$ vector fields having weak $L^2$ divergence.  We will often be interested in vector fields that lie in the intersection of these spaces, and are in addition $\mathbb{S}^1$- valued. 

We recall that a map $u \in BV(\Omega, \R^2)$ is approximately continuous in $\Omega \backslash J_u$ where $J_u$ is the jump set of $u$ and is countably $1-$rectifiable. By rectifiability, we note that $J_u$ is contained in an at most countable union of $C^1$ curves up to an $\sh^1$ null set, where $\sh^1$ denotes one-dimensional Hausdorff measure. We fix a regular orientation of these $C^1$ curves that contain almost all of $J_u,$ and let $(\tau_u,\nu_u)$ denote the approximate unit tangent and unit normals to $J_u$ that respect this orientation. Denoting the half planes 
\begin{align*}
H^{\pm}_{\nu_u(x)} := \{y \in \R^2 : (y - x) \cdot \nu_u(x) \geqslant 0 \mbox{ resp. } \leqslant 0 \},
\end{align*} 
$u$ admits traces along $J_u.$ That is, there exist two measurable functions $u_\pm$ on $J_u$ such that for $\sh^1-$a.e. $x \in J_u,$ we have 
\begin{align*}
\lim_{r \downarrow 0} \frac{1}{r^2} \int_{Q_r(x,\nu_u(x)) \cap H^\pm_{\nu_u(x)}} |u(y) - u_\pm (x)|\,dy = 0,
\end{align*}
with $Q_r(x, \nu_u(x))$ denoting the square of side  length $r,$ centered at $x,$ that has one side parallel to $\nu_u(x)$. 

Now, if $u \in BV \cap \hdiv$, then along the jump set $J_u,$ an application of the divergence theorem shows that one must have $u_+ (x)\cdot \nu_u(x) = u_-(x) \cdot \nu_u(x)$ for $\sh^1-$a.e. $x \in J_u.$  It follows that the jump in $u$ along $J_u$ is equal to the jump in the tangential component of $u$ across $J_u.$ 
 
Concerning the space $\hdiv,$ we recall that elements of $\hdiv$ have a well-defined normal trace on $\partial\Omega$, viewed as a distribution in the Sobolev space $H^{-1/2}(\partial\Omega)$, cf. \cite[Ch. 1]{Te}. This distribution is defined by the integration by parts formula 
\begin{align} \label{eq:IBP}
\langle (u \cdot \nu_{\partial \Omega}), \phi \rangle := \int_\Omega \nabla \Phi \cdot u \,dx + \int_\Omega (\dive u) \Phi \,dx, 
\end{align} 
where $\phi \in H^{1/2}(\partial \Omega),$ and $\Phi$ is an $H^1(\Omega)$ extension of $\phi.$  

We will frequently be concerned with vector fields $u \in BV(\Omega;\mathbb{S}^1) \cap H_{\mathrm{div}}(\Omega;\mathbb{S}^1) $ satisfying $|u(x)| = 1.$ For such vector fields, we in fact have that the distribution $u \cdot \nu$ is induced by an $L^\infty(\partial \Omega) = \big(L^1(\partial \Omega)\big)^\ast$ function. To see this, let $\phi \in \eqref{eq:IBP}$ be an $L^1(\partial \Omega)$ function and let $\Phi \in W^{1,1}(\Omega)$ denote an extension of $\phi$ to $\Omega.$ We again define $\langle (u \cdot \nu_{\partial \Omega}), \phi \rangle$ by the formula \eqref{eq:IBP}. While linearity of $(u \cdot \nu)$ is immediate, its continuity follows by applications of the H\"{o}lder and Sobolev embedding inequalities. It can be checked by an approximation argument that this definition is independent of the extension $\Phi$ of $\phi.$   

For a given $g\in H^{1/2}(\Omega;\R^2)$, we will also denote by $H^1_g(\Omega;\R^2)$ the Sobolev space of $H^1$ vector-valued functions having trace $g$ on $\partial\Omega.$
 
We close this section by briefly recalling differentiability properties of $BV$ vector fields, drawing from \cite[Section 3.9]{AFP}. For any vector field $U \in BV(\Omega;\R^2)$ we let $S_U$ denote the points of its approximate discontinuity set; see \cite[Definition 3.63]{AFP}. Recalling that $S_U$ is the set of points where the approximate limit of $U$ fails to exist, one can decompose the gradient via $DU = D^a U + D^s U,$ where $D^a U$ is absolutely continuous with respect to Lebesgue measure and $D^s U$ is the mutually singular part. The singular part $D^s U$ is further decomposed as $D^s U = D^j U + D^c U$ where $D^j U := D^sU \lfloor J_U$ denotes the jump part of $DU$ and $D^c U = D^s U \lfloor (\Omega \backslash S_U)$ is its Cantor part. 

We also recall the $BV$ chain rule \cite[Theorem 3.96, Pg. 189]{AFP} which we state in a form specialized to our setting. For a $BV$ vector field $U \in BV(\Omega; \R^2)$ and a vector field $F \in C^1(\R^2;\R^2)$, we define the composition $V := F \circ U.$ Then $V$ is a vector field of bounded variation and $DV$ satisfies the decomposition of measures
\begin{equation} 
DV= \tilde{D} V + D^j V,\label{eq:der}
\end{equation}
with
\begin{align}
\tilde{D} V &:= \nabla F(U) \nabla U \mathcal{L}^2 + \nabla F(\tilde{U}) D^c U, \\ \label{eq:jumpder}
D^j V &:= (F(U_+) - F(U_-)) \otimes \nu_U \sh^1 \lfloor J_U,
\end{align}
Here $J_U$ is the jump set of $U,$ the unit vector $\nu_U$ is the approximate unit normal field along $J_U$, $\mathcal{L}^2$ is two-dimensional Lebesgue measure, $D^c U$ is the Cantor part of the derivative $U,$ and $\tilde{U}(x)$ is the approximate limit of $U$ at $x$, for any $x \in \Omega \backslash J_U$ (cf. \cite[Definition 3.63, Pg. 160]{AFP}). Applying traces on both sides of equations \eqref{eq:der}-\eqref{eq:jumpder}, yields 
\begin{align*}
\dive V &= \trace(\tilde{D}V) + \trace(D^j V), \\
\trace(\tilde{D}V)&= \trace\big( \nabla F(U)\nabla U\big)\mathcal{L}^2 + \trace(\nabla F(\tilde{U})) D^cU, \\
\trace(D^jV) &= (F(U_+) - F(U_-))\cdot \nu_U \,d\sh^1 \lfloor J_U. 
\end{align*}
 \section{Compactness and partial $\Gamma$-convergence on a general domain}
 
We begin our rigorous analysis with the following compactness theorem for energy bounded sequences. 
\bthm[Compactness] \label{comp} Assume $\{v_\e\}\subset H^1(\Omega;\R^2)$ satisfies the uniform energy bound 
\[
\sup_{\e>0}E_\e(v_\e)<\infty.
\]
Then there exists a subsequence (still denoted here by $v_\e$) and a function $v\in\hdivS$ such that
\begin{eqnarray}
&&v_\e\rightharpoonup v\quad\mbox{in}\;\hdiv, \label{weakdiv}\\
&& v_\e\to v\quad\mbox{in}\;L^2(\Omega;\R^2) \label{strongL2}
\end{eqnarray}
\ethm
We will write $v_\e\divcon v$ when both \eqref{weakdiv} and \eqref{strongL2} hold.
Property $\eqref{weakdiv}$ is immediate in light of the uniform bound on the $L^2$-norm of the divergence, while \eqref{strongL2} follows from the proof of \cite{DKMO}, Prop. 1.2.
The hypotheses of this proposition from \cite{DKMO} differ from our setting in that their sequence is assumed to be divergence-free whereas ours has the weaker assumption of a uniform $L^2$ bound on the divergence. However, a minor modification of their proof
allows for accommodation of this weaker assumption. 

Before proceeding, we wish to stress that a uniform energy bound does 
{\it not} allow one to conclude that the limit lies in $BV(\Omega;\mathbb{S}^1)$; see the discussion on  \cite[pg. 338-340]{ADM} or Remark \ref{rmknotbv} below. 
Our partial $\Gamma$-convergence result in this section, however, is phrased with this extra assumption.
To this end, we fix boundary data $g\in H^{1/2}(\partial\Omega;\mathbb{S}^1)$ for admissible functions in $E_\e$.
We point out that for a sequence $\{u_\e\} \subset H^1(\Omega;\R^2)$ satisfying $u_\e \cdot \nu_{\partial \Omega} = g \cdot \nu_{\partial \Omega}$, under the topology $u_\e\divcon u$ with $u$ assumed to lie in  $ BV(\Omega, \mathbb{S}^1) \cap H_{\rm{div}} (\Omega, \mathbb{S}^1)$,
it follows that  
\begin{equation}
u_{\partial\Omega}(x)\cdot\nu_{\partial \Omega}(x)=g(x)\cdot \nu_{\partial \Omega}\quad\mbox{for}\; \sh^1-a.e. \;x\;\mbox{on}\; \partial\Omega.\label{normaltrace}
\end{equation}
Here we denote by $u_{\partial \Omega}$ its trace on $\partial \Omega$. Indeed, for any $\phi\in H^1(\Omega)$ the divergence theorem yields
\begin{eqnarray*}
&&
\int_{\partial\Omega}u_{\partial\Omega}\cdot \nu_{\partial\Omega}\,\phi\,d\sh^1(x)=\int_{\Omega}\nabla \phi\cdot u\,dx+
\int_{\Omega}{\rm div}\,u\,\phi\,dx\\
&&=\lim_{\e\to 0}\int_{\Omega}\left\{\nabla \phi\cdot u_\e\,dx+{\rm div}\,u_\e\,\phi\right\}\,dx=
\int_{\partial\Omega}g\cdot \nu_{\partial\Omega}\,\phi\,d\sh^1(x)
\end{eqnarray*}

 Now for any  $u\in BV(\Omega, \mathbb{S}^1) \cap H_{\rm{div}} (\Omega, \mathbb{S}^1)$ such that $u_{\partial\Omega}\cdot\nu_{\partial \Omega}=g\cdot \nu_{\partial \Omega}$ on $\partial\Omega$ we define our candidate $E_0$ for the $\Gamma$-limit of $E_\e$ via
 \begin{equation}
 E_0(u) := \frac{L}{2} \int_{\Omega} (\dive u)^2 \,dx + \frac{1}{6}\int_{J_u} |u_--u_+|^3 \,d \sh^1 + \frac{1}{6} \int_{\partial \Omega} |u_{\partial\Omega}-g|^3 \,d\sh^1. \label{Ezero}
 \end{equation}
 We remark that if one introduces the measurable function $X : J_u \to [0,\pi/2]$ by 
\begin{align*}
X := \frac{1}{2}\min\left|\widehat{u_\pm, u_\mp} \right|
\end{align*}
so that $X$ denotes the minimal half-angle between the unit vectors $u_+$ and $u_-$, then the quantity $|u_--u_+|$ arising  in the $\Gamma$-limit can be equivalently expressed as $2\sin X$. Similarly one can express $|u_{\partial\Omega}-g|$ as $2\sin X_{\partial\Omega}$
where $X_{\partial \Omega} := \frac{1}{2}\min \left| \widehat{u_{\partial \Omega}, g }\right|$. Of course, for all $x\in\partial\Omega$ such that
$u_{\partial\Omega}=g$, the last integral in \eqref{Ezero} vanishes, whereas the condition that $\abs{u_{\partial\Omega}}=1$ along with \eqref{normaltrace} imply that whenever $u_{\partial\Omega}(x)\not=g(x)$, one necessarily has
\[
\abs{u_{\partial\Omega}(x)-g(x)}=2\sqrt{1-(g(x)\cdot\nu_{\partial\Omega}(x))^2}.
\] 
Similarly, another alternative to the expression $ |u_-(x)-u_+(x)|$ is 
\[
|u_-(x)-u_+(x)| = 2\sqrt{1-(u_+(x)\cdot\nu_u)^2} = 2\sqrt{1-(u_-(x)\cdot\nu_u)^2},
\]
where $\nu_u$ denotes the measure-theoretic normal to the jump set $J_u$.

The main result of this section is a $\Gamma$-convergence type of result relating $E_\e$ to $E_0$ under the assumption of $BV(\Omega;\mathbb{S}^1)$ competitors for $E_0.$ 
\bthm\label{main} Let $u \in \hdivS \cap BV(\Omega; \mathbb{S}^1)$ with $u_{\partial \Omega} \cdot \nu_{\partial \Omega} = g\cdot \nu_{\partial \Omega}.$\\
\noindent
(i) If $u_\e \in H^1_g(\Omega, \R^2)$ is a sequence of functions such that $u_\e \divcon u$, then 
\begin{equation}
\liminf_{\e \to 0} E_\e(u_\e) \geq E_0(u).\label{lb}
\end{equation}
\noindent
(ii)  There exists $w_\e \in H^1_g(\Omega;\R^2)$ with $w_\e \divcon u$ satisfying 
\begin{align}
\limsup_{\e \to 0}E_\e(w_\e) = E_0(u). \label{rs}
\end{align}
\ethm
\begin{proof} {\bf (i)} We begin with the proof of lower-semicontinuity \eqref{lb}; see also \cite{AG}. We borrow some technical ingredients of the proof from similar arguments in \cite{ARS}. We will use the following notation 
\[
e_\e(v) := \frac{1}{2} \left( \e |\nabla v|^2 + \frac{1}{\e} (|v|^2 - 1)^2 + L (\dive v)^2 \right)\]
and
\[
E_\e(v, A) := \int_A e_\e(v)\,dx, 
\]
where $A$ is any measurable subset of $\Omega$ and $v \in H^1(\Omega;\R^2).$

\medskip\noindent {\bf{(a)}} We suppose $u \in BV(\Omega) \cap \hdiv $ with $|u(x)| = 1$ for a.e. $x \in \Omega,$ and for now assume $u_{\partial \Omega} = g$ along $\partial \Omega.$ The more general case where $(u_{\partial \Omega} - g)\cdot \nu_{\partial \Omega} = 0$ will be treated in step (c). We also suppose that $u_\e \divcon u$ and assume that $\liminf_{\e \to 0} E_\e(u_\e) \leqslant \Lambda < \infty$ for some $\Lambda > 0$ since otherwise \eqref{lb} is trivial. We must show that 
\begin{align*}
\lim_{\e \to 0} E_\e(u_\e) \geqslant \frac{1}{2} L \int_{\Omega } (\dive u)^2 \,dx + \frac{1}{6}\int_{J_u} |u_+- u_-|^3 \,d \sh^1. 
\end{align*}

{
We recall that the jump set $J_u$ is rectifiable. For simplicity, in Steps (a)-(b) we assume that $J_u$ is in fact a $C^1$ embedded curve $\Gamma$ of finite $\sh^1$ measure. The more general case where $J_u$ is merely rectifiable, and is consequently contained in a countable union of $C^1$ curves up to a null set is treated in Step (c) below. 
}

We let $\delta > 0$ be an arbitrary number and denote $\nbdsig := \{x \in \Omega: \dist(x, J_u ) \leqslant \delta\} $ for a $\delta-$neighborhood of $J_u.$  Putting $ \Omega_\delta := \Omega \backslash \nbdsig,$ we then have 
\begin{align*}
E_\e(u_\e) &\geqslant \frac{L}{2} \int_{\Omega_\delta} (\dive u_\e)^2 + \frac{1}{2} \int_{\nbdsig}  \left\{ \e |\nabla u_\e|^2 + \frac{1}{\e} (|u_\e|^2 - 1)^2 + L (\dive u_\e)^2  \right\}\,dx \\
&= A_\e^\delta (u_\e) + B_\e^\delta(u_\e).
\end{align*}
We note first that for any fixed $\delta > 0,$ by convexity and the resulting lower-semicontinuity, it follows that
\begin{align*}
L \liminf_{\e \to 0} \int_\Omega(\dive u_\e)^2 \,dx \geqslant \liminf_{\e \to 0} L\int_{\Omega_\delta} (\dive u_\e)^2 \geqslant L\int_{\Omega_\delta} (\dive u)^2 .
\end{align*}
The left hand side of the last inequality is independent of $\delta$ and therefore we may let $\delta \to 0$ to conclude that 
\begin{align} \label{Aeps}
\liminf_{\delta \to 0}\liminf_{\e \to 0} A_\e^\delta (u_\e) \geqslant \frac{L}{2}\int_\Omega (\dive u)^2 \,dx. 
\end{align}

\medskip \noindent{\bf (b)} We will next show that
\begin{align} \label{eq:Beps}
\liminf_{\e \to 0} B^\delta_\e(u_\e) \geqslant \frac{1}{6} \int_{J_u} |u_- - u_+|^3 \,d \sh^1 - o_\delta(1),
\end{align}
from which the desired result will follow.  We first assume in this step that $J_u$ is given by a smooth curve $\Gamma$ with closure contained in $\Omega,$ and we continue to assume that $u_{\partial \Omega} = g$ - $\sh^1-$a.e. on $\partial \Omega.$ Along $\Gamma,$ we let $(\tau_u, \nu_u)$ denote the pair of unit tangent and unit normal vectors, oriented directly with the basis $(e_1, e_2)$ to $\R^2.$  

For each point $x \in \Gamma$ which is a Lebesgue point of $u^\pm,$ we obtain $r_x > 0$ such that 
\begin{align} \label{eq:lebpts}
\frac{1}{|\Gamma \cap B(x,r_x)|}\int_{\Gamma \cap B(x,r_x)} |u^\pm(y) - u^\pm (x)|\,d\sh^1 (y) \leqslant \frac{\delta}{|\Gamma|}.
\end{align}
Without loss of generality, we can assume that $r_x < \delta.$ We extract a Besicovitch subcover of $\{B(x, r_x)\}_{x \in \overline{\Gamma}},$ say balls $B_1 := B(x_1, r_1), \cdots, B_N :=  B(x_N, r_N),$ with $x_j \in \Gamma,$ and no more than $C_0$ balls overlapping for some universal constant $C_0.$ We then let $\{\phi_j\}_{j=1}^N$ denote a partition of unity subordinate to this covering. Letting $\pi$ denote the nearest point projection from $\nbdsig$ onto $\Gamma,$ we recall that $u_\e \divcon u,$ in $\Omega.$ Before proceeding with the proof of \eqref{eq:Beps}, we require some preliminary calculations: 
writing $\nu(x_j) = (\cos \theta_j, \sin \theta_j)$ with $\theta_j \in [0, 2\pi),$ we define 
\begin{align*}
\partial_X := \cos \theta_j \partial_x + \sin \theta_j \partial_y,  \hspace{1cm} \partial_Y := -\sin \theta_j \partial_x + \cos \theta_j \partial_y,
\end{align*}
and for $\alpha \in \R$ we set $u_\e^\alpha :=  \begin{pmatrix}
\cos \alpha & -\sin \alpha \\
\sin \alpha & \cos \alpha
\end{pmatrix}u_\e.$ Defining finally the rotated divergence for a vector function $v$ by  $\mathrm{div}_jv := \partial_X v^1 + \partial_Y v^2,$ a simple calculation then shows that 
\begin{align*}
\mathrm{div}_j u_\e^\alpha = \cos(\alpha + \theta_j) \cdot \dive u_\e  - \sin(\alpha + \theta_j) \cdot \rot \, u_\e. 
\end{align*}
In particular, choosing $\alpha = - \theta_j,$ we find 
\begin{align} \label{eq:rotateddive}
\mathrm{div}_j u_\e^{-\theta_j} = \dive u_\e. 
\end{align}
Next, we define the vector field $\Xi$ given by 
\begin{align*}
\Xi (v_1, v_2) = 2 \left(\frac{1}{3} v_2^3 + v_2 v_1^2 - v_2, \frac{1}{3} v_1^3 + v_1 v_2^2 - v_1\right) =:\big(\Xi_1(v_1,v_2), \Xi_2(v_1,v_2)\big). 
\end{align*} 
This vector field was introduced by Jin and Kohn \cite{JinKohn} in their study of the Aviles-Giga problem.
We calculate that  
\begin{align*}
\mathrm{div}_{X,Y} \Xi(v_1,v_2) &=  \big(\partial_X \Xi_1 + \partial_Y \Xi_2 \big)\\ &= 2(|v|^2 - 1)(\partial_X v_2 + \partial_Y v_1) + 4 v_1 v_2 (\partial_X v_1 + \partial_Y v_2 ) .
\end{align*}
Setting $v = u_\e^{-\theta_j}$ in the preceding identity and using \eqref{eq:rotateddive} gives
\begin{align} \label{eq:divvect}
\dive_j \Xi(u_\e^{-\theta_j}) = 2 (|u_\e|^2 -1) (\partial_X v_2 + \partial_Y v_1) + 4 v_1 v_2 \dive u_\e. 
\end{align}
Consequently
\begin{align} \label{eq:jk}
|\dive_j \Xi(u_\e^{-\theta_j})| &\leqslant \frac{1}{\e}(|u_\e|^2 - 1)^2 + \e\big( \partial_X v_2 + \partial_Y v_1)^2 + \frac{4}{L}v_1^2 v_2^2 + L (\dive u_\e)^2.
\end{align}
{
We have
\begin{align*}
\e\left( \partial_X v_2 + \partial_Y v_1\right) ^2 &\leqslant \e|\nabla_{X,Y} v|^2 + 2\e \left( \partial_X v_2 \partial_Y v_1 - \partial_X v_1 \partial_Y v_2 \right) + 2\e \partial_X v_1 \partial_Y v_2\\
&\leqslant \e|\nabla_{X,Y} v|^2 -2\e \mathrm{Jac}(v) + \e\left(2 \partial_X v_1 \partial_Y v_2 +  (\partial_X v_1)^2 + (\partial_Y v_2)^2\right)\\
& = \e|\nabla_{X,Y} v|^2 - 2\e \mathrm{Jac}(v) + \e(\dive_j v)^2,
\end{align*}
with $\mathrm{Jac}(v)$ denoting the Jacobian determinant of $Dv.$ 
Using this in \eqref{eq:jk} along with \eqref{eq:rotateddive} we find 
\begin{align*}
|\dive_j \Xi(u_\e^{-\theta_j})| \leqslant e_\e(u_\e) + \frac{4}{L}(u_\e^1u_\e^2)^2 - 2\e \mathrm{Jac}(Dv) + \e(\dive u)^2
\end{align*}
}

For any smooth nonnegative compactly supported function $\psi \in C_c^\infty(\Omega),$ we then find 
\begin{align*}
\int_\Omega \psi |\dive_j \Xi(u_\e^{-\theta_j})| &\leqslant 2\int_\Omega \psi e_\e(u_\e) + \e \int_\Omega \psi (\dive u_\e)^2 \\ &+ \frac{4}{L} \int_\Omega \psi v_1^2 v_2^2 + 2 \e \int_\Omega \psi \dive_{X,Y} (v_2 \partial_Y v_1, - v_2 \partial_X v_1) .
\end{align*}
Integrating the last term by parts and using the H\"{o}lder's inequality, we find that 
\begin{multline*}
\limsup_{\e \to 0} 2\e \left|\int_\Omega \psi \dive_{X,Y} (v_2 \partial_Y v_1, - v_2 \partial_X v_1) \right| \\ = \limsup_{\e \to 0} 2\e \left|\int_\Omega  \nabla_{X,Y} \psi \cdot  (v_2 \partial_Y v_1, - v_2 \partial_X v_1)\right| \leqslant \limsup_{\e \to 0} C (\psi)\e^{1/2} = 0. 
\end{multline*}
By the energy bound, we also know that $\e \int_\Omega \psi (\dive u_\e)^2 \leqslant \|\psi\|_{L^\infty}\Lambda \e.$ 
Continuing with the proof of the lower bound we obtain 
\begin{align} \notag
\Lambda \geqslant \liminf_{\e \to 0} E_\e(u_\e, \nbdsig) &= \sum_{j=1}^N \liminf_{\e \to 0} \int \phi_j(\pi(y)) e_\e(u_\e)\,dy \\ \notag
&\geqslant \frac{1}{2}\sum_{j=1}^N \liminf_{\e \to 0} \int \phi_j (\pi(y)) \left(\left|\dive_j (\Xi(u_\e^{-\theta_j})) \right| - \frac{1}{L }(u_\e)_1^2 (u_\e)_2^2\right) \,dy \\ \label{eq:lwrbndjump1}
&\geqslant \frac{1}{2}\sum_{j=1}^N \int \phi_j(\pi(y)) |\dive_j (\Xi(u^{-\theta_j}))| - \frac{1}{L}r_j^2 
\end{align}
where in the last step we have used the standard properties of weak convergence of the sequence of measures $\{|\dive_j \Xi(u_\e^{-\theta_j})|\,dy\}_{\e > 0}$ for each fixed $j = 1, \cdots, N.$ Next using the $BV$ chain rule stated in section \ref{prelim} and recalling that the measure $D^c U$ is mutually singular with respect to Lebesgue measure, we find that for any non-negative,  bounded Borel function $\psi,$ and any Borel set $A \subset \Omega$, 
\begin{align} \notag
\int_A \psi |\dive V| &= \int_{A \backslash (\spt(D^cU) \cup J_U)} \psi |\trace\big(\nabla F(U) \nabla U \big)|\,dy \\ \notag & + \int_{A \cap \spt(D^c U)} \psi |\trace\nabla F(\tilde{U})|d D^cU \\ \label{eq:measureineq} & + \int_{A \cap J_U} \psi\left| \left(F(U_+) - F(U_-)\right) \cdot \nu_U\right|\,d\sh^1  \\ \notag
&\geqslant \int_{A \backslash (\spt(D^cU) \cup J_U)} \psi |\trace\big(\nabla F(U) \nabla U \big)|\,dy \\ \notag &+ \int_{A \cap J_U} \psi\left| F(U_+) - F(U_-) \cdot \nu_U\right|\,d\sh^1.
\end{align}
 We now apply inequality \eqref{eq:measureineq} to \eqref{eq:lwrbndjump1} with $A = \nbdsig $ for some fixed $\delta > 0,$ $F = \Xi^{\theta_j}, U = u^{-\theta_j}$ for each $j = 1,\cdots, N,$ and $\psi = \phi_j\circ \pi.$ Since by \eqref{eq:rotateddive} we have that $\dive_j \big(\Xi(u^{-\theta_j})\big) = \dive_j \Big(\big(\Xi(u^{-\theta_j})\big)^{\theta_j}\Big)^{-\theta_j} = \dive \Big(\big( \Xi(u^{-\theta_j})\big)^{\theta_j}\Big),$ we obtain
\begin{multline*}
\Lambda \geqslant \liminf_{\e \to 0} E_\e(u_\e, \nbdsig) \geqslant
\frac{1}{2}\sum_{j=1}^N \int \phi_j(\pi(y)) |\dive_j (\Xi(u^{-\theta_j}))| - \frac{1}{L} r_j^2 \\
\geqslant \frac{1}{2}\sum_{j=1}^N \int_{\nbdsig \backslash \spt(D^c u^{-\theta_j}) \cup J_u} \phi_j(\pi(y)) |\trace(\nabla \Xi^{\theta_j}(u^{-\theta_j}) \nabla (u^{-\theta_j})|\,dy \\ + \int_{J_u} \phi_j(\pi(y)) \Big|\left(\Xi\Big(\big(u^{-\theta_j}\big)_+\Big)- \Xi\Big(\big(u^{-\theta_j}\big)_-\Big)\right) \cdot \nu^{-\theta_j} \Big|\,d\sh^1 - \frac{1}{L} r_j^2.
\end{multline*}
The first term above, corresponding to the absolutely continuous part of the divergence measure, is nonnegative and in fact, $o_\delta(1)$ as $\delta \to 0,$ by the monotone convergence theorem since the integrand, along with the summation is an $L^1$ function in $\nbdsig \backslash (\mbox{spt }  D^c u \cup J_u).$ Also,  we can estimate the sum $\sum_{j=1}^N r_j^2 \leqslant \left( \max_{1 \leqslant j \leqslant N} r_j\right) (r_1 + \cdots + r_N) \leqslant C_0 |\Gamma| \left( \max_{1 \leqslant j \leqslant N} r_j\right) \leqslant C \delta \to 0$ as $\delta \to 0.$ As for evaluating the jump term, we observe that at the point $x_j,$ since $\nu^{-\theta_j} = (1,0),$ we have 
\begin{align*}
\Xi\Big( \big(u^{-\theta_j}\big)^+ \Big) - \Xi\Big( \big(u^{-\theta_j}\big)^- \Big) \cdot (1, 0) = \frac{4}{3}(u \cdot \nu^\perp)^3 = \frac{4}{3} \sin^3 X(x_j) = \frac{1}{6}|u_+(x_j) - u_-(x_j)|^3.
\end{align*}
Therefore, by choice of $x_j,$ \eqref{eq:lebpts}, and also the $C^1$ nature of the curve $\Gamma = J_u,$ 
\begin{align*}
\liminf_{\e \to 0} E_\e(u_\e,\nbdsig) \geqslant \frac{1}{6} \sum_{j=1}^N \int_{J_u} |u_+ - u_-|^3\,d\sh^1  - C\delta - o_\delta(1). 
\end{align*}
Letting $\delta \to 0,$ we conclude the proof of the lower bound in the case when $J_u$ is a single smooth curve and $u_{\partial \Omega} = g$ along the boundary. 

\medskip \noindent {\bf{(c)}} Next we address \eqref{eq:Beps} in the full generality of $J_u$ being merely rectifiable and $(u_{\partial \Omega} - g)\cdot \nu_{\partial \Omega} = 0$. To unify our arguments to include the contribution of the boundary integral, we define $J := J_u \cup \partial \Omega,$ which is a new rectifiable set. We use the convention that along $\partial \Omega,$ the inner trace is given by $u_{\partial \Omega}$ and the ``outer trace" by $g.$ By rectifiability, we know that $J = \bigcup_{k=1}^\infty \Gamma_k \cup \Gamma^0$ with $\Gamma_k$ being $C^1$ embedded curves and $\sh^1(\Gamma^0) = 0.$ We fix an arbitary $\delta > 0$ and select an integer $N = N_\delta$ such that 
{

\begin{align}
\frac{1}{6}\int_{J} |u_--u_+|^3 \,d \sh^1 \leq \sum_{k=1}^{N_\delta} \frac{1}{6}\int_{\Gamma_k} |u_--u_+|^3 \,d \sh^1+\delta\label{finite} . 
\end{align}
}
By low dimensionality of the overlaps of the curves $\Gamma_1, \cdots, \Gamma_N,$ outside of an at most countable collection of balls $D_j$ of radius $\frac{\delta}{2^j},$ it represents no loss of generality in assuming that the curves $\big(\overline{\Gamma_k}\big)_{k=1}^N$ are disjoint, $C^1$ embedded curves. 
{
 We denote by $\beta_\delta$ the minimal separation given by
\beqn
\beta_\delta:=\min_{k,k'\in\{1,2,\ldots,N_\delta\},\;k\not=k'}\rm{dist}\;(\overline{\Gamma}_k,\overline{\Gamma}_{k'}).\label{beta}
\eeqn
Then for each $k\in\{1,2,\ldots,N_\delta\}$ we introduce an open neighborhood $J_u^k$ of $\Gamma_k$ via 
\[
J_u^k=\left\{x\in\Omega:\,{\rm dist}\,(x,\Gamma_k)<\min\{\frac{\beta_\delta}{2},\frac{\delta}{\mathcal{H}^1(\Gamma_k)k^2}\}\right\}
\]
From \eqref{beta} we see that these neighborhoods are disjoint and we also note that
$\abs{J_u^k}\leq \frac{\delta}{k^2}$ so that
\beqn
\abs{\bigcup_{k=1}^{N_\delta}J_u^k}\leq C\delta.\label{leb}
\eeqn
where here $\abs{\,\cdot\,}$ denotes Lebesgue measure.
\noindent
Now
\begin{equation}
E_\e(u_\e) \geq \frac{L}{2} \int_{\Omega\backslash \cup_{k=1}^{N_{\delta} }J_u^k} (\dive u_\e)^2 + \int_{\cup_{k=1}^{N_{\delta} }J_u^k} e_\e(u_\e)\,dx\label{brk1}
\end{equation}
and by convexity and the resulting lower-semicontinuity, it follows that
\begin{align*}
 \liminf_{\e \to 0}  \int_{\Omega\backslash \cup_{k=1}^{N_{\delta} }J_u^k} (\dive u_\e)^2 \,dx  \geqslant  \int_{\Omega\backslash \cup_{k=1}^{N_{\delta} }J_u^k} (\dive u)^2 .
\end{align*}
Hence,
condition \eqref{lb} will follow from \eqref{finite} and \eqref{leb} by letting $\delta$ approach $ 0$ once we can establish that
\begin{equation} 
\liminf_{\e\to 0}\int_{ \cup_{k=1}^{N_{\delta} }J_u^k}e_\e(u_\e)\,dx\geq  \sum_{k=1}^{N_\delta} \frac{1}{6}\int_{\Gamma_k} |u_--u_+|^3 \,d \sh^1-O(\delta)
\label{jump1}
\end{equation}
\noindent
Appealing now to our work in Steps \textbf{(a)-(b)} we find that for each $k = 1, \cdots N_\delta$ we have 
\begin{align*}
\liminf_{\e \to 0} \int_{J_u^k} e_\e(u_\e) \,dx \geqslant \frac{1}{6}\int_{\Gamma_k} |u_- - u_+|^3 \,d \sh^1 - C|J_u^k|. 
\end{align*}
In the last inequality, $C$ is independent of $k.$ Summing over $k = 1, \cdots, N_\delta,$ using the disjointedness of the neighborhoods $J_u^k$ along with \eqref{leb} yields the inequality \eqref{jump1}, completing the proof of the lower bound in this general case. 
}

\bigskip
\noindent {\bf (ii)} The proof of \eqref{rs} follows the approach of \cite{CD} rather closely, and therefore we present only an outline of the argument, highlighting the steps that are different for our problem by focusing primarily on the treatment of the divergence term in the energies. To facilitate comparison with \cite{CD}, we adopt the notation of that proof wherever possible.

\medskip\noindent
{\bf (a)} \underline{Preparation:} We let $\phi: \R^2 \to [0,1],$ be a smooth radially-symmetric bump function with $\int \phi = 1$ and $\spt(\phi) \subset B_1.$ For any $\e > 0,$ we denote as usual $\phi_\e(\cdot) := \frac{1}{\e^2} \phi(\frac{\cdot}{\e})$ and set $u_\e := u\ast \phi_\e.$ 

We next introduce a class of step functions. For any $x_0 \in \R^2, $ $\nu \in \mathbb{S}^1 $ and $\theta \in \R,$ we introduce the function 
\begin{align*}
s_{x_0,\nu,\theta} (x) := (\cos \theta)\nu - \sin \theta H\big( (x- x_0)\cdot\nu \big) \nu^\perp,
\end{align*} 
where $H(t) = 1$ for $t > 0$ and $H(t) = -1$ for $t < 0.$ We then let $\mathcal{S}_{x_0}$ denote the collection of all such step functions at $x_0.$

For $u$ as in the statement of the Theorem, we let $x_0 \in J_u$ be a point at which its approximate unit normal $\nu_u$ is defined and we consider $s_{x_0} \in \mathcal{S}_{x_0}$ such that $s_{x_0}^\pm = u_\pm(x_0)$ and $\nu = \nu_u.$ We point out that this choice of $s$ depends on the given function $u.$ To alleviate notation therefore, we just denote one subscript rather than all three. 

Fixing now $\e > 0, \eta > 0, k \geq 1$ and $\overline{\theta} > 0,$ we define ``good points'' on the jump set $J^g(\overline{\theta},k,\eta,\e)$ to be those $x_0 \in J_u$ such that

\medskip\noindent$\bullet$ The step function $s_{x_0}$ associated to $x_0$ satisfies $|\sin \theta| \geqslant \sin \overline{\theta},$ and 
\begin{align} \label{g1}
\| \nabla u\|(B_{2k\e}(x_0)) \geqslant k \e \sin \overline{\theta},
\end{align}
and 
\begin{align} \label{g1.5}
\frac{1}{|B_{2k\e}|}\int_{B_{2k\e}(x_0)} |u - s_{x_0}| \,dx \leqslant \eta.
\end{align}
$\bullet$ For the finitely many balls $B_\e(y) \subset B_{2k\e}(x_0)$ with $y \cdot \nu = x_0 \cdot \nu$ and $(y-x_0)\cdot \nu^\perp \in 2 \e \mathbb{Z},$ one has 
\begin{align} \label{g2}
\int_{B_\e(y) \cap J_u} |[u]|^3 \,d \sh^1 \geqslant |2 \sin \theta|^3 2 \e - \eta \e. 
\end{align}
We denote $\Omega^g := \{x \in \Omega: \dist(x, J^g) < k\e /2\}$ and set $\Omega^{(\e)} := \{x \in \Omega : \dist (x, \partial \Omega) > \e\}.$ 

For any $A \subset \R^2$ and $w \in H^1(A),$ it is also convenient to introduce the notation 
\begin{align*}
F_\e[w;A] := \int_A \e|\nabla w|^2 + \frac{1}{\e} (|w|^2 - 1)^2 \,dx. 
\end{align*}

\medskip\noindent
{\bf (b)} \underline{Estimates away from $\Omega^g.$ } In this step, we show 
\begin{align} \label{eq:domaindecomp}
\mbox{Lim } E_\e [u_\e; \Omega^{(\e)} \backslash \Omega^g ]:= \varlimsup_{\overline{\theta} \downarrow 0}\, \varlimsup_{k \uparrow \infty}\, \varlimsup_{\eta \downarrow 0}\, \varlimsup_{\e \downarrow 0}\, E_\e [u_\e; \Omega^{(\e)} \backslash \Omega^g ] = \frac{L}{2} \int_\Omega (\dive u)^2 \,dx. 
\end{align}
This statement is the analog of \cite[Proposition 1]{CD}. At its heart, the argument relies on a scale-invariant Poincar\'{e} inequality, which asserts that for any $\delta > 0,$ denoting $v_\delta := v \ast \phi_\delta,$ we have 
\begin{align} \label{eq:poincare}
\left(\int_{B_\delta} |v - v_\delta(0)|^2 \,dx \right)^{1/2} \leqslant c \|Dv\|(B_\delta)
\end{align}
for every $v \in BV(B_\delta),$ where $c > 0$ is independent of $\delta.$ Immediate consequences of the Poincar\'{e} inequality are the following linear and quadratic estimates: for every $k \geqslant 1,$ we have 
\begin{align*} 
\mbox{ (linear:) } \hspace{1cm} F_\e[u_\e; B_{k\e}] &\leqslant C \|Du\|(B_{2k\e}) ,\\
\mbox{ (quadratic:) } \hspace{1cm}  F_\e[u_\e; B_{k\e}] &\leqslant \frac{C}{\e} \left(\|Du\|(B_{2k\e})\right)^2,
\end{align*}
with the constant $C$ being independent of $\e, k.$  The proof of \eqref{eq:domaindecomp} proceeds by partitioning the set $\Omega^{(\e)} \backslash \Omega^g$ according to how $\|Du\|(B_{2k\e}(x))$ scales in $k\e.$ On most of $\Omega,$ where the scaling is sublinear, one uses the quadratic estimate to show vanishing of the $F_\e$ energy, while away from the jump set where the scaling of the total variation measure $\|Du\|$ is linear or superlinear, one uses the linear estimate, along with fine properties of $BV$ functions to argue that once again, the $F_\e$ energy vanishes. We refer the reader to \cite[Proposition 1]{CD} for further details.

\medskip \noindent
{\bf (c)} \underline{Estimates within $\Omega^g.$ } Having shown that the energy of the mollification $u_\e$ outside of the set $\Omega^g$ is asymptotically just the bulk divergence, we simply set our desired recovery sequence $w_\e := u_\e$ on $\Omega \backslash \Omega^g.$ We next define $w_\e $ in $\Omega^g$ in order to capture the wall energies in the limit. To this end, let $\mathcal{F}^j = \{B_{2k\e}(x_i^j)\}_i,$ for $1 \leqslant j \leqslant N$ be $N$ families of disjoint balls with $x_i^j \in J^g$ and the $B_{k\e}(x_i^j)$ cover $\Omega^g.$ Here $N$ is a universal constant obtained from Besicovitch's covering theorem. For fixed $k,$ let $\psi \in C_c^\infty(B_k)$ denote a smooth cut off function such that $\psi \equiv 1$ on $B_{k-1}.$ For every $\e > 0,$ we define $\psi^\e \in C_c^\infty(B_{k\e})$ by the formula $\psi^\e (x) := \psi \big( \frac{x}{\e}\big). $ 

Setting $v^0 := u,$ we inductively define $\{v^j\}_{j=1,\cdots, N}$ as follows. At the $j^{\mathrm{th}}$ step, on the family of balls $\mathcal{F}^j,$ we define 
\begin{align*}
v^j(x) := \left\{
\begin{array}{ll}
\big( 1 - \psi_\e(x - x_i^j)\big) v^{j-1} (x) + \psi_\e(x - x_i^j) s_i^j(x), & \mbox{ if } x \in B_{k\e}(x_i^j) \mbox{ for some } i, \\
v^{j-1}(x) & \mbox{ otherwise}. 
\end{array}
\right.
\end{align*}
Here $s_i^j$ is the simple function associated to $u$ at $x_i^j.$ Set then $v := v^N \ast \phi_\e.$ 

For every $i,j$ we define $R_i^j$ to be the largest rectangle of the form $a < (x - x_i^j)\cdot (\nu^i_j)^\perp < b, |(x-x_i^j)\cdot \nu |< \sqrt{k} \e $ where $a,b \in \R$ to be contained in the ball $B_{(k-2)\e} (x_i^j)$ without intersecting any ball $B_{(k+1)\e}(x_{i^\prime}^{j^\prime})$ with $j^\prime > j.$ Existence of such a rectangle is immediate; for the proof of uniqueness of the rectangle $R_i^j,$ we refer the reader to the geometric argument in \cite[Proposition 2]{CD} 
. The main estimates of the present step correspond to \cite[Proposition 2]{CD}: 
\begin{align} \label{outsiderectangles}
\mbox{Lim } E_\e\left[v, \Omega^{(\e)}\backslash \bigcup_{i,j} R_i^j\right] = \frac{L}{2}\int_\Omega (\dive u)^2 \,dx. 
\end{align}
On each $R_i^j$ one has $v = \phi_\e \ast s_i^j$ and 
\begin{align} \label{jumpupbnd}
\mbox{Lim } \frac{1}{6}\sum_{i,j} \int_{R_i^j \cap J_{s_i^j}} |[s_i^j]|^3 \,d \sh^1 \leqslant \frac{1}{6} \int_{J_u} |u_+ - u_-|^3 \,d \sh^1. 
\end{align}
The proof of the assertions that $v = \phi_\e \ast s_i^j$ on $R_i^j$ and of estimate \eqref{jumpupbnd} follow exactly as in \cite{CD}. The key idea is of course to use the fact that each $x_i^j$ is a good point on the jump set, so that we can invoke \eqref{g2}. For any ball $B^j_{il}$ of the type considered in \eqref{g2}, one has the estimate 
\begin{align*}
2\e |2 \sin\theta|^3 \leqslant \int_{B_{ijl} \cap J_u} |[u]|^3 \,d \sh^1 + \eta \e. 
\end{align*}
Since the balls $B^j_{il}$ are disjoint, and $l \leqslant 2k,$ we find  using \eqref{g1} that 
\begin{align*}
\sum_{i,j} \int_{R_i^j \cap J_{s_i^j}} |[s_i^j]|^3 \,d \sh^1 &\leqslant \sum_{ijl} \int_{B^l_{ij} \cap J_u} |[u]|^3 \,d \sh^1 + \sum_{i,j} 2k \eta \e \\
&\leqslant \int_{\Omega \cap J_u} |[u]|^3 \,d \sh^1 + \sum_{ij} \frac{2\eta}{\sin \overline{\theta}} \|\nabla u\|(B_{2k\e}(x_i^j)) \\
& \leqslant \int_{\Omega \cap J_u} |[u]|^3 \,d \sh^1 + \frac{2N\eta}{\sin \overline{\theta}} \|\nabla u\|(\Omega),
\end{align*}
where once again, $N$ is the Besicovitch constant. The result \eqref{jumpupbnd} follows by applying $\biglim.$ We now turn to the proof of \eqref{outsiderectangles}. In light of our work in step (b) above, it suffices to prove the estimates
\begin{gather} \label{5.1}
\mbox{Lim } F_\e \left[ v ; \bigcup_{i,j} B_{(k+1)\e}(x_i^j) \backslash B_{(k-2)\e} (x_i^j)\right] = 0,  \\  \mbox{Lim } F_\e \left[ v , \bigcup_{i,j} B_{(k-2)\e}(x_i^j) \backslash R_i^j\right] = 0
\end{gather}
and
\begin{equation}
\label{5.3}
\mbox{Lim } \int_{\cup_{i,j} B_{(k+1)\e}(x_i^j) } (\dive v)^2 \,dx = 0. 
\end{equation}
The proof of \eqref{5.1} is identical to the proof of Equations (4.3) and (4.4) in \cite{CD} to which we refer the reader. We prove \eqref{5.3}. A basic estimate in the proof of \eqref{5.1} used in \cite{CD} is the inequality 
\begin{align} \label{eq:eta}
\frac{1}{|B_{2k\e}|}\int_{B_{2k\e}(x_i^j)} |v^J - s_i^j|\,dx \leqslant C \eta,
\end{align}
holding for each fixed $i, j$ and each $J = 0, \cdots, N.$ This inequality is proved by induction on $J.$ By testing against arbitrary $L^2$ functions, it is easy to check that $\dive s_i^j = 0$ for each $i,j$. Attributing each $x$ in the union $\bigcup_{j=1}^N B_{(k+1)\e}(x_i^j)$ to the level $j$ where $v(x)$ was last modified, i.e. to the largest $j$ such that $x \in B_{(k+1)\e}(x_i^j)$ for some $i$ in the $j^{th}$ family, we find inductively that
\begin{multline*} 
\frac{1}{2}\int_{\bigcup_{i,j} B_{(k+1)\e}} (\dive v)^2 \,d x \leqslant \sum_{i,j} \int_{B_{k\e}(x_i^j)} \left|\nabla \psi_\e(x - x_i^j) \cdot (v^{j-1} - s_i^j) \ast \phi_\e\right|^2 \,dx \\ + \int_{B_{(k+1)\e}(x_i^j)} (\dive v^{j-1} \ast \phi_\e)^2 \,dx \\ 
(\mbox{\small{Young's inequality}})\;\leqslant  \sum_{i,j} \int_{B_{k\e}(x_i^j)} \left| \nabla \psi_\e(x-x_i^j) \cdot (v^{j-1} - s_i^j) \right|^2\,dx \\ + \int_{B_{(k+1)\e}(x_i^j)} (\dive v^{j-1})^2 \,dx  \\ 
 (\mbox{\small{proceeding inductively}})\; \leqslant N \sum_{i,j} \int_{B_{k\e}(x_i^j)} \frac{1}{\e^2} |v^{j-1} - s_i^j|^2 \,dx \\ + N\int_{\Omega^g} (\dive u)^2 \,dx \\ 
 (\mbox{since } |v^j|,|s_i^j| \leqslant 1 )\;  \leqslant 8\pi N k^2\sum_{i,j} \frac{1}{|B_{2k\e}|} \int_{B_{2k\e}(x_i^j)} |v^{j-1} - s_i^j|\,dx  \\ +  N\int_{\Omega^g} (\dive u)^2 \,dx \\ 
 (\mbox{ by }\eqref{eq:eta} \; )\; \leqslant 8\pi Nk^2 \eta + N \int_{\Omega^g} (\dive u)^2 \,dx. 
\end{multline*}
Since the foregoing estimates are uniform in $\e,$ we send $\e \downarrow 0, \eta \downarrow 0,k \uparrow \infty$ and $\overline{\theta} \downarrow 0$ in that order, to arrive at \eqref{5.3}, where for the second integral we have applied the monotone convergence theorem. 

\medskip\noindent {\bf (d)} \underline{Estimates within the rectangles:}
Finally, it remains to modify the construction $v$ from the preceding step within the boxes $R_i^j.$ This step relies on the \\
{\bf Claim:} There is a smooth function $w_\e$ such that $w_\e = v$ outside $R_i^j$ and 
\begin{align} \label{eq:rectangle}
F_\e[w_\e;R_i^j] \leqslant \ell_{ij} \frac{1}{6}|2\sin \theta|^3 + C \e \left( \frac{1}{\overline\theta} + k e^{-\overline{\theta}\sqrt{k}}\right). 
\end{align}
Here $\theta$ is the angle of the step function $s_i^j$ and $\ell_{ij} = \sh^1(J_{v^N} \cap R_i^j)$ the length of the rectangle. The proof of this claim follows by using a standard ``Modica-Mortola'' type heteroclinic within the rectangle along with a linear interpolant as in step (c) to match the boundary conditions. Control on the divergence term follows as in step (c), and control of the remaining terms proceeds as in \cite{CD}. Briefly, within each rectangle, we have using \eqref{g2}, 
\begin{multline*}
F_\e [w_\e;R_i^j] \leqslant \frac{1}{6} |2\sin \theta|^3 \sh^1(J_u \cap R_i^j) + \frac{C\e}{\overline{\theta}} + C\e k e^{-\overline{\theta}\sqrt{k}} \\
\leqslant \frac{1}{6} \int_{J_u \cap R_i^j} |[u]|^3 \,d \sh^1 + C\e + k\e \eta + \frac{C\e}{\overline{\theta}} + C\e k e^{-\overline{\theta} \sqrt{k}}.
\end{multline*}
In the above estimate we have used fact that the rectangle $R_i^j$ contains no more than $k$ disjoint balls of the type in \eqref{g2}, and that the sum of their diameters is at least $\sh^1(J_{v^N} \cap R_i^j) - 4\e.$ 
Summing over the rectangles $R_i^j,$ we find using \eqref{g1} that 
\[
F_\e[w_\e;R_i^j] \leqslant \frac{1}{6}\int_{J_u}|[u]|^3 \,d \sh^1 + \sum_{i,j} \left( \frac{C(\overline{\theta})}{k} + C(\overline{\theta}) \eta + C e^{-\overline{\theta}\sqrt{k}} \right) \|\nabla u\|(B_{2k\e}(x_i^j)). 
\] 
Taking $\biglim,$ we complete the requisite estimates, and the proof of the recovery sequence construction follows now by a diagonalization procedure. 

\end{proof}

\section{Criticality conditions and solution via characteristics for the limiting energy $E_0$}

We begin this section by identifying the free boundary problem satisfied by critical points of the limiting functional $E_0$, cf.\eqref{Ezero}. 
We will use the criticality conditions derived below to later construct critical points for specific domains $\Omega$ and with specific boundary data $g$.
\bthm\label{critthm}
Consider any $u\in BV(\Omega,  \mathbb{S}^1) \cap H_{\rm{div}} (\Omega,  \mathbb{S}^1)$ such that $u_{\partial\Omega}\cdot\nu_{\partial \Omega}=g\cdot \nu_{\partial \Omega}$ on $\partial\Omega$. Denote by $J_u$ its jump set. Then if the first variation of $E_0$ evaluated at $u$ vanishes when taken with respect to perturbations compactly supported in $\Omega\setminus J_u$, one has the condition
 \begin{equation}
\label{el:5}
u^\perp\cdot\nabla\dive u =0\mbox{ holding weakly on }\Omega\backslash J_u,
\end{equation}
where $u^\perp=\left(-u_2,u_1\right)$.

Furthermore, if the first variation vanishes at $u$ when taken with respect to perturbations that fix $J_u$ and are supported within any ball centered at a smooth point of $J_u\cap\Omega$ and if the traces $\dive u_+$ and $\dive u_-$ are sufficiently smooth, then one has the condition
\begin{equation}
\label{el:4}
L\left[\dive u\right]+4{\left(1-(u\cdot\nu_u)^2\right)}^{1/2}(u\cdot\nu_u)=0\mbox{ on }J_u\cap\Omega,
\end{equation}
where $[\cdot]=\cdot_+-\cdot_-$ represents the jump of across $J_u$ and $\nu_u$ is the unit normal to $J_u$ pointing from the $+$ side of $J_u$ to the $-$ side. If, on the other hand, the perturbations are supported in $B\cap\Omega$ for a ball $B$ containing an arc of $J_u\cap\partial\Omega$ and if these perturbations again fix $J_u$,  then one has
\begin{equation}
\label{el:45}
L\,\dive u+4{\left(1-(g\cdot\nu_{\partial\Omega})^2\right)}^{1/2}(g\cdot\nu_{\partial\Omega})=0\mbox{ on }J_u\cap\partial\Omega,
\end{equation}
provided the trace $\dive u$ on $J_u\cap\partial\Omega$ is sufficiently smooth.

Finally, a vanishing first variation of $E_0,$ evaluated at $u$ that allows for local perturbations of the jump set $J_u\cap\Omega$ itself, leads to the condition
\begin{multline}
 \label{el:13}
(\dive u_+)^2-(\dive u_-)^2+\left(\dive u_++\dive u_-\right)^\prime\left(u_+\cdot \tau_u-u_-\cdot \tau_u\right)\\=\frac{8\kappa}{3L}{\left(1-(u\cdot\nu_u)^2\right)}^{1/2}{\left(1+2\,(u\cdot\nu_u)^2\right)}\mbox{ on } J_u\cap\Omega,
 \end{multline}
 whenever $J_u$, $u_+$ and $u_-$ are sufficiently smooth. Here $\kappa$ denotes the curvature of $J_u$ and $\left(\dive u_++\dive u_-\right)^\prime$ refers to the tangential derivative along the jump set. 
\ethm
\begin{cor}\label{conservation} Suppose $u$ is smooth and critical for $E_0$ in the sense of \eqref{el:5}. Then writing $u$ locally in terms of a lifting as $u(x,y)=e^{i\theta(x,y)}$ and defining the scalar $v:=\dive u$ one has that \eqref{el:5} is equivalent to the following system for the two scalars $\theta$ and $v$:
\begin{eqnarray}
&&-\sin\theta\, \theta_x+\cos\theta \,\theta_y=v\label{thetav}\\
&&-\sin\theta\, v_x+\cos\theta \,v_y=0.\label{vconstant}
\end{eqnarray}
Consequently, starting from any initial curve in $\Omega$ parametrized via $s\mapsto \big(x_0(s),y_0(s)\big)$ along which $\theta$ and $v$ take values
$\theta_0(s)$ and $v_0(s)$ respectively, the characteristic curves, say
$t\mapsto \big(x(s,t),\,y(s,t)\big)$, are given by
\begin{eqnarray}
&&x(s,t)=\frac{1}{v_0(s)}\left[\cos\big(v_0(s)t+\theta_0(s)\big)-\cos\theta_0(s)\right]+x_0(s),\label{xst}\\
&&y(s,t)=\frac{1}{v_0(s)}\left[\sin\big(v_0(s)t+\theta_0(s)\big)-\sin\theta_0(s)\right]+y_0(s),\label{yst}
\end{eqnarray}
whenever $v_0(s)\not=0.$ The corresponding solutions $\theta(s,t)$ and $v(s,t)$ are given by
\begin{equation}
\theta(s,t)=v_0(s)t+\theta_0(s),\quad v(s,t)=v_0(s),\label{tvst}
\end{equation}
so that the characteristics are circular arcs of curvature $v_0(s)$ and carry constant values of the divergence. In case the divergence vanishes somewhere along the initial curve, i.e. $v_0(s)=0$, then the characteristic is a straight line.
\end{cor}
\begin{proof}[Proof of Theorem \ref{critthm}]
We consider $u\in BV(\Omega,  \mathbb{S}^1) \cap H_{\rm{div}} (\Omega,  \mathbb{S}^1)$ such that $u_{\partial\Omega}\cdot\nu_{\partial \Omega}=g\cdot \nu_{\partial \Omega}$ on $\partial\Omega$.
 Then
 \begin{multline}
 \label{el:1}
 E_0(u+\delta u)-E_0(u) = \frac{L}{2} \int_{\Omega} \left[\left(\dive u+\dive \delta u\right)^2-(\dive u)^2\right] \,dx \\ + \frac{1}{6}\int_{J_{u+\delta u}} |\left(u_-+\delta u_-\right)-\left(u_++\delta u_+\right)|^3 \,d \sh^1- \frac{1}{6}\int_{J_u} |u_--u_+|^3 \,d \sh^1 \\ + \frac{1}{6} \int_{\partial \Omega\cap J_{u+\delta u}} |u_{\partial\Omega}+\delta u_{\partial\Omega} -g|^3 \,d\sh^1\\ -\frac{1}{6} \int_{\partial \Omega\cap J_{u}} |u_{\partial\Omega}-g|^3 \,d\sh^1+\frac{1}{2}\int_{\Omega}\lambda\left({\left|u+\delta u\right|}^2-{\left|u\right|}^2\right)\,dx,
 \end{multline}
for any $\delta u$ in $BV(\Omega, \mathbb R^2)\cap H_{\rm{div}} (\Omega, \mathbb R^2)$. The Lagrange multiplier $\lambda$ in \eqref{el:1} enforces the constraint $u\in \mathbb{S}^1$. 

We suppose first that the perturbation $\delta u $ is either supported away from $J_u$ or else is supported in a ball containing only a smooth portion of $J_u\cap\Omega$ and leaves the jump set unaltered, i.e. $J_{u+\delta u}=J_u$. We recall that the normal component of any vector field $w\in H_{\rm{div}}(\Omega,  \mathbb{S}^1)$ is continuous across the jump set of $w$ and $|w_--w_+|=2\sqrt{1-(w\cdot\nu_u)^2}$. We have from \eqref{el:1} that
 \begin{multline}
 \label{el:2}
 \delta E_0(u) = L \int_{\Omega} \dive u\,\dive \delta u \,dx-4\,\int_{J_{u}\cap\Omega} {\left(1-(u\cdot\nu_u)^2\right)}^{1/2}(u\cdot\nu_u)(\delta u\cdot\nu_u) \,d \sh^1 \\ +\int_{\Omega}\lambda(u\cdot\delta u) = \int_{\Omega} \left[-L\nabla\dive u+\lambda\,u\right]\cdot\delta u \,dx \\ -\int_{J_{u}\cap\Omega} \left[L\left(\dive u_+-\dive u_-\right)+4{\left(1-(u\cdot\nu_u)^2\right)}^{1/2}(u\cdot\nu_u)\right](\delta u\cdot\nu_u) \,d \sh^1.
 \end{multline}
From the consideration of perturbations $\delta u$ supported away from $J_u$ we conclude that $u$ satisfies the equation
\begin{equation}
\label{el:3}
-L\nabla\dive u+\lambda\,u=0\mbox{ in }\Omega\backslash J_u,
\end{equation}
which is equivalent to \eqref{el:5}. Then allowing for perturbations that meet $J_u\cap\Omega$ but that leave the jump set unaltered we see that $u$ is subject to the natural boundary conditions \eqref{el:4}. If instead, the perturbation is supported in a ball that contains a portion of $J_u\cap\partial\Omega$ then a calculation analogous to \eqref{el:2} leads to the condition \eqref{el:45}.

Before deriving the last condition \eqref{el:13} of the theorem , we wish to re-interpret the criticality condition \eqref{el:5} as a system of conservation laws. To this end, we suppose an $\mathbb{S}^1$-valued vector field $u$ is critical in the sense of \eqref{el:5} and that we locally write $u$ in terms of a lifting as $u(x,y)=e^{i\theta(x,y)}.$ Assuming $u$ is sufficiently smooth, we introduce the scalar $v:=\dive u$ and find that \eqref{el:5} is equivalent to the following system for the two scalars $\theta$ and $v$:
\begin{eqnarray}
&&-\sin\theta\, \theta_x+\cos\theta \,\theta_y=v\label{thetav1}\\
&&-\sin\theta\, v_x+\cos\theta \,v_y=0.\label{vconstant1}
\end{eqnarray}
Starting from any initial curve in $\Omega$ parametrized via $s\mapsto \big(x_0(s),y_0(s)\big)$ along which $\theta$ and $v$ take values
$\theta_0(s)$ and $v_0(s)$ respectively, one readily solves \eqref{thetav1}-\eqref{vconstant1} to obtain \eqref{xst}, \eqref{yst} and \eqref{tvst}.   We will exploit this property of constant divergence along these circular characteristics in a construction below.

Now we consider a competitor $u$  that is critical in the sense of \eqref{el:5}-\eqref{el:4} and is such that within some ball $B\subset \Omega$ centered on a point of smoothness of $J_u\cap\Omega$ one has the conditions: (i) $\dive u$ is continuous on both sides of $J_u\cap B$ and (ii) the traces of $\dive u$ on $J_u$ are differentiable along $J_u\cap B$ {with integrable derivatives}. We let $J_w$ be a small perturbation of $J_u\cap B$, where a part of a smooth curve in $J_u$ is replaced by another smooth curve (Fig.~\ref{fig1}). 
\begin{figure}[htb]
\centering
    \includegraphics[width=2in]
                    {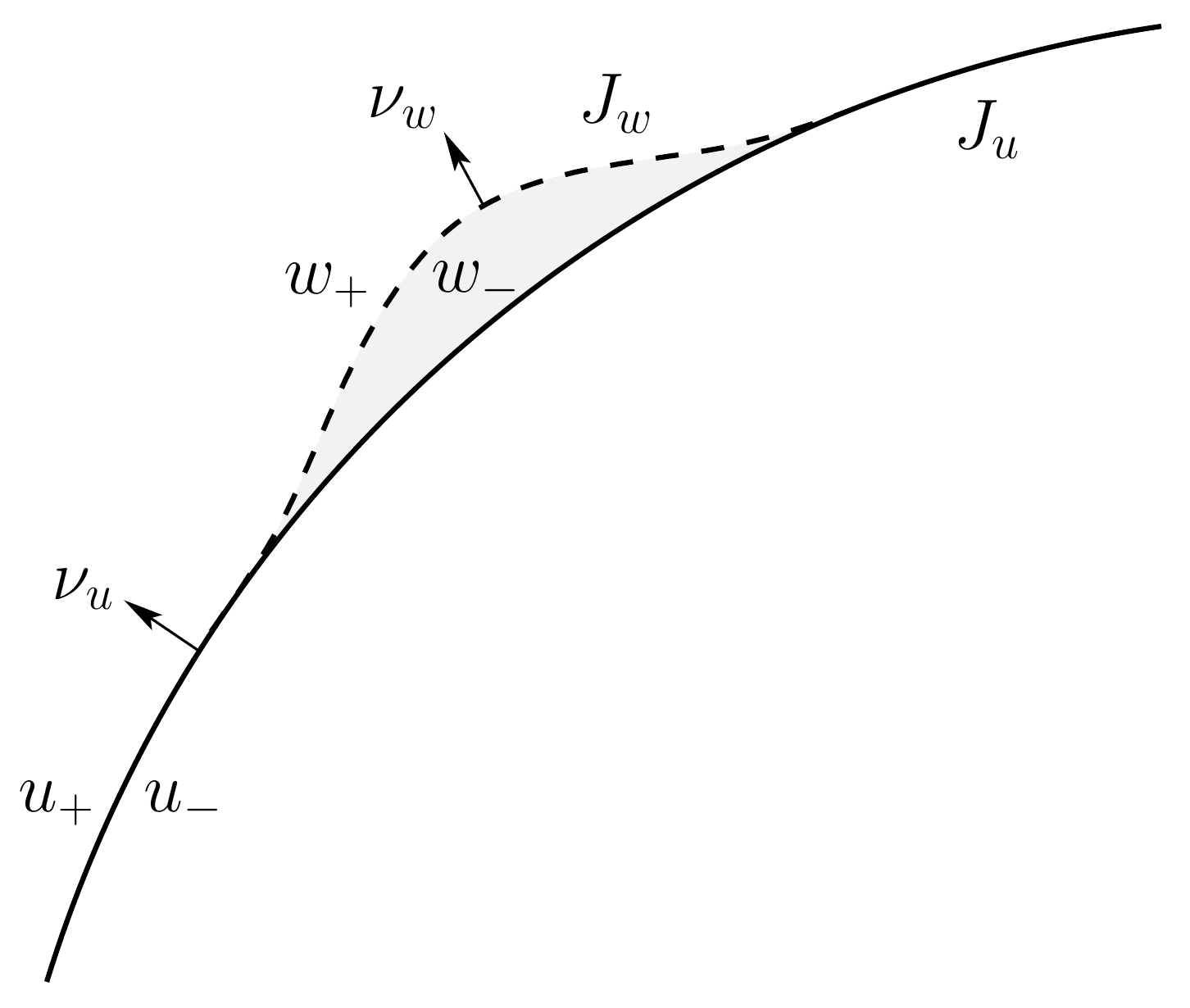}
    \caption{Perturbation of the jump set.}
  \label{fig1}
\end{figure}
\begin{figure}[htb]
\centering
 \includegraphics[width=4in]
                    {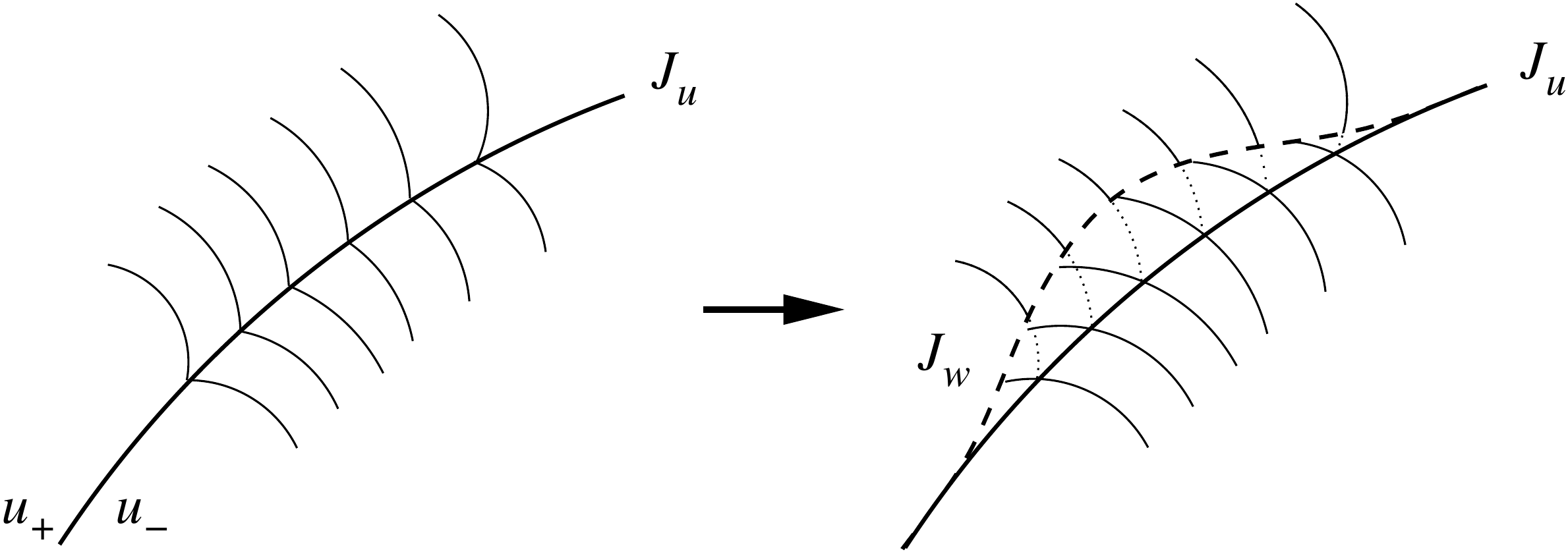}
  \caption{Construction of the perturbed minimizer $w$. The circular arcs of the characteristics defining $u$ meet at the jump set $J_u$ (left). The characteristics on the right side of $J_u$ are used to extend $u_-$ into the interior of the region $J_w\triangle J_u$ (right). In order to satisfy the continuity condition on $J_w$ for the normal component of $w$, a small perturbation is added to $u$ on the left side of $J_w$ (not shown). The smallness of this perturbation is guaranteed by the continuity of $u$ and its divergence on both sides of $J_u$.}
  \label{fig2}
\end{figure}
We assume that the new curve maintains the connectivity of $J_u$, connects smoothly to $J_u$, and lies on one side of the original curve. Here, to fix ideas, we assume that $J_w$ lies on the left side of $J_u$ corresponding to $u_+$. We construct the perturbation $w$ of $u$ as follows. Supposing that on the right side of $J_u$ the function $w$ coincides with $u_-$, we use the characteristics on the right side of $J_u$, using $u_-$ as initial values, to extend $u_-$ into the interior of the region $J_w\triangle J_u$ thus defining $w$ in that region (Fig.~\ref{fig2}). 
{ The characteristics extension of $u_-$ into $J_w\triangle J_u$ allows us to maintain control over $\mathrm{div}\,w_--\mathrm{div}\,u_-$ in that region.

We let $\Omega_w^+$ denote the region to the left of $J_w$ in Fig.~\ref{fig1} and denote by $w_-$ the trace of $w$ on $J_w$ as the boundary is approached from within the region $\mathrm{int}\,(J_w\triangle J_u)$. In order to make sure that the function $w$ is in $H_{\rm{div}}(\Omega,  \mathbb{S}^1)$, it must have the trace 
\begin{equation}
\label{contdiv}
w_+=\left(2\nu_w\otimes\nu_w-I\right)w_-
\end{equation}
on $J_w$ as $J_w$ is approached from within the region $\Omega_w^+$. Indeed, as long as \eqref{contdiv} holds, we have $w_+\cdot\nu_w=w_-\cdot\nu_w$ and $\left(\nu_w\otimes\nu_w-I\right)w_+=-\left(\nu_w\otimes\nu_w-I\right)w_-$. 

We take advantage of continuity of $u$ and $\dive u$ away from $J_u$ which ensures that the difference between $w_+$ as defined in \eqref{contdiv} and $u$ on $J_w$ is small. In particular, if $u=e^{i\theta_u}$ to the left of $J_u$ and $w_+=e^{i\theta_{w_+}}$ on $J_w$, then $\delta\theta_+=\theta_u-\theta_{w_+}$ is small on $J_w$. We introduce a small perturbation $\delta\theta$ compactly supported in $\Omega_w^+$ and such that the trace of $\delta\theta$ on $J_w$ is $\delta\theta_+$. Then we set $w=e^{i(\theta_u+\delta\theta)}$ in $\Omega_w^+$ so that $w\in BV(\Omega,  \mathbb{S}^1)\cap H_{\rm{div}} (\Omega,  \mathbb{S}^1)$. Further, if we let $\delta u:=w-u$ then $\delta u\in BV(\Omega_{w_+}, \mathbb R^2)\cap H_{\rm{div}} (\Omega_{w_+}, \mathbb R^2)$ is a small, complex-valued perturbation compactly supported in the closure of $\Omega_w^+$.} 

Next, we suppose that $J_u$ has the arc-length parametrization $r_u(s)$, where $s\in I$. We introduce the function $h:I\to \R$ with small $C^1-$norm such that $h$ vanishes along with its derivatives at the endpoints of $I$. We now assume that $r_w(s)=r_u(s)+h(s)\nu_u(s)$ for $s\in I$ defines $J_w$. We let $\tau_u(s)=r_u^\prime(s)$ so that $\nu_u(s)=\tau_u^{\perp}$. 

By our assumptions on divergence and using the characteristics construction of $u$ and $w$, it follows that  $\|\delta u\|_{1,\infty}=O(\|h\|_{1,\infty})$. To simplify the notation, we assume that all equivalences in the derivation of the criticality condition appearing below up to \eqref{el:9.1} are true up to terms of order $O\left(\|h\|_{1,\infty}^2\right)$. 

Integrating by parts and using \eqref{el:3}, we have 
\begin{align*}
\frac{L}{2} \int_{\Omega_w^+} \left\{(\dive w)^2-(\dive u)^2\right\} \,dx &=L \int_{\Omega_w^+} \dive u\,\dive \delta u \,dx\\ &=-L \int_{J_w\backslash J_u} (\dive u)\,\delta u\cdot\nu_w \,d \sh^1,
\end{align*}
where $\nu_w$ is the unit normal to $J_w$ pointing into $\Omega_w^+$ (See Fig.~\ref{fig1}). The variation of the energy is then given by
\begin{multline}
 \label{el:6}
 E_0(w)-E_0(u) = \frac{L}{2} \int_{\mathrm{int}\left(J_w\triangle J_u\right)} \left\{(\dive w)^2-(\dive u)^2\right\} \,dx \\ -L \int_{J_w\backslash J_u} (\dive u)\,\delta u\cdot\nu_w \,d \sh^1+\frac{4}{3}\,\int_{J_w\backslash J_{u}} {\left(1-(w\cdot\nu_w)^2\right)}^{3/2}\,d \sh^1\\-\frac{4}{3}\,\int_{J_u\backslash J_{w}} {\left(1-(u\cdot\nu_u)^2\right)}^{3/2}\,d \sh^1.
 \end{multline}
We estimate the third term in \eqref{el:6} as follows. Because
\[
\left|r_w^\prime\right|=\left|(1-h\kappa)\tau_u+h^\prime\nu_u\right|=\sqrt{{(1-h\kappa)}^2+{\left(h^\prime\right)}^2}=1-h\kappa
 \]
 and
 \[
\nu_w=\frac{(1-h\kappa)\,\nu_u-h^\prime \tau_u}{\left|(1-h\kappa)\tau_u+h^\prime\nu_u\right|}=\nu_u-h^\prime \tau_u,
 \]
 we have
 \begin{multline*}
1-(w\cdot\nu_w)^2=1-\left(w(r_u)+h\,\nabla w(r_u)\,\nu_u)\cdot\left(\nu_u-h^\prime \tau_u\right)\right)^2\\=1-\left(w(r_u)\cdot\nu_u+h\,\nabla w(r_u)\,\nu_u\cdot\nu_u-h^\prime \tau_u\cdot w(r_u)\right)^2\\=1-\left(w(r_u)\cdot\nu_u\right)^2-2\left(w(r_u)\cdot\nu_u\right)\left(h\,\nabla w(r_u)\,\nu_u\cdot\nu_u-h^\prime \tau_u\cdot w(r_u)\right),
\end{multline*}
on $J_w\backslash J_u$ so that
 \begin{multline*}
\left.\frac{4}{3}\left(1-(w\cdot\nu_w)^2\right)^{3/2}\right|_{J_w\backslash J_u}\\=\frac{4}{3}\left(1-\left(w(r_u)\cdot\nu_u\right)^2-2\left(w(r_u)\cdot\nu_u\right)\left(h\,\nabla w(r_u)\,\nu_u\cdot\nu_u-h^\prime \tau_u\cdot w(r_u)\right)\right)^{3/2}\\=\frac{4}{3}\left(1-\left(w(r_u)\cdot\nu_u\right)^2\right)^{3/2}\\-4\left(1-\left(w(r_u)\cdot\nu_u\right)^2\right)^{1/2}\left(w(r_u)\cdot\nu_u\right)\left(h\,\nabla w(r_u)\,\nu_u\cdot\nu_u-h^\prime \tau_u\cdot w(r_u)\right)\\=\frac{4}{3}\left(1-\left(u\cdot\nu_u\right)^2\right)^{3/2}\\-4\left(1-\left(u\cdot\nu_u\right)^2\right)^{1/2}\left(u\cdot\nu_u\right)\left(h\,\nabla u_-\,\nu_u\cdot\nu_u-h^\prime u_-\cdot \tau_u\right)
\end{multline*}
for $s\in I$. With the help of \eqref{el:4}, we conclude that
\begin{multline}
 \label{el:7}
\frac{4}{3}\,\int_{J_w\backslash J_{u}} {\left(1-(w\cdot\nu_w)^2\right)}^{3/2}\,d \sh^1=\frac{4}{3}\,\int_{I} {\left(1-(w\cdot\nu_w)^2\right)}^{3/2}{\left(1-h\kappa\right)}\,ds\\=\frac{4}{3}\,\int_{I} {\left(1-(u\cdot\nu_u)^2\right)}^{3/2}{\left(1-h\kappa\right)}\,ds+L\int_I(\dive u_+-\dive u_-)\,\nabla u_-\,\nu_u\cdot\nu_u\,h\,ds\\+L\int_I\left\{(\dive u_+)\, \left(u_+\cdot \tau_u\right)+(\dive u_-) \left(\,u_-\cdot \tau_u\right)\right\}\,h^\prime\,ds,
 \end{multline}
because $w(r_u)\cdot \tau_u=-u_+\cdot \tau_u=u_-\cdot \tau_u$ on $J_u$. Similarly,
\begin{multline}
 \label{el:8}
-L \int_{J_w\backslash J_u} (\dive u)\,\delta u\cdot\nu_w \,d \sh^1\\=-L \int_{I} (\dive u(r_w))\,\left(w_-(r_w)\cdot\nu_w-u(r_w)\cdot\nu_w\right)\,ds\\=-L \int_{I} (\dive u_+)\,\left(w(r_u)+h\,\nabla w(r_u)\,\nu_u\right.\\\left.-u_+-h\,\nabla u_+\,\nu_u\right)\cdot\left(\nu_u-h^\prime \tau_u\right)\,ds\\=L \int_{I} (\dive u_+)\,\left(\nabla u_+\,\nu_u-\nabla u_-\,\nu_u\right)\cdot\nu_u\,h\,ds\\-L\int_I(\dive u_+)\, \left(u_+\cdot \tau_u-u_-\cdot \tau_u\right)\,h^\prime\,ds,
 \end{multline}
since $w(r_u)\cdot\nu_u=u\cdot\nu_u$ on $J_u$. Adding \eqref{el:7} and \eqref{el:8}, we find 
\begin{multline}
 \label{el:9}
-L \int_{J_w\backslash J_u} (\dive u)\,\delta u\cdot\nu_w \,d \sh^1+\frac{4}{3}\,\int_{J_w\backslash J_{u}} {\left(1-(w\cdot\nu_w)^2\right)}^{3/2}\,d \sh^1\\=\frac{4}{3}\,\int_{I} {\left(1-(u\cdot\nu_u)^2\right)}^{3/2}{\left(1-h\kappa\right)}\,ds\\+L\int_I\left\{(\dive u_+)\,\nabla u_+\,\nu_u\cdot\nu_u-(\dive u_-)\,\nabla u_-\,\nu_u\cdot\nu_u\right\}\,h\,ds\\-L\int_I\left\{(\dive u_+)\, \left(u_+\cdot \tau_u\right)-(\dive u_-) \left(\,u_-\cdot \tau_u\right)\right\}\,h^\prime\,ds.
 \end{multline}
 Finally, changing the coordinates $(x,y)=r_u(s)+t\,h\,\nu_u(s)$ and using our continuity assumptions, we have for the first integral in \eqref{el:6} that
 \begin{multline}
 \label{el:9.1}
 \frac{L}{2} \int_{\mathrm{int}\left(J_w\triangle J_u\right)} \left\{(\dive w)^2-(\dive u)^2\right\} \,dx = \frac{L}{2} \int_{I}\int_{0}^{h}\left\{(\dive w)^2-(\dive u)^2\right\} (1-h\,\kappa)\,dt\,ds \\=\frac{L}{2} \int_{I}\int_{0}^{h}\left\{(\dive u_-)^2-(\dive u_+)^2\right\}\,dt\,ds=\frac{L}{2} \int_{I}\left\{(\dive u_-)^2-(\dive u_+)^2\right\}\,h\,ds
 \end{multline}
The equations \eqref{el:6}, along with \eqref{el:9} and \eqref{el:9.1}, give the following variation of the energy functional
\begin{multline}
 \label{el:10}
 \delta E_0(u) = \frac{L}{2} \int_{I} \left\{(\dive u_-)^2-(\dive u_+)^2\right\}\,h \,ds-\frac{4}{3}\,\int_{I} {\left(1-(u\cdot\nu_u)^2\right)}^{3/2}\,h\kappa\,ds\\+L\int_I\left\{(\dive u_+)\,\nabla u_+\,\nu_u\cdot\nu_u-(\dive u_-)\,\nabla u_-\,\nu_u\cdot\nu_u\right\}\,h\,ds\\-L\int_I\left\{(\dive u_+)\, \left(u_+\cdot \tau_u\right)-(\dive u_-) \left(\,u_-\cdot \tau_u\right)\right\}\,h^\prime\,ds.
 \end{multline}
Now, observe that the identities \[\nabla u\,\nu_u\cdot\nu_u=\dive u-\nabla u\,\tau_u\cdot\tau_u\] and \[(u\cdot\tau_u)^\prime=\nabla u\,\tau_u\cdot\tau_u+\kappa\,u\cdot\nu_u\] hold separately for $u_-$ and $u_+$ on $J_u$. Substituting these expressions into \eqref{el:10} and integrating by parts, we have
\begin{multline}
 \label{el:11}
 \delta E_0(u) = \frac{L}{2} \int_{I} \left\{(\dive u_+)^2-(\dive u_-)^2\right\}\,h \,ds-\frac{4}{3}\,\int_{I} {\left(1-(u\cdot\nu_u)^2\right)}^{3/2}\,h\,\kappa\,ds\\-L\int_I\left\{(\dive u_+)\nabla u_+\,\tau_u\cdot\tau_u\,-(\dive u_-)\,\nabla u_-\,\tau_u\cdot\tau_u\right\}\,h\,ds\\-L\int_I\left\{(\dive u_+)\, \left(u_+\cdot \tau_u\right)-(\dive u_-) \left(\,u_-\cdot \tau_u\right)\right\}\,h^\prime\,ds\\=\frac{L}{2} \int_{I} \left[(\dive u_+)^2-(\dive u_-)^2\right]\,h \,ds-\frac{4}{3}\,\int_{I} {\left(1-(u\cdot\nu_u)^2\right)}^{3/2}\,h\,\kappa\,ds\\-L\int_I\left\{(\dive u_+)\left((u_+\cdot\tau_u)^\prime-\kappa\,u\cdot\nu_u\right)\,-(\dive u_-)\,\left((u_-\cdot\tau_u)^\prime-\kappa\,u\cdot\nu_u\right)\right\}\,h\,ds\\+L\int_I\left\{(\dive u_+)^\prime\, \left(u_+\cdot \tau_u\right)-(\dive u_-)^\prime \left(\,u_-\cdot \tau_u\right)\right\}\,h\,ds\\+L\int_I\left\{(\dive u_+)\, \left(u_+\cdot \tau_u\right)^\prime-(\dive u_-) \left(\,u_-\cdot \tau_u\right)^\prime\right\}\,h\,ds,
 \end{multline}
for any smooth, positive $h$ with a compact support in $I$. The same expression can be established for smooth, negative $h$ with a compact support in $I$ by considering perturbations of the jump set that lie on the right side of $J_u$. From this we immediately conclude that $J_u$ is stationary whenever
\begin{multline}
 \label{el:12}
(\dive u_+)^2-(\dive u_-)^2+\left(\dive u_++\dive u_-\right)^\prime\left(u_+\cdot \tau_u-u_-\cdot \tau_u\right)\\=\frac{8\kappa}{3L}{\left(1-(u\cdot\nu_u)^2\right)}^{3/2}-2\kappa \left(\dive u_+-\dive u_-\right)\left(u\cdot \nu_u\right)\mbox{ on }J_u.
 \end{multline}
With the help of \eqref{el:4}, the condition \eqref{el:12} can also be expressed as  in \eqref{el:13}.
\end{proof}

 \section{Results for the special case of a disc or an annulus}

{Now we present some examples where we take $\Omega$ to be a disc or an annulus. For the disc we will discuss three choices of boundary data $g:\partial\bD\to \mathbb{S}^1$. Our focus is on optimizing the $\Gamma$-limit $E_0$ where
we recall the normal component of competitors $u\in H_{\rm div}(\bD;\mathbb{S}^1) \cap BV(\bD;\mathbb{S}^1)$ is required to satisfy $u_{\partial \bD}\cdot \nu_{\partial \bD} =g_{\partial \bD}\cdot \nu$ on $\partial\bD.$ Our discussion on the annulus is a bit more formal, and we present examples that indicate situations where the wall is potentially curved, possibly occurring along the boundary. }

Throughout this section, $\widehat{e}_r := (x,y)/\sqrt{x^2 + y^2}$ denotes the unit radial vector field, while \[\widehat{e}_\theta := (-y,x)/\sqrt{x^2 + y^2}\] denotes the unit angular vector field. 

\subsection{Tangential boundary conditions: $g(x,y) = (-y,x)$\,} In this case, a minimizer is clearly given by the divergence-free vector field
\[
u(x,y)=\widehat{e}_\theta.
\]
since for this choice of $u$ one has $E_0(u)=0.$ 

From the characteristics viewpoint laid out in the preceding section, this critical point is composed of characteristics which are simply radii through the origin of $\bD$ to the boundary, corresponding to $v \equiv 0$ on each of these characteristics. We point out that for the Aviles-Giga energy, the authors in \cite{JOP} classify zero energy states of the Aviles-Giga energy. More recently, \cite{NewLorent} provides a quantitative version of the result in \cite{JOP}. 
\subsection{Hedgehog boundary conditions: $g(x,y) = (x,y).$ } 
Here we can again precisely determine the minimizers of $E_0$:
\bthm \label{radialhedge} For $\Omega=\bD$ and boundary data $g=(x,y)$ the two  functions 
${ u}_*^\pm :=r {\widehat{e}_r} \pm \sqrt{1-r^2}\widehat{e}_\theta$ are the only  minimizers of the problem
\[
\inf E_0({ u})
\]
taken over competitors $u\in H_{\rm div}(\bD;\mathbb{S}^1) \cap BV(\bD;\mathbb{S}^1)$ satisfying ${ u}_{\partial \bD}\cdot \nu_{\partial \bD} =g\cdot \nu_{\partial \bD}=1$ on $\partial\bD.$
\ethm
\begin{proof}
We note first that since $u\cdot\nu\equiv 1$ on $\partial\bD$ and $\abs{u}=1$, necessarily competitors must have traces satisfying
${ u}={ g}(x,y)=(x,y)$ along $\partial\bD.$

Now given any competitor ${ u}$, an application of the Cauchy-Schwarz and the Divergence Theorem gives
\beq
E_0({ u})\geq \frac{L}{2}\int_{\bD} (\dive { u})^2 \,dx\geq \frac{L}{2}\frac{1}{\pi}\left(\int_{\bD}\dive { u}\,dx \right)^2=2\pi L=E_0({ u}_*^{\pm}).\label{ustar}
\eeq
Hence ${ u}_*^{\pm}$ are minimizers and any other minimizing competitor would have to yield equality in both of the inequalities above.
Consequently, the only possible candidates for minimizers ${ u}$ must satisfy $J_{ u}=\emptyset$ so that ${ u}\in W^{1,1}(\bD)$ and $\dive { u}\equiv constant.$
The Divergence Theorem and the boundary conditions then imply that in fact $\dive { u}\equiv 2$ throughout $\bD$.

Now we expand the competitor ${ u}$ in a Fourier series as 
\begin{align*}
{ u} = \sum_{n \in \bZ} { u}_n(r) e^{i n \theta},
\end{align*}
where ${ u}_n(r) = f_n(r) + i g_n(r),$ are a sequence of complex valued functions that satisfy ${ u}_0(1) = 1$ and $ { u}_n(1) = 0$ if $n \neq 1.$
  In order to compute the divergence of ${ u}$ written in the Fourier development, we write ${ V}_n(r,\theta) := u_n(r) e^{i n \theta},$ and note that written as a vector field in $\R^2,$ we have 
\begin{align*}
{ V}_n (r,\theta) = \left(\begin{array}{l}
f_n(r) \cos n \theta - g_n(r) \sin n \theta \\
g_n(r) \cos n \theta + f_n(r) \sin n \theta
\end{array} \right).
\end{align*} 
A calculation then yields that 
\[
\dive { V}_n = \left( f_n^\prime(r) + \frac{n f_n(r)}{r} \right) \cos (n - 1)\theta - \left( g_n^\prime(r) + \frac{n g_n(r)}{r} \right) \sin ( n-1) \theta. 
\]
Using Plancherel and arguing as in \eqref{ustar} we find
\[
E_0({ u}) =  \frac{L}{2}\sum_n\int_\bD (\dive { V}_n)^2 \,dx\geq  \frac{L}{2}\int_\bD (\dive { V}_0)^2 \,dx\geq 2\pi L, 
\]
and so ${ u}={ V}_0={ u_0}$ with necessarily $\dive { V}_0=f_0^\prime + \frac{f_0}{r}\equiv 2$. Solving this ODE with the boundary condition $f_0(1)=1$ we find
$f_0(r)=r$ and since $\abs{{ u}}=1$, it follows that $g_0(r)=\pm\sqrt{1-r^2}$ so that ${ u}={ u}_*^+$ or ${ u}_-^+$.
\end{proof}

\subsection{Degree $-1$ boundary conditions: $g(x,y) = (x/R,-y/R).$ } \label{deg-1section}

In this section, we develop a solution of the Euler-Lagrange boundary value problem \eqref{el:5}-\eqref{el:4} with the symmetries hinted by a numerical solution of the relaxed problem. Although we do not claim that our construction yields a minimizer of the limiting functional, the minimizing property of our solution seems plausible given its close resemblance to the numerics, at least for a certain range of parameters of the problem.

We used the COMSOL Multiphysics\textsuperscript{\textregistered} finite elements software \cite{COMSOL} to solve the Euler-Lagrange equation associated with the energy functional \eqref{energy} in the circle of the radius $R=0.6$, subject to the boundary conditions $g(x,y) = (x/R,-y/R)$. The (local) minimizers in COMSOL were found by simulating the gradient flow for $E_\varepsilon$ on time intervals that were sufficiently large for a solution to reach an equilibrium. The results for $L=0.5$ and $\varepsilon=0.005$ are shown in Figs.~\ref{f:nd-1}-\ref{f:nd-3}.
\begin{figure}[htb]
\centering
    \includegraphics[width=3in]
                    {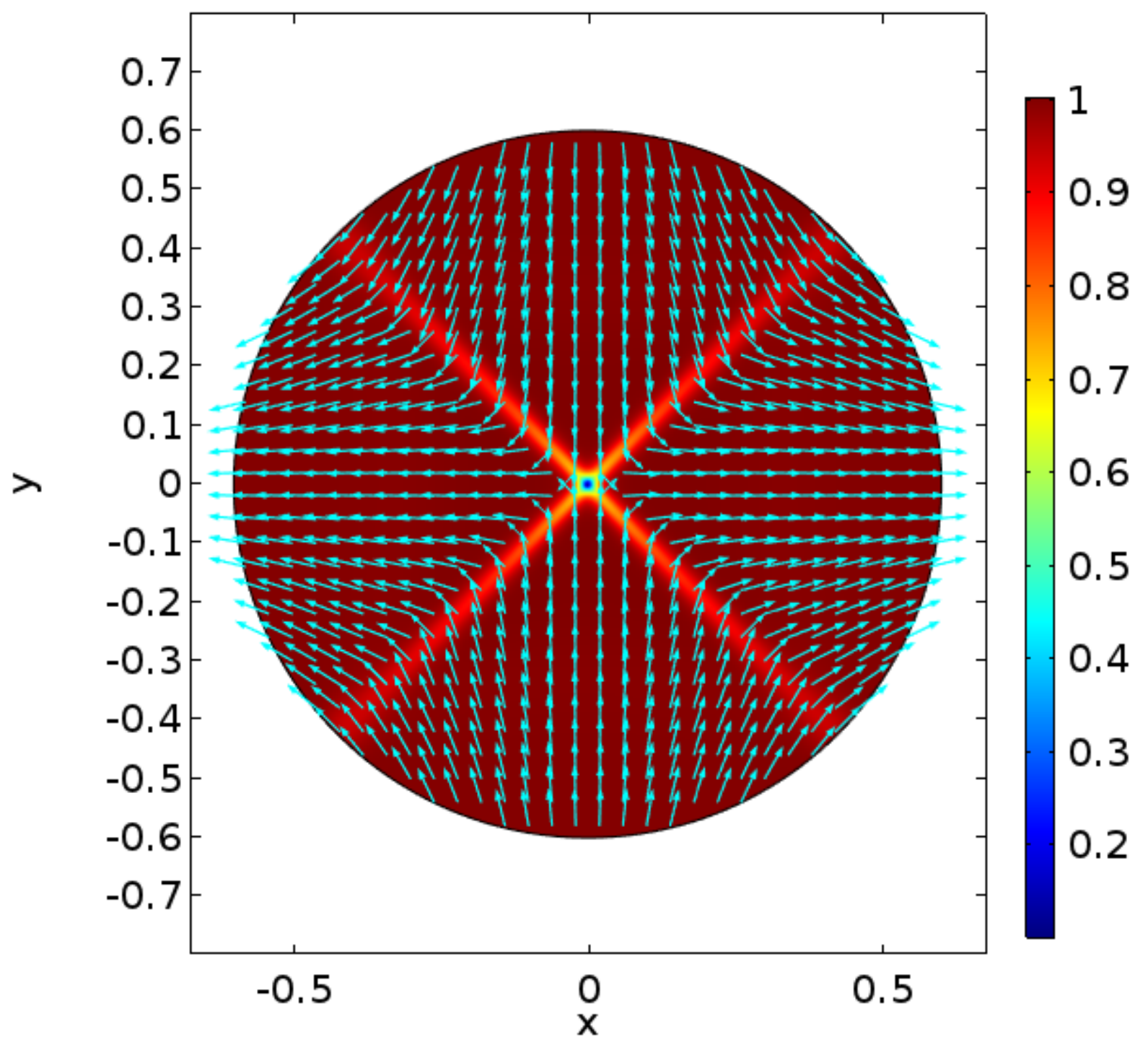}
    \caption{A solution $u$ of the Euler-Lagrange equation associated with the energy functional \eqref{energy} in the circle of the radius $R=0.6$, subject to the boundary conditions $g(x,y) = (x/R,-y/R)$. Both $u$ and $|u|$ are shown.}
  \label{f:nd-1}
\end{figure}
\begin{figure}[htb]
\centering
    \includegraphics[width=3in]
                    {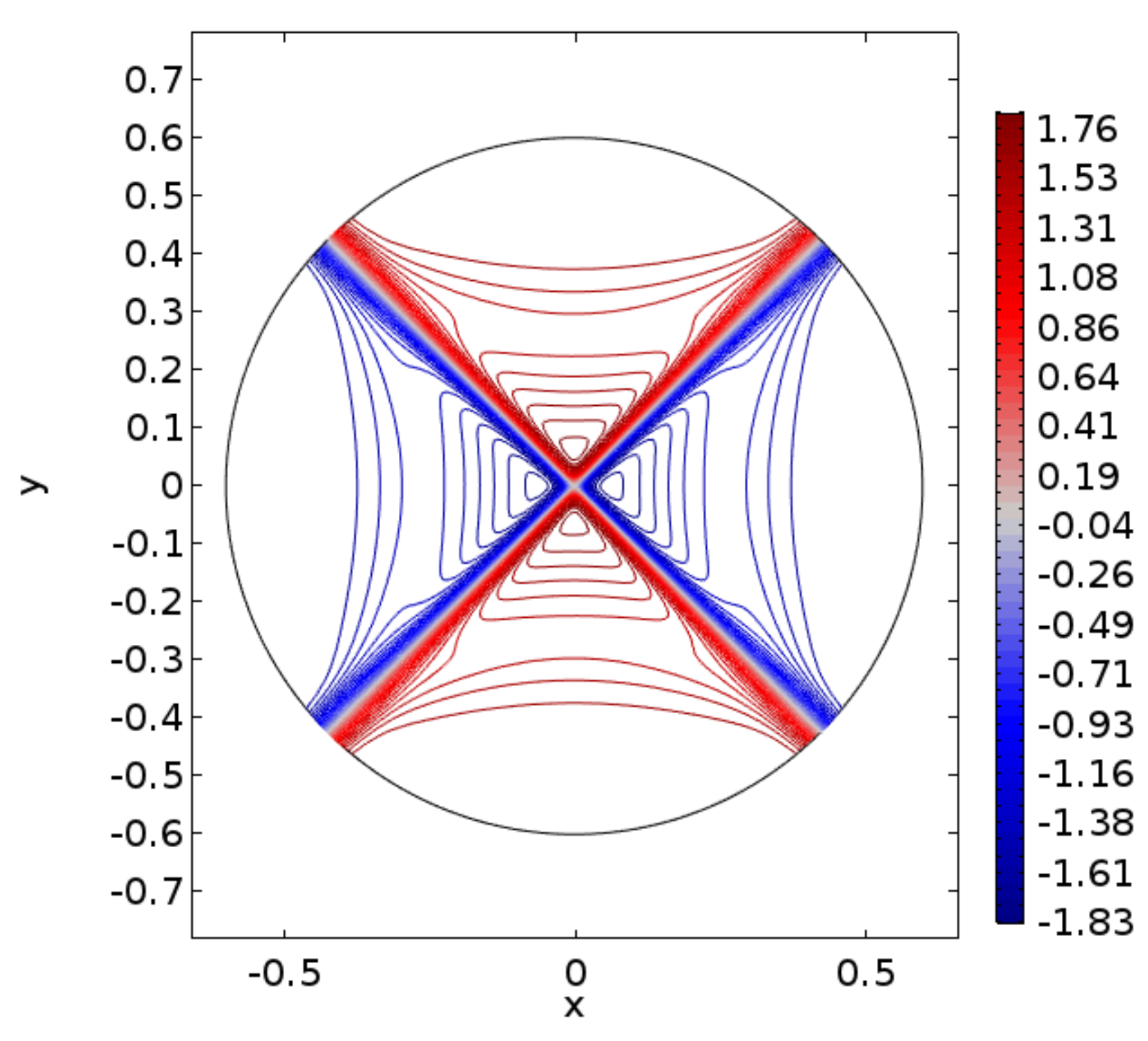}
    \caption{Level curves for the divergence of $u$, where $u$ is depicted in Fig.~\ref{f:nd-1}.}
  \label{f:nd-2}
\end{figure}
\begin{figure}[htb]
\centering
    \includegraphics[width=3in]
                    {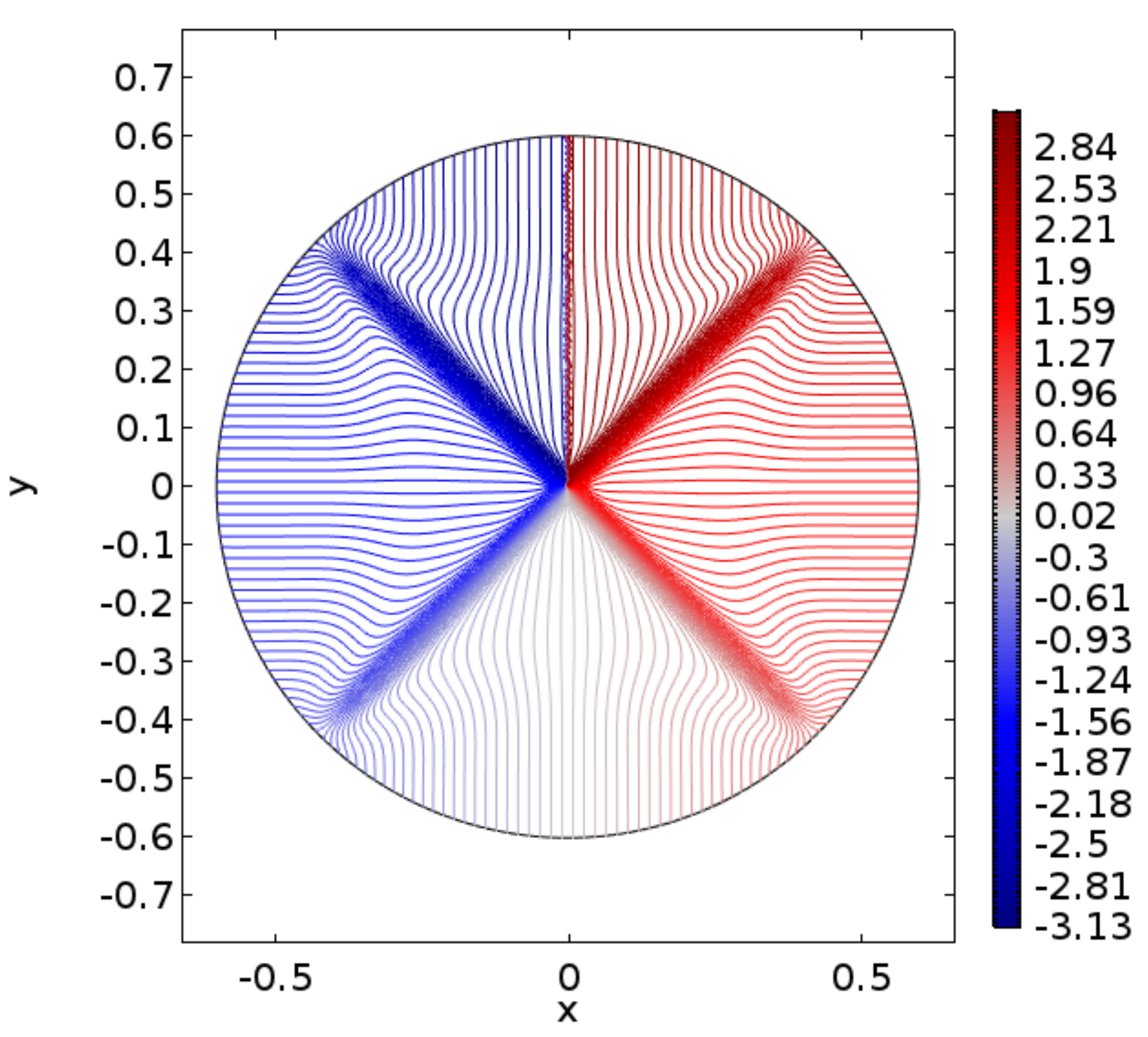}
    \caption{Level curves for the angle $\theta$, where $u=(\cos{\theta}, \sin{\theta})$ is depicted in Fig.~\ref{f:nd-1}.}
  \label{f:nd-3}
\end{figure}

First, we observe that (i) the jump set of the solution in Fig.~\ref{f:nd-1} consists of two straight lines inclined at $45^\circ$ to the horizontal axis and (ii) the solution is symmetric with respect to reflections about both these lines, as well as the vertical and horizontal axes. Along the lines of the jump set, the symmetry is such that the normal components from either side are equal, while the tangential components are equal in absolute value and opposite in sign. Further, (iii) on both axes, the solution vector is parallel to the axis itself and (iv) Fig.~\ref{f:nd-2} indicates that the sum of the traces of the divergence of $u$ on both sides of the jump set equals zero. The last observation is consistent with the required criticality condition \eqref{el:13} since the curvature of the jump set is zero. Thus, it would be sufficient to look for the solution of \eqref{el:5}-\eqref{el:4} in one eighth of a circle of radius $R$, and then extend the construction to the rest of the circle via symmetry. 
\begin{figure}[htb]
\centering
    \includegraphics[width=3in]
                    {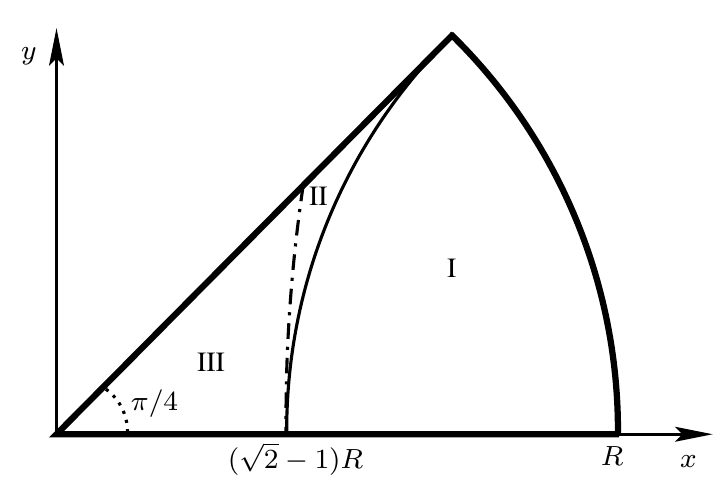}
    \caption{Regions corresponding to different characteristics families.}
  \label{fig:dcont1}
\end{figure}

Let $\Omega$ be a sector of the circle of radius $R$ as depicted in Fig.~\ref{fig:dcont1}. We seek a solution $u$ of \eqref{el:5}-\eqref{el:4} in the form \eqref{xst}-\eqref{tvst}, where $$J_u=\left\{(x,y)\in\mathbb R^2:y=x,x\in\left(0,R/\sqrt{2}\right)\right\},$$ subject to the Dirichlet boundary conditions 
\begin{equation}
\label{eq:d-1.1}
u=(1,0)\mbox{ when }y=0\mbox{ and }u=(x/R,-y/R)\mbox{ when }x^2+y^2=R^2.
\end{equation}
By our symmetry assumptions, the jump of $\dive u$ on $J_u$ is equal to $-2\,\dive u_-$, hence \eqref{el:4} takes the form
\begin{equation}
\label{el:4n}
L\,\dive u-2\,{\left(1-(u\cdot\nu_u)^2\right)}^{1/2}(u\cdot\nu_u)=0\mbox{ on }J_u,
\end{equation}
where we dropped the subscript $"-"$ for notational convenience. Our last assumption is based on the behavior of the numerical solution in Fig.~\ref{f:nd-1}. Considering the solution in the part of the disc corresponding to $\Omega$ in Fig.~\ref{f:nd-1} and recalling that $u=(\cos{\theta},\sin{\theta})$, in what follows we work with $\theta$ instead of $u$ and assume that 
\begin{equation}
\label{eq:trange}
\theta:\bar\Omega\to\left[-\pi/4,0\right].
\end{equation}

We begin by identifying two distinct families of characteristics that originate on the $x$-axis and recover the solution of the limiting problem in the regions $I$ and $III$ in Fig.~\ref{fig:dcont1}. 

\medskip\noindent {\bf 1.} First, taking into account \eqref{eq:d-1.1}, we construct a characteristic $$\left(x(s,t),y(s,t),\theta(s,t),v(s,t)\right)$$ with the initial data 
\begin{equation*}
\label{eq:d-1.2}
\left(x(s,0),y(s,0),\theta(s,0),v(s,0)\right)=\left(s,0,0,v_0(s)\right)\mbox{ for }s\in[s_0,R],
\end{equation*}
which terminates at some point $$\left(x\left(s,t^*(s)\right),y\left(s,t^*(s)\right)\right)=R\left(\cos{\left(\psi\left(s,t^*(s)\right)\right)},\sin{\left(\psi\left(s,t^*(s)\right)\right)}\right)$$ on the circular component of $\partial\Omega$ so that 
\begin{multline}
\label{eq:d-1.3}
\left(x\left(s,t^*(s)\right),y\left(s,t^*(s)\right),\theta\left(s,t^*(s)\right),v\left(s,t^*(s)\right)\right)\\=\left(R\cos{\left(\psi\left(s,t^*(s)\right)\right)},R\sin{\left(\psi\left(s,t^*(s)\right)\right)},-\psi\left(s,t^*(s)\right),v\left(s,t^*(s)\right)\right),
\end{multline}
for all $s\in[s_0,R]$. Here $\psi$ represents the polar angle for a vector $(x,y)$ while the parameter $s_0>0$ and the functions $v_0$ and $t^*$ are all to be determined in the course of solving the problem. Note that, as a consequence of \eqref{thetav1}, the characteristics and the field $u$ are mutually perpendicular at all points in $\Omega$, hence a characteristic intersecting the $x$-axis must be perpendicular to this axis at all points of intersection.
\begin{figure}[htb]
\centering
 \includegraphics[width=1.5in]
                    {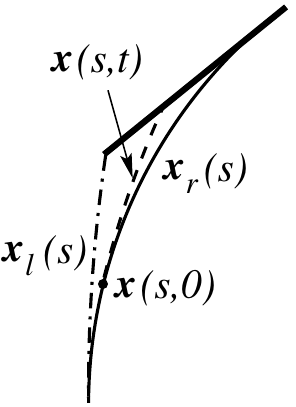}
  \caption{Characteristics construction in the intermediate region $II$.}
  \label{fig:dcont2}
\end{figure}

From \eqref{xst}-\eqref{tvst}, we conclude that
\begin{align}
&x(s,t)=\frac{1}{v_0(s)}\left[\cos\big(\theta(s,t)\big)-1\right]+t,\label{xstd}\\
&y(s,t)=\frac{1}{v_0(s)}\sin\big(\theta(s,t)\big),\label{ystd}\\
&\theta(s,t)=v_0(s)t,\label{tstd}\\
&v(s,t)=v_0(s),\label{vstd}
\end{align}
for all $s\in[s_0,R]$. Substituting $t^*(s)$ into these equations and using \eqref{eq:d-1.3} gives
\begin{align*} 
R\cos{\left(\psi\left(s,t^*(s)\right)\right)}=\frac{1}{v_0(s)}\left[\cos\big(\psi\left(s,t^*(s)\right)\big)-1\right]+t,
\end{align*}
\begin{align}
R\sin{\left(\psi\left(s,t^*(s)\right)\right)}=-\frac{1}{v_0(s)}\sin\big(\psi\left(s,t^*(s)\right)\big),\label{ystd1}
\end{align}
\begin{align*}
\psi\left(s,t^*(s)\right)=-v_0(s)t^*(s),
\end{align*}
for all $s\in[s_0,R]$. It follows from \eqref{ystd1} that
\begin{equation*}
\label{eq:d-1.4}
v_0\equiv-\frac{1}{R}
\end{equation*}
on $[s_0,R]$, that is all characteristic curves that intersect both the $x$-axis and the circular part of the boundary are themselves arcs of circles of radius $R$, centered on the $x$-axis. These curves clearly foliate a region in $\Omega$ labeled by $I$ in Fig.~\ref{fig:dcont1} and bounded from the left by the mirror image of the boundary arc with respect to the line $x=R/\sqrt{2}$. The corresponding leftmost characteristic curve in family $I$ will be denoted by ${\textbf{\textit x}}_r$. It intersects the $x$-axis at $x=(\sqrt{2}-1)R$ and is given by
\begin{align}
&x_r(t)=\sqrt{2}R-R\cos\big(t/R\big),\label{xstd2}\\
&y_r(t)=R\sin\big(t/R\big),\label{ystd2}\\
&\theta_r(t)=-t/R,\label{tstd2}
\end{align}
for all $t\in[0,\pi R/4]$.

\medskip\noindent {\bf 2.} Next, we turn our attention to the region labeled $III$ in Fig.~\ref{fig:dcont1}. This region is foliated by the characteristic curves intersecting both the $x$-axis and jump set $J_u=\left\{(x,y):y=x\right\}.$ Because they originate on the $x$-axis, these characteristics are given for $s\in[0,s_0]$ by the same equations as in \eqref{xstd}-\eqref{vstd}. For the remainder of this construction, we assume that $s\in[0,s_0]$. Suppose that intersection with the line $y=x$ occurs at some point $\left(x\left(s,t^*(s)\right),y\left(s,t^*(s)\right)\right).$ Then
\begin{align}
&x\left(s,t^*(s)\right)=y\left(s,t^*(s)\right),\label{cond1}\\
&Lv_0(t)+\cos^2{\theta\left(s,t^*(s)\right)}-\sin^2{\theta\left(s,t^*(s)\right)}=0.\label{cond2}
\end{align}
Here the second equation is the natural boundary condition \eqref{el:4n} recast into a simpler form using trigonometric identities. Equation \eqref{cond2} along with \eqref{eq:trange} imply that
\begin{equation}
\label{eq:d-1.9}
v_0(s)\leq0.
\end{equation}
From \eqref{xstd}, \eqref{ystd}, and \eqref{cond1}, we obtain
\begin{equation}
\label{eq:d-1.5}
\cos{\theta\left(s,t^*(s)\right)}-\sin{\theta\left(s,t^*(s)\right)}=1-sv_0(s).
\end{equation}
Then \eqref{cond2} and \eqref{eq:d-1.5} allow us to conclude that 
\begin{equation*}
\label{eq:d-1.6}
\cos{\theta\left(s,t^*(s)\right)}+\sin{\theta\left(s,t^*(s)\right)}=-\frac{Lv_0(s)}{1-sv_0(s)}
\end{equation*}
and
\begin{align*}
&2\cos{\theta\left(s,t^*(s)\right)}=1-tv_0(s)-\frac{Lv_0(s)}{1-sv_0(s)},\\
&2\sin{\theta\left(s,t^*(s)\right)}=-1+tv_0(s)-\frac{Lv_0(s)}{1-sv_0(s)}.
\end{align*}
Hence
\begin{equation*}
{\left(1-sv_0(s)\right)}^4-2{\left(1-sv_0(s)\right)}^2+L^2v_0^2(s)=0,
\end{equation*}
and
\begin{equation}
\label{eq:d-1.8}
{\left(1-sv_0(s)\right)}^2=1+\sqrt{1-L^2v_0^2(s)}.
\end{equation}
Here the sign in front of the square root follows from \eqref{eq:d-1.9}. Now let
\[F(p):={\left(1-tp\right)}^2-\sqrt{1-L^2p^2}-1.\]
Clearly, $F$ is continuous on $[-1/L,0]$ for every $s\in[0,s_0]$ and $$F(0)=-1<0\mbox{ and }F\left(-\frac{1}{L}\right)={\left(1+\frac{s}{L}\right)}^2-1\geq0.$$
Thus, there exists $-\frac{1}{L}\leq v_0(s)<0$ such that \eqref{eq:d-1.8} holds. Furthermore, by \eqref{eq:d-1.8}, we have the bound
$1-sv_0(s)<\sqrt{2},$
so that
$v_0(s)>-(\sqrt{2}-1)/s,$
and, in particular,
\begin{equation}
\label{eq:d-1.10}
v_0((\sqrt{2}-1)R)>-\frac{1}{R}.
\end{equation}
Note that the rightmost characteristic ${\textbf{\textit x}}_l$ in the family $III$ originates from the same point $((\sqrt{2}-1)R,0)$ on the $x$-axis as the characteristic ${\textbf{\textit x}}_r$ in the family $I$ and both ${\textbf{\textit x}}_l$ and ${\textbf{\textit x}}_r$ are tangent to each other at $((\sqrt{2}-1)R,0)$. The inequality \eqref{eq:d-1.10} demonstrates that the radius of ${\textbf{\textit x}}_r$ is smaller than the radius of ${\textbf{\textit x}}_l$ and so there is a wedge-shaped region in $\Omega$, labeled $II$ in Fig.~\ref{fig:dcont2}, which is covered neither by the characteristics from the family $I$ nor by the characteristics from the family $III$. In Step {\bf 3.} below, we construct the third family of characteristics that extends the solution to the region $II$.

We conclude this part of the construction by showing that the characteristics of the family $III$ indeed foliate the region $III$. We take the derivative of both sides of \eqref{eq:d-1.8} with respect to $s$ and solve for $v_0^\prime(s)$ to obtain
\[v_0^\prime(s)=-v_0(s){\left[s-\frac{L^2v_0(s)}{2(1-sv_0(s))\sqrt{1-L^2v_0^2(s)}}\right]}^{-1}>0.\]
It follows that the characteristic curves in the region $III$ are the circular arcs having curvature that increases with $s$. Since these curves also cross the $x$-axis at $90^\circ$, they completely cover the region $III$ without intersecting one another. We also note that $\lim_{s\to0}v_0(s)=-\frac{1}{L}$ and so the divergence of our solution in the region $III$ remains bounded.

\medskip\noindent {\bf 3.} Finally, we use characteristics to extend the solution to the region $II$. The procedure is illustrated in Fig.~\ref{fig:dcont2}. We use the curve \eqref{xstd2}-\eqref{tstd2} as the initial data for the new family of characteristics. For the remainder of this section, we will assume that $s\in(0,\pi R/4)$.  Let
\begin{align*}
&x_0(s)=\sqrt{2}R-R\cos\big(s/R\big),\\
&y_0(s)=R\sin\big(s/R\big),\\
&\theta_0(s)=-s/R.
\end{align*}
Then, from \eqref{xst}-\eqref{tvst}, we have that
\begin{align}
&x(s,t)=\frac{1}{v_0(s)}\left[\cos\big(\theta(s,t)\big)-\cos\big(s/R\big)\right]+\sqrt{2}R-R\cos\big(s/R\big),\label{xstd3}\\
&y(s,t)=\frac{1}{v_0(s)}\left[\sin\big(\theta(s,t)\big)+\sin\big(s/R\big)\right]+R\sin\big(s/R\big),\label{ystd3}\\ \nonumber
&\theta(s,t)=v_0(s)t-s/R,\\ \nonumber
&v(s,t)=v_0(s).
\end{align}
The new characteristic curves are still assumed to terminate on the jump set $y=x$, hence they must satisfy the conditions \eqref{cond1}-\eqref{cond2}. Setting $\theta^*(s)=\theta(s,t^*(s))$ and simplifying, these conditions take the form
\begin{align}
&\cos{\theta^*(s)}-\sin{\theta^*(s)}=A(s),\label{cond3}\\
&\cos{\theta^*(s)}+\sin{\theta^*(s)}=-\frac{Lv_0(s)}{A(s)},\label{cond4}
\end{align}
where
\begin{equation}
\label{eq:d-1.11}
A(s):=\sqrt{2}\left[(Rv_0(s)+1)\sin\big(s/R+\pi/4\big)-Rv_0(s)\right].
\end{equation}
The assumption \eqref{eq:trange} implies that
\begin{equation}
\label{eq:d-1.12}
v_0(s)\leq0\mbox{ and }A(s)>0.
\end{equation}
Following the same procedure as in Step {\bf 2.}, we find that $v_0(s)$ satisfies
\begin{equation}
\label{eq:d-1.13}
A^2(s)=1+\sqrt{1-L^2v_0^2(s)},
\end{equation}
hence
\begin{equation}
\label{eq:d-1.14}
v_0(s)\geq-\frac{1}{L}\mbox{ and }A(s)\leq\sqrt{2}.
\end{equation}
The second inequality in \eqref{eq:d-1.14} is equivalent to 
\[v_0(s)\geq-\frac{1}{R}\]
and, combining this inequality with the first inequality in \eqref{eq:d-1.12} and the first inequality in \eqref{eq:d-1.14}, we have
\begin{equation}
\label{eq:d-1.15}
-\min\left\{\frac{1}{R},\frac{1}{L}\right\}\leq v_0(s)\leq0.
\end{equation}
Now, let
\[F(p):=2{\left[(Rp+1)\sin\big(s/R+\pi/4\big)-Rp\right]}^2-\sqrt{1-L^2p^2}-1\]
and
\[q=\min\left\{\frac{1}{R},\frac{1}{L}\right\}.\]
Clearly, $F$ is continuous on $[-q,0]$ for every $s\in(0,\pi R/4)$ and $$F(0)=-2\cos^2\big(s/R+\pi/4\big)<0,$$while$$F\left(-q\right)=\left\{\begin{array}{ll}2{\left[(1-R/L)\sin\big(s/R+\pi/4\big)+R/L\right]}^2-1>0,&L\geq R,\\1-\sqrt{1-{(L/R)}^2}>0,&L<R.\end{array}\right.$$
This implies that there exists $v_0(s)\in(-q,0)$ such that \eqref{eq:d-1.13} holds and, therefore, (i) $v$ is uniformly bounded in the region $II$, (ii) the inequality in \eqref{eq:d-1.15} can be considered to be strict, and (iii) $v$ experiences a jump on ${\textbf{\textit x}}_r$. Note that, at the same time, $\theta$ is continuous across ${\textbf{\textit x}}_r$ by construction.

It remains to show that the characteristic curves cover the entire region $II$, without intersecting each other. We begin by proving 
\begin{lemma}
\label{l:inc}
The functions $v_0$ and $\theta^*$ are, respectively, strictly increasing and strictly decreasing on $(0,\pi R/4)$.
\end{lemma}
\begin{proof}
Taking the derivative of both sides of \eqref{eq:d-1.13} with respect to $s$, solving for $v_0^\prime(s)$, and using \eqref{eq:d-1.15}, we determine that $v_0^\prime(s)>0$ for all $s\in(0,\pi R/4)$. This establishes monotonicity of $v_0$. Likewise, solving \eqref{cond3}-\eqref{cond4} for $\cos{\theta^*}$, taking the derivative with respect to $s$ and using the just established fact that the $v_0^\prime>0$ on $(0,\pi R/4)$, along with \eqref{eq:d-1.11}, \eqref{eq:d-1.15}, proves that $\theta^{*\prime}<0$ on $(0,\pi R/4)$.
\end{proof}
To demonstrate that no two characteristic curves can intersect, we suppose, by contradiction, that a circular arc of a characteristic $C_1$ intersects another circular arc of a characteristic $C_2$ before reaching $y=x$, where $C_1$ corresponds to $s= s_1$ whereas $C_2$ corresponds to $s=s_2$ with $s_1<s_2$. Using \eqref{xstd3}, \eqref{ystd3} and the monotonicity of $v_0$, we know that the curvature of $C_1$ is greater than the curvature of $C_2$. Since $C_1$ starts out (i.e. at $t=0$) to the left of $C_2$, this intersection could not be merely tangential since a circle of larger curvature can't sit outside of a circle of smaller curvature. Thus, the intersection is transversal. Now the angle between an incoming characteristic and the line $y=x$ is the non-negative angle $\theta^* + \pi/4$ and, if the intersection is transversal, then necessarily $\theta^* (s_1) + \pi/4 < \theta^*(s_2) + \pi/4$ contradicting Lemma \ref{l:inc}.

We end this section by plotting the analytical counterparts of Figs.~\ref{f:nd-2}-\ref{f:nd-3} obtained in MATLAB\textsuperscript{\textregistered} using the characteristics solutions constructed above.

\begin{figure}[htb]
\centering
    \includegraphics[width=3in]
                    {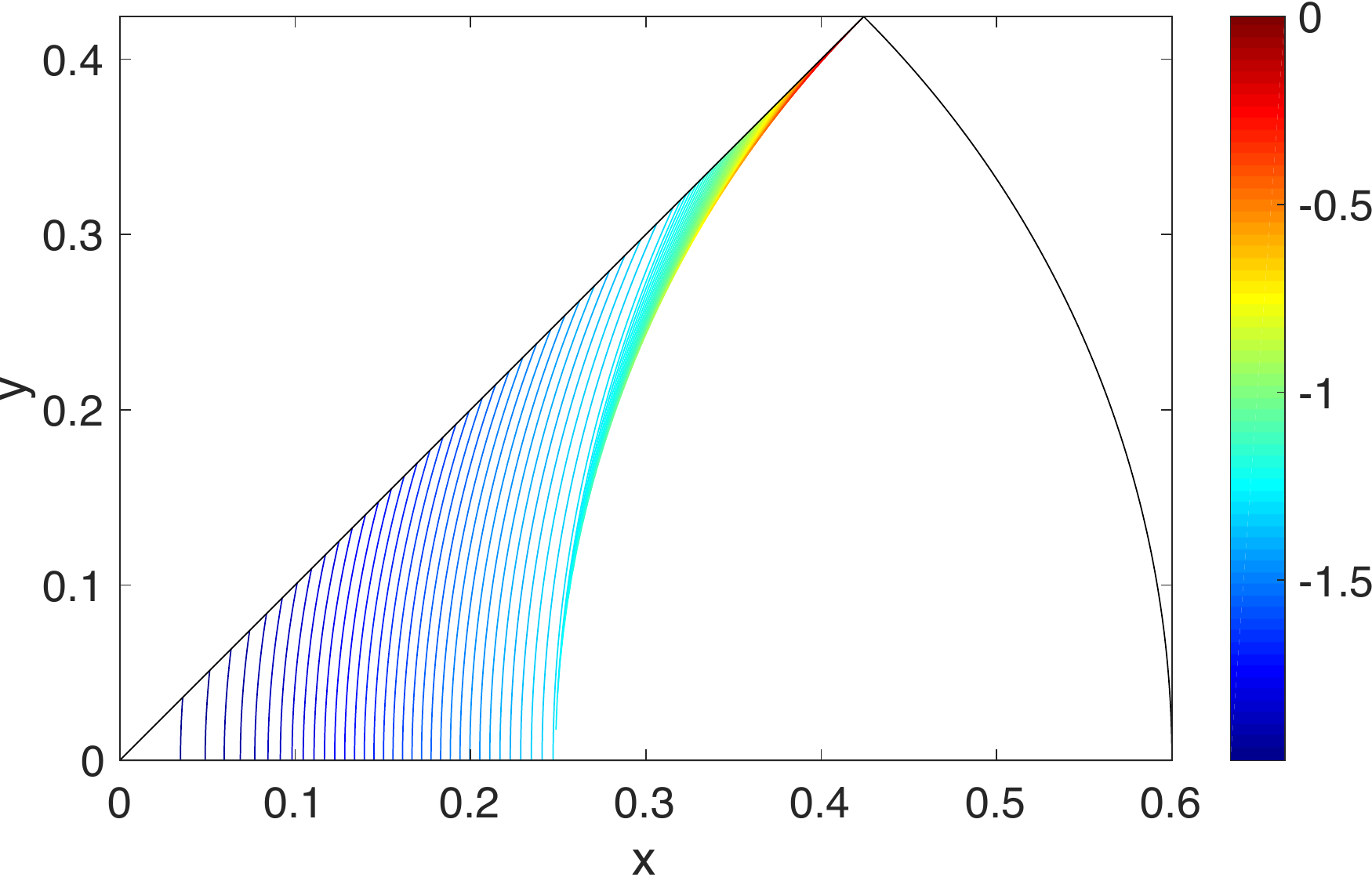}
    \caption{Level curves for the divergence of $u$, where $u$ is a solution obtained using characteristics. The divergence is constant in the empty region.}
  \label{f:cs-1}
\end{figure}
\begin{figure}[htb]
\centering
    \includegraphics[width=3in]
                    {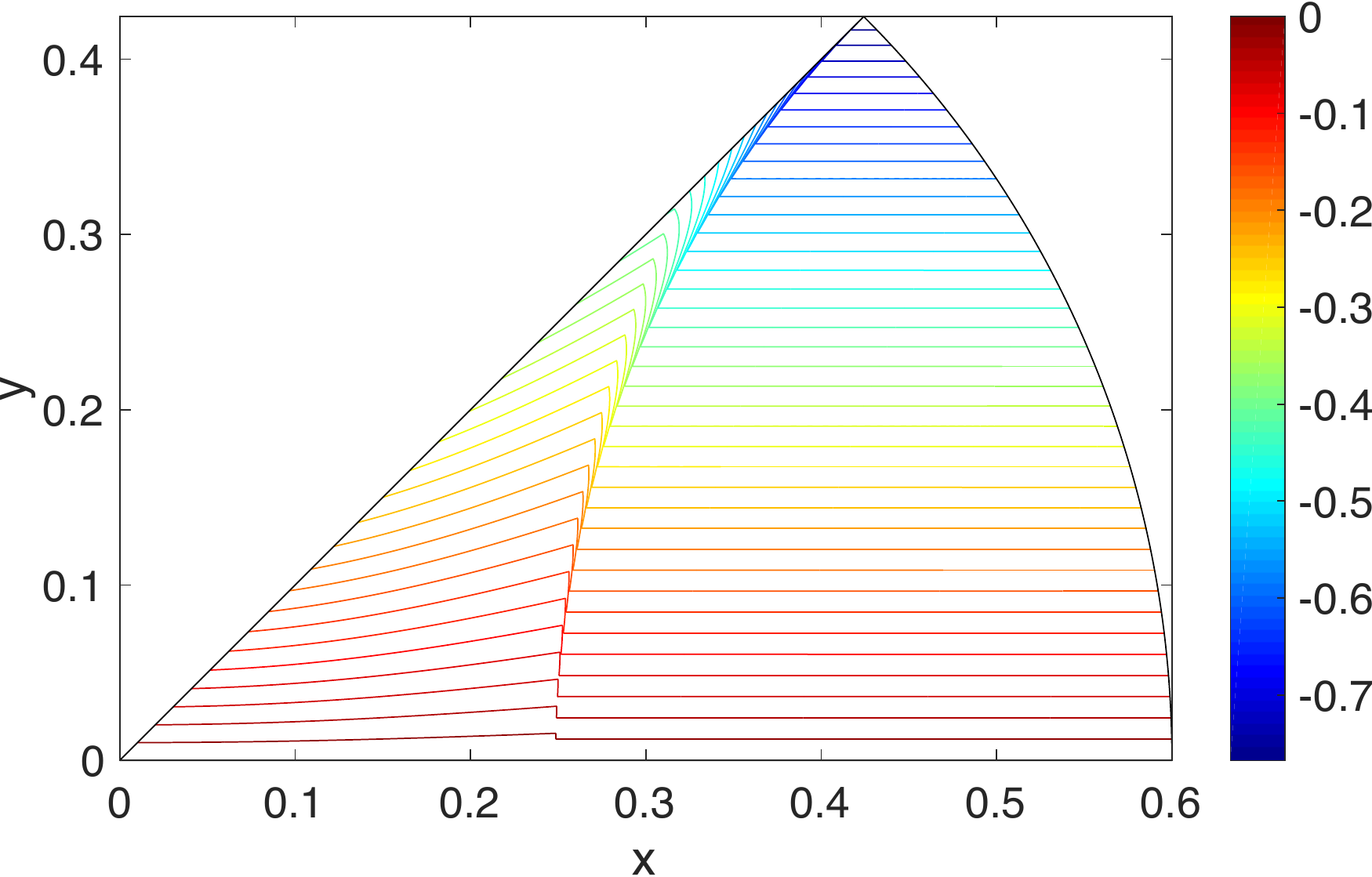}
    \caption{Level curves for the angle $\theta$, where $u=(\cos{\theta}, \sin{\theta})$ is depicted in Fig.~\ref{f:cs-1}.}
  \label{f:cs-2}
\end{figure}
The Figs.~\ref{f:cs-1}-\ref{f:cs-2} should be compared to the solution in the sector in Figs.~\ref{f:nd-2}-\ref{f:nd-3}, corresponding to the polar angle ranging between $0$ and $45^\circ$. The regions $I$ and $II$ are clearly visible in Figs.~\ref{f:nd-2}-\ref{f:nd-3} and there is a good match between Figs.~\ref{f:cs-1}-\ref{f:cs-2} and Figs.~\ref{f:nd-2}-\ref{f:nd-3} in these regions. The discrepancy between the solutions in the region $III$ can be attributed to the qualitative differences between minimizers of the $\varepsilon$-level and $\Gamma$-limit problems. The energies of the characteristics and numerical solutions are depicted in Fig.~\ref{f:cs-3} for a small range of $L$ values. The plots demonstrate that both numerical solution and the solution constructed using charateristics have energy increasing with $L$ on $L\in[0.1,0.7]$. The systematic difference between the graphs can once again be explained by the fact that the corresponding functions are critical points of the different energy functionals.
\begin{figure}[htb]
\centering
    \includegraphics[width=3in]
                    {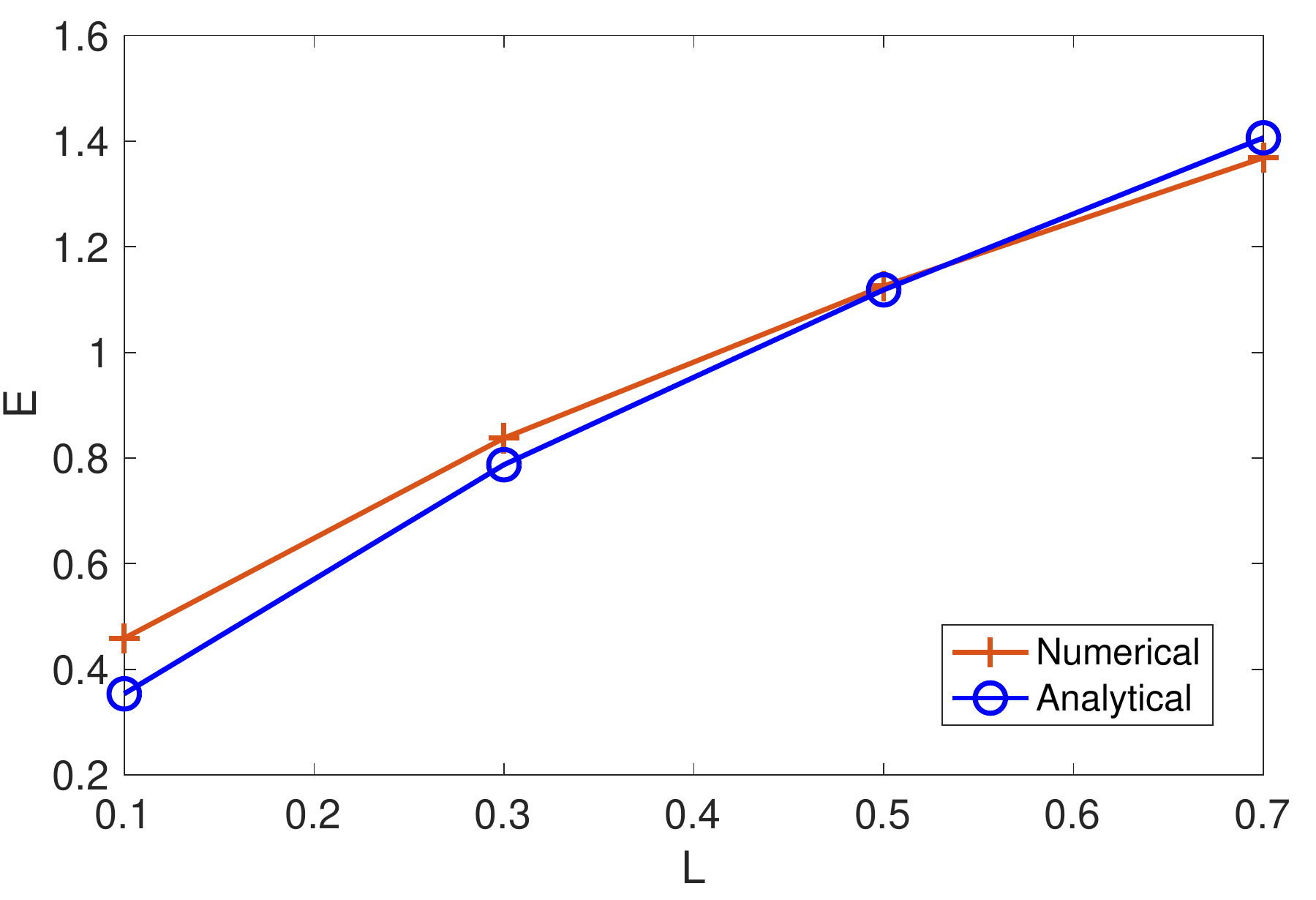}
    \caption{Energy of the critical point as a function of $L$.}
  \label{f:cs-3}
\end{figure}
{
\subsection{An example in an annulus: curved walls } 

In this section we briefly outline an example where our analysis suggests that the jump set can occur along a portion of the boundary with a jump set, and might in general not be a straight line segment.

We fix a number $R > 1,$ and let $\Omega$ denote an annulus described in polar coordinates by $\Omega := \{1 < r < R \}.$ For the boundary conditions $g$ defined by $ g(1,\theta) = -\widehat{e}_\theta, 
g(R, \theta) = \widehat{e}_\theta,$ we study the problem of minimizing the $E_0$ energy among competitors $u_{\partial \Omega}\cdot \nu_{\partial \Omega} = g\cdot \nu_{\partial \Omega} = 0.$ It is reasonable to expect that a minimizer is radial so we work within the ansatz 
\begin{align} \label{ann-ansatz}
u(r, \theta) = p(r) \widehat{e}_r + q(r) \widehat{e}_\theta, 
\end{align} 
where $p^2 + q^2 \equiv  1, p(1) = 0 = p(R).$  Within this ansatz, $\dive u = \frac{1}{r}(r p(r))_r,$ and the jump set is composed of a union of circles, possibly occurring at the boundary of $\Omega.$ Away from jumps, criticality of $\int (\dive u)^2 \,dx $ within this ansatz requires that $p(r)$ satisfies the ODE $\frac{\partial}{\partial r} \left( \frac{1}{r} \frac{\partial}{\partial r} (r p(r)) \right) = 0,$ so that  $p(r)$ takes the form $p(r) = Cr + \frac{D}{r}$ for constants $C, D.$ In the absence of a jump circle in $\Omega$, the boundary conditions on $p$ would force $p(r) \equiv 0$, and then either $q(r) \equiv 1$ or $q(r) \equiv -1$. This results in boundary walls, either along the circle $\rho = 1$ or along $\rho = R$, respectively carrying energies $E_0(\hat{e}_\theta) = \frac{8\pi}{3}$ or $E_0(-\hat{e}_\theta) = \frac{8\pi R}{3}.$ 

Elementary calculations that we present below demonstrate that for any $R>1,$ for an interval of $L-$values of the form $(0,L_*(R))$ where $L_*(R) < \frac{8}{3} \frac{R^2 - 1}{R^2 + 1} \left( 1 - \frac{\sqrt{2} R}{\sqrt{R^2+1}} \left(\frac{3}{4}\right)^{3/2} \right)$, the energy $E_0$  within the ansatz \eqref{ann-ansatz} has an internal wall with energy strictly smaller than $\frac{8\pi}{3},$ which is the energy associated to a boundary wall. 
 
On the other extreme, we also show below that for any fixed $R > 1$ and $L$ sufficiently large depending on $R,$ the minimizer of $E_0$ with these ``mismatch'' boundary conditions and the radial ansatz \eqref{ann-ansatz} necessarily has its wall at the inner boundary $\rho = 1$. The associated energy is $E_0(\hat{e}_\theta) = \frac{8\pi}{3}.$ 

In the absence of boundary walls, we see that an internal wall is present which we suppose occurs along the circle of radius $\rho$ with corresponding normal component given by $a = p(\rho).$  It then follows that the function $p$ is given by the formula 
\begin{align*}
p(r) = \left\{
\begin{array}{cc}
\frac{a \rho}{\rho^2 - 1}\left(r - \frac{1}{r} \right) & 1 < r < \rho\\
- \frac{a \rho}{R^2 - \rho^2} \left( r - \frac{R^2}{r} \right) & \rho < r < R.
\end{array}
\right. 
\end{align*}
We set $q(r) = - \sqrt{1 - p^2(r)} $ for $1 < r < \rho$ and $q(r) = \sqrt{1- p^2(r)}$ for $\rho < r < R.$ Computing the energy of $u_\ast = u_\ast^{\rho, a(\rho)},$ we find 
\begin{align} \label{eq:energy}
E_0(u_\ast) = 2 \pi L a^2 \rho^2 \left(\frac{1}{\rho^2 - 1} + \frac{1}{R^2 - \rho^2} \right) + \frac{8}{3} \pi \rho (1 - a^2)^{3/2}. 
\end{align}
For $\rho \in (1, R),$ we now enforce the natural boundary conditions \eqref{el:4} and the criticality condition for the jump circle \eqref{el:13} where we must use $\kappa = -\frac{1}{\rho}.$ The natural boundary conditions, \eqref{el:4} yield
\begin{align} \label{eq:nbc}
2 a L \rho \left( \frac{1}{\rho^2 - 1} + \frac{1}{R^2 - \rho^2} \right) = 4 a (1 - a^2)^{1/2}.
\end{align}
Similarly, criticality of jump, \eqref{el:13}, yields 
\begin{align} \label{eq:jump}
\frac{4a^2 \rho^2}{(R^2 - \rho^2)^2} - \frac{4a^2 \rho^2}{(\rho^2 - 1)^2}  = -\frac{8}{3L\rho} (1 - a^2)^{1/2} (1 + 2a^2). 
\end{align}
We define $c_a := \frac{3a^2}{2a^2 + 1} \in (0,1),$ and use \eqref{eq:nbc} in \eqref{eq:jump} to discover that 
\begin{align*}
\rho^2 \left( \frac{1}{\rho^2 - 1} - \frac{1}{R^2 - \rho^2} \right) = \frac{1}{c_a},
\end{align*}
or equivalently,
\begin{align} \label{quadratic}
(1 - 2c_a)\rho^4 - \rho^2(1+R^2)(1-c_a) +R^2 = 0. 
\end{align}
If $c_a = 1/2$ or equivalently $a = 1/2,$ we find that $\rho^2 = \frac{2R^2}{1+R^2}.$ We argue that when $a > 1/2,$ the equation \eqref{quadratic} viewed as a quadratic in $\rho^2$ has no real zeroes, while if $a < 1/2,$ it has a unique zero in $(1,R^2).$ Indeed, the sum of the roots $\rho_1^2 + \rho_2^2 = \frac{(1+R^2)(1-c_a)}{1 - 2c_a}.$ This yields the condition that $c_a < \frac{1}{2},$ or equivalently, that $a < 1/2.$  For any such $a,$ it is easy to argue that \eqref{quadratic} has a unique zero $\rho^2(a) \in (1 , R^2).$ Indeed, evaluating the LHS of \eqref{quadratic} at $\rho^2 = 1,$ yields  $c_a (R^2 - 1) > 0,$ while evaluating \eqref{quadratic} at $\rho^2 = R^2$ yields  $c_a R^2(1-R^2) < 0.$   It must therefore have an odd number of real zeroes in this interval, implying uniqueness of the zero.  

We therefore have that a solution of the pair of equations \eqref{eq:nbc}-\eqref{eq:jump} satisfies $0 < a \leqslant 1/2.$ To derive a  bound on $\rho$ satisfying these equations we use the definition of $c_a$ in \eqref{quadratic} to find
\begin{align} \label{eq:asquare}
a^2 = \frac{(\rho^2 - 1)(R^2 - \rho^2)}{- 4 \rho^4 + ( 1 + R^2)\rho^2 + 2 R^2}
\end{align}
The condition $0 < a^2 < \frac{1}{4},$ then yields from \eqref{quadratic} that $\rho^2 < \frac{2R^2}{R^2 + 1} .$ 
Using \eqref{eq:asquare} in \eqref{eq:nbc}, we arrive at a single equation in $\rho,$ 
\begin{align}
L \rho \left( \frac{R^2 - 1}{(\rho^2 - 1)(R^2 - \rho^2)}\right) = 2 \left( \frac{3(\rho^4 - R^2)}{4\rho^4 - (1+R^2)\rho^2 - 2R^2}\right)^{1/2}. 
\end{align}
Introducing $z := \rho^2,$ we are led to seeking $z \in (1, \frac{2R^2}{1+R^2})$ solving the polynomial 
\begin{align*}
g_{R,L}(z) := L^2(R^2-1)^2 z \left( z^2 - \frac{1+R^2}{4}z - \frac{R^2}{2} \right) + 3 (R^2-z^2)(z-1)^2(R^2-z)^2. 
 \end{align*}
Note that $g_{R,L}(1) = \frac{3}{4} L^2(1 - R^2)^3 <0.$  Concerning this polynomial, we note that for \textit{fixed} $R > 1,$ and $L$ sufficiently large depending on $R,$ the polynomial $g_{R,L}(z) < 0$ for $z \in (1 , \frac{2R^2}{1+R^2}).$ Consequently, $g_{R,L}$ does not have a zero for such $(R,L)$ values. It follows that for such $(R,L)-$ values, an internal wall can not occur. Therefore, among radial competitors with a single wall, the minimizer must have a wall at the inner boundary $\rho = 1,$ with corresponding $a = 0.$ As discussed earlier, this entails that the minimizer is $u_\ast \equiv \widehat{e}_\theta,$ with $E_0(u_\ast) = \frac{8}{3}\pi.$ 

We next argue that for any $R > 1,$ if $L$ is sufficiently small, the wall occurs in the interior. To see this, we simply construct a competitor whose energy is smaller than $\frac{8\pi}{3}.$ Setting $a  = 1/2,$ and $\rho^2 = \frac{2R^2}{R^2 + 1}$ in \eqref{eq:energy}, we find that the energy of this competitor is smaller than $\frac{8\pi}{3},$ provided 
\begin{align*}
L < \frac{8}{3} \frac{R^2 - 1}{R^2 + 1} \left( 1 - \frac{\sqrt{2} R}{\sqrt{R^2+1}} \left(\frac{3}{4}\right)^{3/2} \right),
\end{align*}
where the quantity on the right is clearly positive since $R > 1.$ A radial minimizer with a single wall of course exists by elementary compactness and continuity arguments. It follows that for any $R > 1,$ for an interval of small positive $L-$values, the radial minimizer has an internal wall.

In Fig.~\ref{f:an-1}, we illustrate observations made in this section by presenting the results of gradient flow simulations for the functional $E_\varepsilon$ for two different values of $L$. For the smaller value of $L=0.2$, the (local) minimizer has a shallower circular wall in the interior of the domain, while the the minimizer for $L=2$ has a deeper wall that coincides with the inner boundary of the annulus. Note that the simulations were done without assuming that competitors are radially symmetric---the apparent symmetry of minimizers suggests that it might be reasonable to consider the ansatz \eqref{ann-ansatz}.

\begin{figure}[htb]
\centering
    \includegraphics[width=3in]
                    {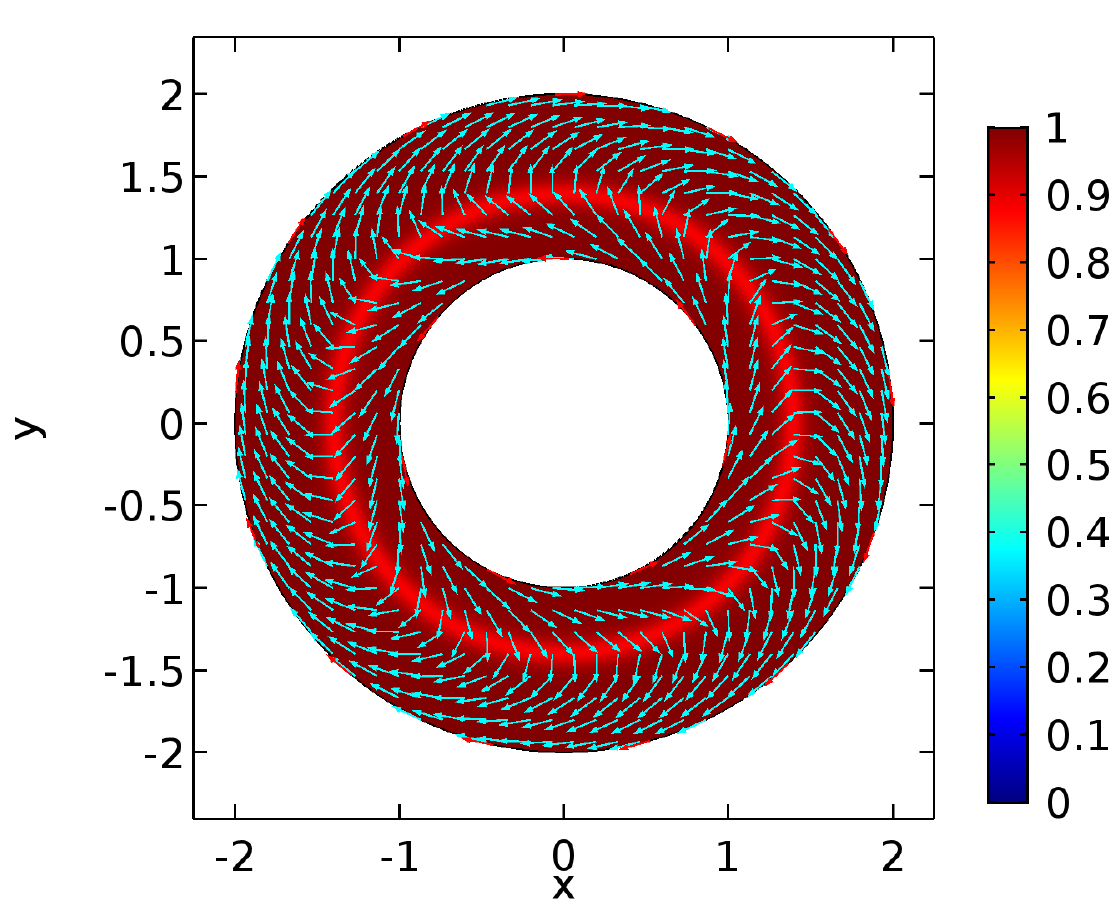}     \includegraphics[width=3in]
                    {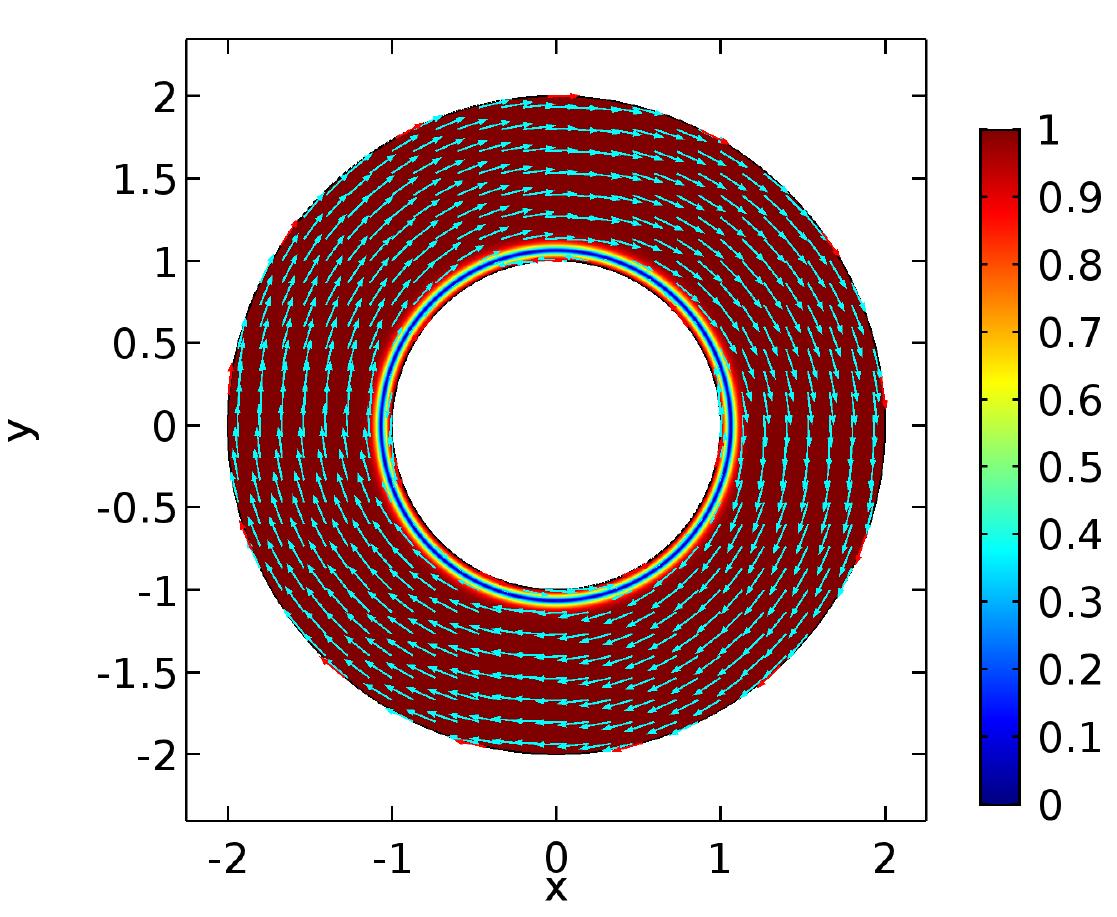}
    \caption{Energy minimizers in an annulus for $L=0.2$ (left) and $L=2$ (right). Here $\varepsilon=0.03$ and the color represents $|u|$.}
  \label{f:an-1}
\end{figure}
}
 \section{Results for the special case of a rectangle}
 
 In this section we pose the problem on a rectangle, taking $\Omega=(-T,T)\times(-H,H)$ for positive constants $T$ and $H$. Furthermore,
 we specialize the boundary conditions on competitors $u:\Omega\to\R^2$ to be given by
 \beq
 u(x,\pm H)=\big(\pm\sqrt{1-a^2},a\big)\;\mbox{for}\;\abs{x}\leq T,\quad u\;\mbox{is}\;2T\mbox{-periodic in}\;x.\label{rectbc}
 \eeq 
 for some constant $a\in [0,1)$. The rationale for considering $E_\e$ and the $\Gamma$-limit $E_0$ in this rather special setting is to 
 focus on the structure of wall transitions in as simple a situation as possible. A primary focus will be on examining the relative favorability of one-dimensional--that is purely $y$-dependent structures versus two-dimensional structures such as cross-ties the one associates with related models in micromagnetics, cf. e.g. \cite{ARS}. Other goals we have in mind concern in focusing on this special case are to better understand the relative weights given to jump energy versus divergence for minimizers as well as the possible emergence of periodic structures on a scale smaller than the fixed rectangle width $2T$.

 \subsection{Study of the problem in a rectangle within a one-dimensional ansatz}
We begin our analysis of $E_\e$ and $E_0$ on the rectangle subject to the boundary conditions \eqref{rectbc} by first studying the variational problem among one-dimensional competitors, i.e. functions of $y$ alone. More specifically, for $0 \leqslant |a| < 1$ we consider the space of admissible functions
\begin{align*}
\sA^{1} (a) := \{ u = u(y) \in H^1((-H,H); \R^2), u(\pm H) = \big(\pm\sqrt{1-a^2},a\big) \}. 
\end{align*}
and consider the variational problem 
\beq
\min_{u \in \sA^1(\alpha)} \oned (u) \label{1deps}
\eeq
where
\begin{align} \label{1denergy}
\oned (u) := \frac{1}{2} \int_{-H}^H \e |u^\prime|^2 + \frac{1}{\e} (|u|^2 - 1)^2 + L ( u_2^\prime)^2 \,dy.
\end{align}
The corresponding $\Gamma$-limit $\onedlim$ is now defined
over the class
\beq
\sA^0:=\{u=(u_1,u_2):\, u_1^3\in BV\big((-H,H)\big),\,u_2\in H^1\big((-H,H)\big),\,u^{(2)}(\pm H)=a,\,\abs{u}=1\;a.e.\;y\}, 
\label{1dadmiss}
\eeq
where the boundary conditions on $u_2$ comes from  \eqref{normaltrace}.
Then $E_0$ from \eqref{Ezero} takes the form
\begin{multline}
\onedlim (u) =: \frac{L}{2}\int_{-H}^H (u_2^\prime)^2 \,dy + 4/3 \sum_{y_j \in J_{u_1}} \big(1 - u_2^2(y_j) \big)^{3/2}\\+
\frac{1}{6}\abs{\,u_1(-H)+\sqrt{1-a^2}}^3+\frac{1}{6}\abs{\,u_1(H)-\sqrt{1-a^2}}^3 .\label{1dgamdefn}
\end{multline}

Not surprisingly, in this one-dimensional setting we can prove a much stronger compactness statement than is possible in the two-dimensional setting of Theorem \ref{comp}. Here we establish
\bthm\label{onedcompactness}
Let $u_\e=(u_\e^{(1)},u_\e^{(2)}) \in \sA^1(a)$ with $\oned (u_\e) \leqslant C.$ Then, up to extraction of subsequences,  one has
$u_\e^{(1)}\to u_1$ in $L^3(-H,H)$ for some function $u_1$ such that $u_1^3\in BV(-H,H)$ and one has
$u_\e^{(2)}\to u_2$ in $C^{0,\gamma}$ for all $\gamma<1/2$. Furthermore, $\abs{(u_1,u_2)}=1$ a.e.
\ethm
\begin{proof}
Precompactness of $\{u_\e^{(2)}\}$ in $C^{0,\gamma}(-H,H)$ for $\gamma<1/2$ is clear from the uniform $H^1$ bound and Sobolev imbedding. The thrust of the rest of the proof will be to prove the statement about $\{u_\e^{(1)}\}.$ To this end, we define        
            \begin{align*}
                \psi_\e(y) := \int_{-u_\e^{(1)}}^{u_\e^{(1)}} \big(1-s^2 - (u_\e^{(2)})^2\big) \,ds = 2u_\e^{(1)}-2u_\e^{(1)}(u_\e^{(2)})^2
                -\frac{2}{3} (u_\e^{(1)})^3 .
            \end{align*}
            Since we have a uniform $L^4$ bound on $u_\e^{(1)}$ from the energy bound $\oned (u_\e) \leqslant C$ it readily follows 
            that $\psi_\e$ is uniformly bounded in $L^1(-H,H).$
            Now we estimate the total variation of $\psi_\e$. We have
            \begin{eqnarray*}
            \int_{-H}^H\abs{\psi_\e'}\,dy&&\leq
            2\int_{-H}^H \abs{1-\abs{u_\e}^2}\abs{u_\e^{(1)}\,'}\,dy+4\int_{-H}^H \abs{u_\e^{(1)}u_\e^{(2)}}\abs{u_\e^{(2)}\,'}\,dy\\
            &&\leq \frac{1}{\e} \int_{-H}^H\left(1-\abs{u_\e}^2\right)^2\,dy+ \e\int_{-H}^H\abs{u_\e^{(1)}\,'}^2\,dy+
             \int_{-H}^H \left((u_\e^{(1)})^4+(u_\e^{(2)})^4\right)\,dy\\ &&+  2\int_{-H}^H (u_\e^{(2)}\,')^2\,dy
             <C.
            \end{eqnarray*}
                        
Concluding the desired compactness of $\{u_\e^{(1)}\}$ relies on an algebraic identity. Using the $BV$ bound on $\{\psi_\e\}$, and passing to subsequences that we do not denote explicitly, we know that $\psi_\e$ converges in $L^1.$ We now show that $\{u_\e^{(1)}\}_{\e > 0}$ is a Cauchy sequence in $L^3.$ For any $0 < \e < \delta,$ we have 
            \[
                \frac{4}{3}\big( (u_\e^{(1)})^3 - (u_\delta^1)^3 \big)
                = \Big(\psi_\e - \psi_\delta \Big) -2(1 - |u_\e|^2) u_\e^{(1)}  +2 (1 - |u_\delta|^2)u_\delta^{(1)} .
                            \]  
            Hence, using Cauchy-Schwarz we obtain
            \begin{align*}
                \frac{4}{3}\int_{-H}^H \big| (u_\e^{(1)})^3 - (u_\delta^{(1)})^3 \big| \,dy\leq\int_{-H}^H |\psi_\e - \psi_\delta| \,dy+ \e^{1/3}\int_{-H}^H |u_\e^{(1)}|^2\,dy + \frac{1}{\e^{1/3}} \int_{-H}^H ( 1 - |u_\e|^2)^2\,dy \\ + \delta^{1/3} \int_{-H}^H |u_\delta^{(1)}|^2 \,dy+ \frac{1}{\delta^{1/3}} \int_{-H}^H(1 - |u_\delta|^2)^2\,dy.
            \end{align*}
            Since  $\{u_\e^{(1)}\}$ is uniformly bounded in $L^4$ by the energy bound, we can invoke the $L^1$ convergence of $\{\psi_\e\}$ to find that
            \begin{align}
                \int_{-H}^H \big|(u_\e^{(1)})^3 - (u_\delta^{(1)})^3\big|  \leqslant o(1),
            \end{align}
            as $\delta \to 0.$ Since $|a - b|^3 \leq 4|a^3 - b^3|$, it follows that $\{u_\e^{(1)}\}$ is Cauchy in $L^3,$ and has a limit in this space, denoted $u_1.$ Denoting the limit of $u_\e^{(2)}$ by $u_2$, it follows from the energy bound that $u_1^2 + u_2^2 = 1$ a.e. in $(-H,H)$. Consequently, the limit of the $\psi_\e$ satisfies 
            \begin{align} \label{eq:L1convpsieps}
                \psi_\e \to \frac{4}{3} (u_{1})^3
            \end{align}
in $L^1.$ By lower semicontinuity of the $BV$ norm under $L^1$ convergence, we conclude that \[(u_{1})^3 \in BV(-H,H).\] It follows that one-sided limits of $(u_{1})^3$ exist at all $y\in(-H,H)$. Combined with $u_{2}$ being continuous on the same interval, this implies that 
$\abs{u_{1},u_{2}}=1$ everywhere on $(-H,H)$.
\end{proof}

In light of the the preceding compactness result Theorem \ref{onedcompactness}, one can establish a full $\Gamma$-convergence result in this one-dimensional setting without an assumption on the limiting functions lying in $BV$. While of course the arguments presented for lower-semicontinuity and for the construction of recovery sequences in proving Theorem \ref{main} apply in this one-dimensional setting as well, we wish to give alternate proofs here since they are so much simpler and therefore make more transparent the key elements of the argument.

\bthm \label{thm1Dmain}
Let $u \in \sA^0.$ Then 
\noindent {\bf 1.} For any sequence $u_\e \in \sA^1(a)$ satisfying $u_\e \divcon u,$ we have,
\beq \label{1dlwb}
\liminf_{\e \to 0} \oned (u_\e) \geqslant \onedlim (u) .
\eeq
\noindent {\bf 2.} There exists a sequence $w_\e \in\sA^1(a)$ with $w_\e \divcon u$ and 
\begin{align} \label{1dupb}
\lim_{\e \to 0} \oned (w_\e) = \onedlim (u). 
\end{align}
\ethm
\begin{proof}
To prove (\ref{1dlwb}), let $u \in \sA^0, u_\e \in \sA^1(a)$ be given with $u_\e \divcon u.$ Let $J = \{a_1,a_2, \cdots \} \subset [-H,H]$ denote the at most countable jump set of $u_1^3,$ and hence that of $u_1.$ We note that $\lim_{y \to a_j^\pm } u_1(y)$ exists for each $a_j \in J$ and are denoted $u_1^\pm(a_j )$ respectively. Let now $\delta > 0$ be arbitrary. If $J$ is an infinite set, we suppose $N = N(\delta) $ is a positive integer such that 
\begin{align*}
 \sum_{j=N+1}^\infty \Big(1 - \big(u_2(a_j)\big)^2\Big)^{3/2} < \delta. 
\end{align*}
If $J$ is a finite set, we simply let $N$ denote the cardinality of $J$ instead. For each $j \in \{1,\cdots, N\},$ we set $I_j := (a_j - \tfrac{\delta}{N^{1+\alpha}}, a_j + \tfrac{\delta}{N^{1+\alpha}}),$ where $\alpha > 0$ is chosen such that the intervals $I_j$ are disjoint. Finally, as in the preceding theorem, we let $\psi_\e :(-H,H) \to \R$ be defined by the formula 
\begin{align}
\psi_\e(y) = 2 \big( u_\e^{(1)} - u_\e^{(1)} ( u_\e^{(2)})^2 - \frac{1}{3} (u_\e^{(1)})^3 \big). 
\end{align}
By the Cauchy-Schwarz inequality, we have 
\begin{align*}
|\psi_\e^\prime| \leqslant \frac{1}{\e} \big( 1 - |u_\e|^2 \big)^2 + \e \big| {u_\e^{(1)}}^\prime\big|^2 + L \big|{u_\e^{(2)}}^\prime\big|^2 + \frac{1}{L} \big(u_\e^{(1)}u_\e^{(2)}\big)^2
\end{align*}
With an eye towards proving (\ref{1dlwb}), we estimate,
\begin{align} \notag
\oned(u_\e) &\geqslant \frac{1}{2}\int_{(-H,H) \backslash \cup_{j=1}^N I_j} L \big|(u_\e^{(2)})^\prime\big|^2 \,dy + \frac{1}{2}\sum_{j=1}^N  \int_{I_j} \e |u_\e^\prime|^2 + \frac{1}{\e}\big( |u_\e|^2 - 1 \big)^2 + L\big|(u_\e^{(2)})^\prime\big|^2 \,dy \\ \label{lwbest1}
&\geqslant \frac{1}{2}\int_{(-H,H) \backslash \cup_{j=1}^N I_j} L \big|(u_\e^{(2)})^\prime\big|^2 \,dy + \frac{1}{2}\sum_{j=1}^N \int_{I_j} |\psi_\e^\prime(y)|\,dy - \frac{1}{2L} \int_{I_j} \big(u_\e^{(1)} u_\e^{(2)}\big)^2 \,dy.
\end{align}  
We claim that the last term 
\begin{align} \label{remainder1d}
\sum_{j=1}^N \int_{I_j} \big(u_\e^{(1)} u_\e^{(2)}\big)^2 \,dy \leqslant O(\delta). 
\end{align}
Assuming this for the time being, we first complete the proof. Taking $\liminf_{\e \to 0}$  on both sides of (\ref{lwbest1}), using the compactness gleaned from Thm. (\ref{onedcompactness}) and Eq. (\ref{eq:L1convpsieps}), we find for fixed $\delta > 0,$  
\begin{align} \label{1dlwblaststep}
\liminf_{\e \to 0} \oned(u_\e) \geqslant \frac{1}{2}\int_{(-H,H) \backslash \cup_{j=1}^N I_j} L \big|u_2^\prime\big|^2 \,dy + \frac{1}{2}\sum_{j=1}^N\frac{4}{3} \Big| u_1^3\big(a_j^+\big) - u_1^3\big(a_j^-\big)  \Big|^3 - O(\delta) - o_\delta(1)
\end{align}
where in the last line $|O(\delta)| \leqslant C\delta$ for a universal constant $C,$ and $o_\delta(1) \to 0$ as $\delta \to 0.$ We point out that we have absorbed the absolutely continuous part of $(\frac{4}{3} u_1^3)^\prime,$ (cf. Eq.(\ref{eq:L1convpsieps})) into the $o_\delta(1)$ term by the monotone convergence theorem, since $|\cup_{j=1}^N I_j| \sim \delta \to 0.$ 
Passing to the limit $\delta \to 0,$ once again invoking monotone convergence for the first term on the right hand side of (\ref{1dlwblaststep}), and using the observation that $u_1(a_j+) = - u_1(a_j-)$ for each $j,$ we conclude that 
\begin{align*}
\liminf_{\e \to 0} \oned(u_\e) \geqslant \frac{L}{2} \int_{-H}^H (u_2^\prime)^2 \,dy + \frac{4}{3}\sum_{a_j \in J} \Big( 1 - u_2^2(a_j)\Big)^{3/2} = \onedlim (u).
\end{align*}
To conclude the proof, we must prove the estimate (\ref{remainder1d}). Once again, using Cauchy-Schwarz, 
\begin{align*}
\int_{\bigcup_{j=1}^N I_j} \big(u_\e^{(1)} u_\e^{(2)}\big)^2 \,dy \lesssim \int_{\bigcup_{j=1}^N I_j}  ((u_\e^{1})^2 + (u_\e^{(2)})^2 - 1)^2 \,dy + \Big| \bigcup_{j=1}^N I_j\Big| \lesssim \e + N \cdot \frac{\delta}{N^{1+\alpha}} \lesssim O(\delta). 
\end{align*}
This completes the proof of the remainder estimate (\ref{remainder1d}) and therefore the proof of the lower bound (\ref{1dlwb}). 

\medskip\noindent To prove (\ref{1dupb}), let $u=(u_1,u_2)\in\sA^0$ so that $\onedlim(u)<\infty$ and denote the jump set of $u_1$ by $J_{u_1}$. We first observe that since $u_2\in H^1(-H,H)$, it is H\"older continuous on $[-H,H]$ for any H\"older exponent $\alpha\in (0,1/2)$. Since also $\abs{u(y)}=1$ at a.e. $y\in (-H,H)$ we may assume, after perhaps redefining $u_1$ on a set of measure zero, that $u_1$ is H\"older continuous with exponent $\alpha/2$ off of its jump set with
\beq
 \abs{u_1(y)}=\sqrt{1-u_2^2(y)}\quad\mbox{at every}\; y\in (-H,H)\setminus J_{u_1}.\label{u1u2}
 \eeq

Now we fix any sequence  $\{\delta_k\}$ approaching $0$ as $k\to\infty$.  Then we introduce the subset $J_{u_1}^k$of the jump set $J_{u_{1}}$ given by $J_{u_1}^k:=\left\{y\in J_{u_1}:\,\big|[u_1(y)]\big|> 2\delta_k\right\}$. We note that for each fixed $k$ there are only finitely many points $\{y_j^k\}_{j=1}^{N_k}$ in $J_{u_1}^k$ since otherwise the term $\sum_{y_j \in J_{u_1}} \big(1 - u_2^2(y_j) \big)^{3/2}$ in 
$\onedlim(u)$ would be infinite.

For each fixed $k$ we will construct a family of competitors $\{w^{k,\e}\}$ in $\oned$ such that 
\beq
\lim_{k\to \infty}\lim_{\e\to 0}\oned(w^{k,\e})= \onedlim(u),\label{finally}
\eeq
and \eqref{1dupb} will follow by a diagonalization argument. 

On the set $\left\{y:\,\abs{u_2(y)}\leq \sqrt{1-\delta_k^2}\right\}$ we will leave $u_2$ unchanged. We note that in light of \eqref{u1u2},
for any $y\in \left\{y:\,\abs{u_2(y)}\leq \sqrt{1-\delta_k^2}\right\}\setminus J_{u_1}$ one has
\beq
\abs{u_1'(y)}=\frac{\abs{u_2(y)}\abs{u_2'(y)}}{\sqrt{1-u_2(y)^2}}\leq \frac{1}{\delta_k}\abs{u_2'(y)},\label{gH1}
\eeq
and so, in particular, $u_1$ lies in $H^1$ on any intervals in this set.

Now for any $y_j^k\in J_{u_1}^k$ we replace $u_1$ in the interval $[y_j^k-2\e^{5/6},y_j^k+2\e^{5/6}]$ by a standard `Modica-Mortola' type heteroclinic connection and linear interpolation. (In fact, using any $\e^p$ for $4/5<p<1$ works but we will use $p=5/6$ for concreteness.) That is, we define the scalar function $h_\e=h_\e(y;y_j^k)$ for $y\in(-\infty,\infty)$ that bridges the
values $\pm\sqrt{1-u_2(y_j^k)^2}$ at $\pm\infty$ or $\mp\infty$, depending on the sign of $[u_1(y_j^k)]$, via
\[
h_\e(y;y_j^k):=sgn([u_1(y_j^k)])\sqrt{1-u_2(y_j^k)^2}\tanh\bigg(\sqrt{1-u_2(y_j^k)^2}\,\frac{(y-y_j^k)}{\e}\bigg)
\]
and then define the recovery sequence $w^{k,\e}$  as
\[
w^{k,\e}(y)=\big(h_\e(y;y_j^k),u_2(y)\big)\quad\mbox{on}\;[y_j^k-\e^{5/6},y_j^k+\e^{5/6}],
\]
and as $w^{k,\e}(y)=\big(\ell_\e(y;y_j^k),u_2(y)\big)$ on  $\{y:\,\e^{5/6}\leq \abs{y-y_j^k}\leq 2\e^{5/6}\}$ where $\ell_\e(y;y_j^k)$ is a linear interpolation between the values $\pm h_\e(y_j^k\pm \e^{5/6};y_j^k)$ and
the value $u_1(y_j^k\pm2\e^{5/6}).$ We observe that $\abs{h_\e(y_j^k\pm \e^{5/6};y_j^k)}$ is exponentially close to $\sqrt{1-u_2(y_j^k)^2}$ and that $\sqrt{1-u_2(y_j^k)^2}=\abs{u_1(y_j^k)^+}=\abs{u_1(y_j^k)^-}.$ 

A standard calculation yields that 
\[
\frac{1}{2}\lim_{\e\to 0}\int_{y_j^k-\e^{5/6}}^{y_j^k+\e^{5/6}}\e |(\frac{d}{dy} h_\e(y;y_j^k))|^2 + 
\frac{1}{\e} (\abs{\big(h_\e(y;y_j^k),u_2(y_j^k)\big)}^2 - 1)^2 \,dy=
\frac{4}{3}\Big( 1 - u_2^2(y_j^k)\Big)^{3/2}
\]
and so it will follow that
\beq
\frac{1}{2}\lim_{\e\to 0}\int_{y_j^k-\e^{5/6}}^{y_j^k+\e^{5/6}}\e |(w^{k,\e})'(y)|^2 + 
\frac{1}{\e} (\abs{w^{k,\e}}^2 - 1)^2 \,dy=
\frac{4}{3}\Big( 1 - u_2^2(y_j^k)\Big)^{3/2}\label{stepa}
\eeq
as well, once we verify that
\[
\lim_{\e\to 0}\frac{1}{\e}\int_{y_j^k-\e^{5/6}}^{y_j^k+\e^{5/6}} 
 \abs{\left(\abs{\big(h_\e(y;y_j^k),u_2(y_j^k)\big)}^2 - 1\right)^2-\left(\abs{\big(h_\e(y;y_j^k),u_2(y)\big)}^2 - 1\right)^2} \,dy=0.
\]
However, the H\"older bound $\norm{u_2}_{C^{0,\alpha}(-H,H)}<C\norm{u_2}_{H^1(-H,H)} $ invoked for any $\alpha\in (1/3,1/2)$ implies that  
\begin{eqnarray}
&&
\lim_{\e\to 0}\frac{1}{\e}\int_{y_j^k-\e^{5/6}}^{y_j^k+\e^{5/6}} 
 \abs{\left(\abs{\big(h_\e(y;y_j^k),u_2(y_j^k)\big)}^2 - 1\right)^2-\left(\abs{\big(h_\e(y;y_j^k),u_2(y)\big)}^2 - 1\right)^2} \,dy\leq \nonumber\\
&&
\lim_{\e\to 0}\frac{1}{\e}\int_{y_j-\e^{5/6}}^{y_j^k+\e^{5/6}} 
C\abs{u_2(y)-u_2(y_j^k)}\,dy\leq
\lim_{\e\to 0}\frac{1}{\e}\int_{y_j^k-\e^{5/6}}^{y_j^k+\e^{5/6}}  C \abs{y-y_j^k}^{\alpha}=0.\label{potengood1}
\end{eqnarray}
Next we argue that the asymptotic contribution of this construction to the total energy $\oned$ is zero on the interpolating intervals 
$\cup_{y_j^k\in  J_{u_1}^k}\left\{y:\,\e^{5/6}\leq \abs{y-y_j^k}\leq 2\e^{5/6}\right\}$. In light of the exponential approach of $\ell$ to the endstates $\pm\abs{u_1(y_j^k)^+}$ and the H\"older continuity of $u_1$ with exponent $\alpha/2$ for any $\alpha<1/2$  away from the jump set, cf. \eqref{u1u2}, we conclude that
\[
\abs{\frac{d\ell_\e}{dy}}=\frac{| h_\e(y_j^k\pm \e^{5/6};y_j^k)- u_1(y_j^k\pm2\e^{5/6})|}{\e^{5/6}}
\sim \frac{| u_1(y_j^k)- u_1(y_j^k\pm2\e^{5/6})|}{\e^{5/6}}
 \leq C\e^{-5/8-\sigma}\quad\mbox{for any}\;\sigma>0,
\]
where $u_1(y_j^k)$ above refers to the traces $u_1(y_j^k)^+$ or $u_1(y_j^k)^-$ as appropriate.
From this it easily follows that 
\beq
\e\left\{\int_{y_j^k-2\e^{5/6}}^{y_j^k-\e^{5/6}} +\int_{y_j^k+\e^{5/6}}^{y_j^k+2\e^{5/6}} \right\}\abs{\frac{d\ell_\e}{dy}}^2\,dy\to 0\;\mbox{as}\;\e\to 0.\label{interpder}
\eeq
Then in a manner similar to \eqref{potengood1} we can estimate the energetic cost of the potential term in the interpolation zone as
\begin{multline}
\frac{1}{\e}\left\{\int_{y_j^k-2\e^{5/6}}^{y_j^k-\e^{5/6}} +\int_{y_j^k+\e^{5/6}}^{y_j^k+2\e^{5/6}} \right\}
\big(\ell_\e^2+u_2^2-1\big)^2\,dy\\ \leq
 \frac{C}{\e}\left\{\int_{y_j^k-2\e^{5/6}}^{y_j^k-\e^{5/6}} +\int_{y_j^k+\e^{5/6}}^{y_j^k+2\e^{5/6}} \right\}\abs{\ell_\e(y;y_j^k)-u_1(y)}\,dy
\\
\leq \frac{C}{\e}\left\{\int_{y_j^k-2\e^{5/6}}^{y_j^k-\e^{5/6}} +\int_{y_j^k+\e^{5/6}}^{y_j^k+2\e^{5/6}} \right\}\abs{u_1(y_j^k)-u_1(y)}\,dy
 \\ \leq\frac{C}{\e}\left\{\int_{y_j^k-2\e^{5/6}}^{y_j^k-\e^{5/6}} +\int_{y_j^k+\e^{5/6}}^{y_j^k+2\e^{5/6}} \right\}\abs{y-y_j^k}^{\alpha/2}\,dy \leq C\e^{\frac{1}{24}-\sigma}\to 0\;\mbox{as}\;\e\to 0,\label{nopot}
\end{multline}
by choosing $\alpha$ sufficiently close to $1/2$ and hence $\sigma$ sufficiently close to $0.$

On the set of $y$-values
\[
(-H,H)\setminus \left( \cup_{y_j^k\in  J_{u_1}^k}\left\{y:\,\abs{y-y_j^k}\leq 2\e^{5/6}\right\} \cup \left\{y:\,\abs{u_2(y)}> \sqrt{1-\delta_k^2}\right\}\right)
\]
we will leave $u$ unchanged, letting $w^{k,\e}\equiv u$, and appeal to \eqref{gH1} along with \eqref{stepa} to conclude that
\begin{eqnarray}
&&\lim_{\e\to 0}\frac{1}{2}\int_{\{y:\,\abs{u_2(y)}\leq \sqrt{1-\delta_k^2}\}} \e |(w^{k,\e})^\prime|^2 + \frac{1}{\e} (|w^{k,\e}|^2 - 1)^2 + L ( u_2^\prime)^2 \,dy=\nonumber\\
&&\frac{L}{2}\int_{\{y:\,\abs{u_2(y)}\leq \sqrt{1-\delta_k^2}\}}   ( u_2^\prime)^2 \,dy+\sum_{y_j^k\in J_{u_1}^k}\frac{4}{3}\Big( 1 - u_2^2(y_j^k)\Big)^{3/2}.\label{stepb}
\end{eqnarray}
It remains to define the recovery sequence on the set $\left\{y:\,\abs{u_2(y)}> \sqrt{1-\delta_k^2}\right\}$ which is of course equivalent to
$\{y:\,\abs{u_1(y)}<\delta_k\}$ and so includes, in particular, all elements of $J_{u_1}\setminus J_{u_1}^k.$ We note that the measure of the set $\{y:\,0<\abs{u_1(y)}<\delta_k\}$ must approach zero as $\delta_k\to 0$ and on any intervals where the function $u_1$ vanishes (so that necessarily either $u_2\equiv 1$ or $u_2\equiv -1$), we leave $u$ unchanged. 

Writing the open set $\{y:\,\abs{u_1(y)}<\delta_k\}$ as a countable union of disjoint open intervals $\{(a_j^k,b_j^k)\}$ in $(-H,H)$ we note that on any one of these intervals,
say $(a_j^k,b_j^k)$, one must have either\\
 (i) $u_1(a_j^k)=u_1(b_j^k)=\pm\delta_k$, or else \\ (ii) $u_1(a_j^k)=-u_1(b_j^k)=\pm\delta_k$.\\
 
In light of the H\"older condition on $u_2$ it follows that in either scenario (i) or (ii), for each interval $(a_j^k,b_j^k)$ necessarily
$u_2(a_j^k)=u_2(b_j^k)=\pm\sqrt{1-\delta_k^2}$ and so we will take the second component of the recovery sequence
$w^{k,\e}_2$ to be the constant,
\beq \pm\sqrt{1-\delta_k^2}\quad \mbox{on}\; (a_j^k,b_j^k).\label{w2const}
 \eeq
 In scenario (i) we may also take the first component of $w^{k,\e}$ to be constant so that on all intervals $(a_j^k,b_j^k)$ coming from scenario (i) we choose $w^{k,\e}\equiv (\pm\delta,
 \pm\sqrt{1-\delta_k^2})$ with the signs taken appropriately so as to ensure continuity with the region outside $(a_j^k,b_j^k).$ Of course, the contribution to the total energy $\oned$ is zero from this constant $S^1$-valued part of the construction.

It remains to define the first component of  $w^{k,\e}$ on intervals coming from scenario (ii). To this end, we first note that for each $k$, the number of $j$-values where (ii) occurs, say $N_k$, is finite since otherwise $u_1^3$ would have infinite total variation. We also note that since $u_1^3\in BV(-H,H)$ then one has
\[
\lim_{\delta_k\to 0}\int_{\{y:\abs{u_1(y)}<\delta_k\}}\abs{(u_1^3)'}=0,
\]
and so it follows that the total variation coming from intervals of type (ii) is asymptotically zero. In particular we then have
\beq
\lim_{k\to\infty}N_k\delta_k^3= 0.\label{keyest}
\eeq

Now fix any interval $(a_j^k,b_j^k)$ on which $u_1$ satisfies scenario (ii) and for example suppose $u_1(b_j^k)=+\delta_k$. Then letting $c_j^k$ denote the midpoint of this interval, we define the first component of $w^{k,\e}$ via the formula
\[
w^{k,\e}_1(y)=\left\{\begin{matrix}
\delta_k\tanh\bigg(\delta_k\,\frac{(y-c_j^k)}{\e}\bigg)&\quad\mbox{for}\;\abs{y-c_j^k}<\e^{5/6}\\
\mbox{linear interpolation}&\quad\mbox{for}\;\e^{5/6}\leq\abs{y-c_j^k}\leq 2\e^{5/6}\\
-\delta_k&\quad\mbox{for}\;a_j^k\leq y<c_j^k-2\e^{5/6}\\
\delta_k&\quad\mbox{for}\;c_j^k+2\e^{5/6}<y\leq b_j^k
\end{matrix}\right.
\]
We are using here the finiteness of $N_k$ to allow that for $\e$ sufficiently small, such a construction fits inside the interval $(a_j^k,b_j^k)$.

As in the derivation of \eqref{interpder} and \eqref{nopot} we find that the contribution to $\oned$ in the interpolation zone is asymptotically zero and analogous to
\eqref{stepb}, one has
\[
\lim_{\e\to 0} \frac{1}{2}\int_{\{y:\,\abs{u_2(y)}> \sqrt{1-\delta_k^2}\}} \e |(w^{k,\e})^\prime|^2 + \frac{1}{\e} (|w^{k,\e}|^2 - 1)^2 + L ( u_2^\prime)^2 \,dy=\sum_1^{N_k}\frac{4}{3}\Big( 1 - w^{k,\e}_2(c_j^k)^2\Big)^{3/2}=\frac{4}{3}N_k\delta_k^3
\]
through the use of \eqref{w2const}. Hence, invoking \eqref{keyest} we see that in taking the double limit $\lim_{k\to\infty}\lim_{\e\to 0}$, there is no contribution to the energy $\oned(w^{k,\e})$ coming from the region $\left\{y:\,\abs{u_2(y)}> \sqrt{1-\delta_k^2}\right\}$.

Finally taking the limit $\delta_k\to 0$ in \eqref{stepb}, we obtain \eqref{finally}.
\end{proof}

\medskip
\brk
The proof of the $\Gamma-$convergence result in the one-dimensional setting as shown above echoes a number of the salient features of the proof of the full two-dimensional result presented in Theorem \ref{main} while suppressing other technicalities. In the proof of the lower bound, the quantity $\psi_\e$ introduced in the proof of Theorem \ref{thm1Dmain} is precisely the first component of the vector field $\Xi(u_\e^\perp)^\perp,$ where $\Xi$ is the Jin-Kohn vector field introduced in the proof of Theorem \ref{main}. 

Comparing the two-dimensional and one-dimensional constructions of recovery sequences, we first note that our one-dimensional recovery sequence construction proceeded as a double limit, similar to, though of course, simpler than the quadruple limit used in proof of Theorem \ref{main}, cf. equation \eqref{eq:domaindecomp} for instance. Aside from the $\e-$limit which is of course common to both proofs, the jumps that are at least $\delta_k,$ as in equation \eqref{gH1} above represent the analog of ``good jump points'' in the two-dimensional proof, cf. \eqref{g1}. Finally, the intermediate scales $k\e$ used in the proof of the Theorem \ref{main} are now replaced by appropriate powers of $\e.$ 
\erk

\medskip
\brk \label{rmknotbv}
We recall that in Theorem \ref{main} we made the assumption $u \in BV.$ That this is not quite the optimal space can already be seen in this simpler one-dimensional setting where one can construct a limiting vector field $u = (u_1,u_2)$ with $u_1$ having a countable collection of jumps of size $(\frac{1}{k})_{k \in \mathbb{N}}.$ Such a construction can be arranged to have finite $E_0$ energy, but necessarily has infinite $BV$ norm. The preceding theorem, however, guarantees the existence of a recovery sequence for such a competitor. 

This phenomenon is well-known for Aviles-Giga, see the discussion on \cite[pg. 338-340]{ADM}. The counter-example there is very similar in spirit, but is understandably a bit more involved due to the constraint imposed by the eikonal equation. 
\erk

Next we pursue an understanding of minimizers of the one-dimensional $\Gamma$-limit $\onedlim$.

\medskip
\bthm \label{1dmain} For any $a\in (0,1)$ the problem
\[
\inf_{\sA^0} \onedlim({ u})
\]
has a unique solution ${ u^*}=(u^*_1,u^*_2)$ where $u^*_1$ has exactly one jump located at $y=0$ and $u^*_2$ is linear on the subintervals $[-H,0]$ and $[0,H]$. More precisely, the components are given by the formulas:
\beq
u^*_2(y)=\left\{
\begin{array}{ll}
a+\frac{M-a}{H}(y+H),  &  y\in(-H,0],   \\
a+\frac{M-a}{H}(H-y), & y\in (0,H),    
\end{array}
\right.\label{bestcomp2}
\eeq
\beq  u^*_1(y)=
\left\{
\begin{array}{ll}
-\sqrt{1- (u^*_2)^2} & \mbox{for}\; y\in[-H,0],   \\
  \sqrt{1-(u^*_2)^2} &  \mbox{for}\; y\in(0,H],
\end{array}
\right.\label{bestcomp1}
\eeq
where the constant $M=M(L,H,a)\in (a,1)$ is the minimizer of the problem
\beq
\min_{m\in [-1,1]}\frac{L}{H}(m-a)^2+\frac{4}{3}(1-m^2)^{3/2}.\label{bestm}
\eeq
In case $a=0$, the nature of the minimizer depends on the ratio $L/H$. If $L/H< 2$, then the minimizer is again unique and has the one-jump structure given by  \eqref{bestcomp2}-\eqref{bestcomp1} and the infimum is $\frac{L}{H}-\frac{1}{12}\frac{L^3}{H^3}$. If $L/H> 2$ then the minimizer is any step function of the form
\[
{ u}(y)=\left\{
\begin{array}{ll}
(-1,0) &\; \mbox{for}\; y\in(-H,y^*],   \\
 (1,0) &  \;\mbox{for}\; y\in(y^*,H),
\end{array}
\right.
\]
where $y*\in [-H,H]$ is arbitrary and the infimum is $4/3$. If $L/H=2$ the family of step functions and the solution given by \eqref{bestcomp2}-\eqref{bestcomp1} are all minimizers.
\ethm
\begin{proof}
 Let ${ u}=(u_1,u_2)$ be any competitor in $\sA^0$. We denote by $J_{ u}$ the jump set of ${u},$ which in the present one-dimensional setting corresponds simply to the jump set of $u_1$, combined with either $-H$ or $H$ or both if either 
 $u_1(-H)\not=-\sqrt{1-a^2}$ or $u_1(H)\not=\sqrt{1-a^2}$. We will write $\bar{J}_{ u}$ for the closure of $J_{ u}$ and
we define the number $M_{ u}$ via
\[
M_{ u}:=\left\{
\begin{array}{ll}
\max_{y\in\bar{J}_{ u}}u_2(y) & \mbox{if}\; \bar{J}_{{ u}}\not=\emptyset,   \\
\max_{y\in [-H,H]}u_2(y) &  \mbox{if}\; \bar{J}_{{ u}}=\emptyset. 
\end{array}
\right.
\]
In light of the continuity of $u_2$ and the compactness of $\bar{J}_{ u}$ we note that this maximum will always be achieved at
at least one point $\bar y\in [-H,H]$. We now proceed in three cases.\\
\noindent {\bf Case 1.} $\bar{J}_{ u}\not=\emptyset$ and $M_{ u}$ is achieved at $\bar y\in (-H,H).$\\
We note that this case includes the possibility that $\bar y\not\in J_{ u}$ but is simply a limit point of a sequence of points $\{y_j\}$ in the jump set.
In this case $\abs{{ u}_-(y_j)-{ u}_+(y_j)}\to 0$, meaning that the difference between the left and right traces of $u_1$ approaches zero.
Since these traces are also opposites of each other, necessarily $u_1(\bar y)=0$, forcing $u_2(\bar y)=1=M_{ u}$.

Whether or not this subcase of Case 1 occurs, we now consider the competitor ${ \bar{u}}=(\bar{u}_1,\bar{u}_2)$ whose second component is given by 
\beq\bar u_2=
\left\{
\begin{array}{ll}
a+\frac{M_{{ u}}-a}{\bar y+H}(y+H) & \mbox{for}\; y\in[-H,\bar y],   \\
a+\frac{M_{{ u}}-a}{\bar y-H}(y-H) &  \mbox{for}\; y\in(\bar y,H],    
\end{array}
\right.\label{comp2}
\eeq
and whose first component is given by 
\beq \bar u_1=
\left\{
\begin{array}{ll}
-\sqrt{1-{\bar u}_2^2} & \mbox{for}\; y\in[-H,\bar y],   \\
  \sqrt{1-{\bar u}_2^2} &  \mbox{for}\; y\in(\bar y,H].
\end{array}
\right.\label{comp1}
\eeq
We calculate that
\begin{multline}
\onedlim({ u})\geq\frac{L}{2}\int_{-H}^{\bar y} (u_2^\prime)^2 \,dy+\frac{L}{2}\int_{\bar y}^H (u_2^\prime)^2 \,dy+\frac{4}{3}
(1-M_{ u}^2)^{3/2}\\
\geq \frac{L}{2(\bar y+H)}\left(\int_{-H}^{\bar y}u_2^\prime\,dy\right)^2+\frac{L}{2(H-\bar y)}\left(\int_{\bar y}^H u_2^\prime\,dy\right)^2
+\frac{4}{3}
(1-M_{ u}^2)^{3/2}\\
 = \frac{L(M_{{ u}}-a)^2}{2(\bar y+H)}+\frac{L(M_{{ u}}-a)^2}{2(H-\bar y)}+\frac{4}{3}
(1-M_{ u}^2)^{3/2}\\=\frac{L}{2}\int_{-H}^{\bar y} ({\bar u}_2^\prime)^2 \,dy+\frac{L}{2}\int_{\bar y}^H ({\bar u}_2^\prime)^2 \,dy
+\frac{4}{3}
(1-{\bar u}_2({\bar y})^2)^{3/2}
=\onedlim (\bar u),
\label{CS}
\end{multline}
by the Cauchy-Schwarz inequality, with the inequality being strict unless $u_2$ is linear on the subintervals $(-H,\bar y)$ and $(\bar y,H)$.
Furthermore, among competitors of the form \eqref{comp2}-\eqref{comp1}, the second to last line of \eqref{CS} reveals that the optimal choice is to have $\bar y=0$ yielding a minimal energy within this class of competitors of the form
\beq
\onedlim(\bar u)=\frac{L}{H}(M_{ u}-a)^2+\frac{4}{3}(1-M_{ u}^2)^{3/2}.\label{bestjump}
\eeq
\\
\noindent
{\bf Case 2.} Suppose $\bar{J}_{ u}=\emptyset.$\\
In this case $u_1$ is continuous with $u_1(\pm H)=\pm\sqrt{1-a^2}.$ Hence there exists a point $ y\in (-H,H)$ such that $u_1(y)=0$,
meaning that $u_2(y)=1$. Therefore in this case, $M_{ u}=u_2(\bar y)=1$ for some $\bar y\in (-H,H).$ 
 Then consider the competitor ${ \bar{u}}=(\bar{u}_1,\bar{u}_2)$ given by \eqref{comp2}--\eqref{comp1} with $M_{ u}=1$ so that
 now $u_1$ is continuous as well. The calculation leading to \eqref{CS}, absent the jump term, implies in this case that
 \[
\onedlim({ u})\geq \onedlim( {\bar{ u}})
\]
with the minimal value 
\beq
\onedlim({ \bar{u}})=\frac{L}{H}(1-a)^2.\label{maxofone}
\eeq
\noindent {\bf Case 3.} Suppose $\bar{J}_{ u}\not=\emptyset$ and either $M_{ u}=u_2(-H)$ or $M_{ u}=u_2(H)$.\\
 In the first case, we have
 \beq\sqrt{1-a^2}
 \onedlim({ u})\geq \frac{1}{6}(u_1(-H)+\sqrt{1-a^2})^3=\frac{4}{3}(1-a^2)^{3/2}=\onedlim({\bar u})\label{left}
 \eeq
 where ${\bar u}\equiv (\sqrt{1-a^2},a)$, while in the second case we have
  \beq
 \onedlim({ u})\geq \frac{1}{6}(u_1(H)-\sqrt{1-a^2})^3=\frac{4}{3}(1-a^2)^{3/2}=\onedlim({\bar u})\label{right}
 \eeq
 where ${\bar u}\equiv (-\sqrt{1-a^2},a)$. Again the inequalities are sharp unless ${ u}\equiv { \bar{u}}.$
 
 Having exhausted all possibilities, we next observe that the optimal formula \eqref{maxofone} from Case 2 corresponds to \eqref{bestjump} with $M_{ u}=1$ and
 the optimal formulas \eqref{left} and \eqref{right} from Case 3 correspond to \eqref{bestjump} with $M_{ u}=a$. Hence, the minimal energy corresponds to the minimization \eqref{bestm}. Clearly this minimum must occur for $m\in [0,1]$ and since for $a\in (0,1)$, the function
 \[
 f(m):=\frac{L}{H}(m-a)^2+\frac{4}{3}(1-m^2)^{3/2}
 \]
 satisfies the conditions $f'(0)<0$ and $f'(1)>0$, the minimum occurs on $(0,1)$. The conclusion of the theorem for this case then follows. When $a=0$ one finds that $f'(0)=0$ and some elementary calculus yields the stated dichotomy depending on the ratio $L/H.$ When $a \neq 0,$ it can be checked by elementary arguments that the interior minimum is unique. 
\end{proof}
\brk \label{rkcrosstie}
The proof of Theorem \ref{main} reveals that resolving the internal structure of walls for the $E_0$ energy at the $\e > 0$ level using a one-dimensional construction is asymptotically optimal. However, it is possible to also have two dimensional recovery sequences with the same energy asymptotics. To see this, set $S := \{|x| < 1/2\}$ and define the map $u:\R^2 \to \R^2$ which is $1-$ periodic in the $x-$direction by
\begin{align*}
u(x,y) = u(r \cos \theta, r \sin \theta) := \left\{
\begin{array}{cc}
\big( \frac{1}{\sqrt{2}}, - \frac{1}{\sqrt{2}}\big) & S \cap \{ 0 \leqslant \theta \leqslant \frac{\pi}{4} \} \\
(\sin \theta, - \cos \theta) & S \cap \{\frac{\pi}{4} \leqslant \theta \leqslant \frac{3\pi}{4}\} \\
\big(\frac{1}{\sqrt{2}}, \frac{1}{\sqrt{2}} \big) & S \cap \{ \frac{3\pi}{4} \leqslant \theta \leqslant \pi\}  \\
\big( - \frac{1}{\sqrt{2}}, \frac{1}{\sqrt{2}} \big) & S \cap \{ \pi \leqslant \theta \leqslant \frac{5\pi}{4}\} \\
(\sin \theta, - \cos \theta) & S \cap \{ \frac{5 \pi}{4} \leqslant \theta \leqslant \frac{7\pi}{4} \} \\
\big( - \frac{1}{\sqrt{2}} , - \frac{1}{\sqrt{2}} \big) & S \cap \{ \frac{7\pi}{4} \leqslant \theta < 2 \pi\}.
\end{array}
\right. 
\end{align*}
 extended to all of $\R$ by $u(x+1,y ) = u(x,y)$ for all $x \in \R.$ We compute the $E_0$ energy per unit length of the cross-tie map $u$, which is divergence free. Across the walls $\{|x| \leqslant 1/2, y = 0\},$ the jump angle is $\pi/4.$ Similarly, along the walls $\{|y| \leq 1/2, x = 1/2\}$ the jump angle is $\pi/4.$ Finally there are walls $\{|y| > 1/2, x = 1/2\},$ along which the angle varies with $y,$ and is in fact equal to $\arctan \big( \frac{1}{2y} \big)$ at height $y.$ Adding up these various jump energies yields the energy per unit length,  
\begin{align*}
E_0(u;S)&= \frac{4}{3} \left[2 \left( \frac{1}{\sqrt{2}}\right)^3 + 2\int_{1/2}^\infty \frac{1}{(1 + 4y^2)^{3/2}} \,dy \right] \\
&= \frac{4}{3}\left[ \frac{1}{\sqrt{2}} + 1 - \frac{1}{\sqrt{2}}\right] = \frac{4}{3}. 
\end{align*}
This construction can be blown down to fit into walls replacing a heteroclinic connecting $(1,0)$ and $(-1,0)$.
This observation is reported without details given in \cite{KohnICM} based on private communication with S. Serfaty. 
\erk
\subsection{A two-dimensional construction with cross-ties}
In this section we construct a critical point to $E_0$ by solving the free boundary problem \eqref{el:5}-\eqref{el:13}. Here our particular interest is to find parameter regimes within which the one-dimensional minimizer from Theorem \ref{1dmain} fails to minimize the full two-dimensional problem \eqref{Ezero}.  
%
The main result of this section is 
\bthm \label{not1d}
Consider the minimization problem for $E_0$ in the rectangle $\Omega = (-T,T) \times (-H, H)$, subject to the boundary conditions \eqref{rectbc} with $a = 0.$  There exist constants $L_0 \approx 1.27$ and $L_1 \approx 2.14$ such that whenever $L/H \in (L_0,L_1)$ and $T = H\tilde T(L/H)$ where $\tilde T(L/H)$ solves \eqref{eq:TDLN}, we have
\beq
\inf E_0(u)<2T\,\inf_{\sA^0} E_0^{1D}(u).\label{2beats1}
\eeq
Here the infimum on the left is taken over all $u\in \hdivS \cap BV(\Omega;\mathbb{S}^1)$ such that $u\cdot \nu=0$ on the top and bottom $y=\pm H$ and $u$ is $2T$-periodic in $x$.
\ethm
\begin{rmrk} \label{scaleinvariance}
Given the energy functional \eqref{Ezero} and the rectangular domain $\Omega$ in the statement of Theorem \ref{not1d}, it is easy to see that by setting
\[\tilde x=\frac{x}{H},\ \tilde y=\frac{y}{H},\ \tilde E_0=\frac{E_0}{H},\]
the rescaled variational problem for $\tilde E_0$ contains two independent parameters: the aspect ratio $\tilde T=T/H$ and the scaled elastic constant $L/H$. Then setting $\tilde u(\tilde x,\tilde y)=u\left(H\tilde x,H\tilde y\right)$ for any admissible $u\in \hdivS \cap BV(\Omega;\mathbb{S}^1)$, assuming that $\tilde T=\tilde T(L/H)$ and writing explicitly the dependence of the energy on $L$ and $H$, we find that
\begin{equation}
\label{eq:scales}
\frac{1}{2T}E_0(u,L,H)=\frac{1}{2\tilde T}\tilde E_0(\tilde u,L/H).
\end{equation}
In other words, the energy per unit length along the $x$-axis is a function of the scaled elastic constant $L/H$ only.
\end{rmrk}
The proof of Theorem \ref{not1d} relies on a construction of a two-dimensional critical point of $E_0$ that resembles cross-tie walls well-known in the studies of micromagnetics (\cite{Ignat,ARS}; see also Remark \ref{rkcrosstie}). Our construction is motivated by the numerics which we will now describe. 

To find two-dimensional critical points of $E_0,$ we used the finite elements software COMSOL\textsuperscript{\textregistered} \cite{COMSOL} to determine the solutions of the Euler-Lagrange equation for $E_\varepsilon$ numerically. Here the (local) minimizers were found by simulating the gradient flow for $E_\varepsilon$ on time intervals that were sufficiently large for a solution to reach an equilibrium. 

In our numerics, we fixed $H=1/2$ and allowed $L$ to vary. Then, for a given $L > 0$, we determined $T=\tilde T(2L)/2$ by solving the equation \eqref{eq:TDLN}. The reason for this choice of $T$ will be explained below. The Euler-Lagrange equation for $E_\varepsilon$ was then solved on the rectangle $(-T,T)\times(-1/2,1/2)$, subject to periodic boundary conditions on $\left\{-T,T\right\}\times[-1/2,1/2]$ and assuming that $u(\cdot,\pm 1/2)=(\pm 1,0)$. 

Our numerical studies allowed us to identify three different regimes. When $L$ is small, the one-dimensional solution (not shown) is recovered as the result of simulations. For intermediate values of $L$, a single-wall cross-tie configuration appears (Figs.~\ref{f:ct-1}-\ref{f:ct-3}). An analytical solution corresponding to this configuration will be constructed below using the conservation laws approach of Corollary \ref{conservation}. To this end, we observe that this configuration (i) has both vertical and horizontal jump sets coinciding with the coordinate axes as well as a pair of defects of degrees $\pm1$ at $(0,T)$ and $(0,0)$, respectively; (ii) the solution is symmetric with respect to reflections about the coordinate axes and the divergence is antisymmetric with respect to these reflections; and (iii) the level curves for divergence in the first quadrant can be distinguished into three different regions as in Fig.~\ref{fig:BLAH}.

We conjecture that this configuration corresponds to the cross-tie construction that we develop in this section. Indeed, when the solution resulting from this construction is plotted (Figs.~\ref{f:ctm-1}-\ref{f:ctm-2}), it closely resembles those in Figs.~\ref{f:ct-2}-\ref{f:ct-3}. 
\begin{figure}[H]
\centering
    \includegraphics[width=3in]
                    {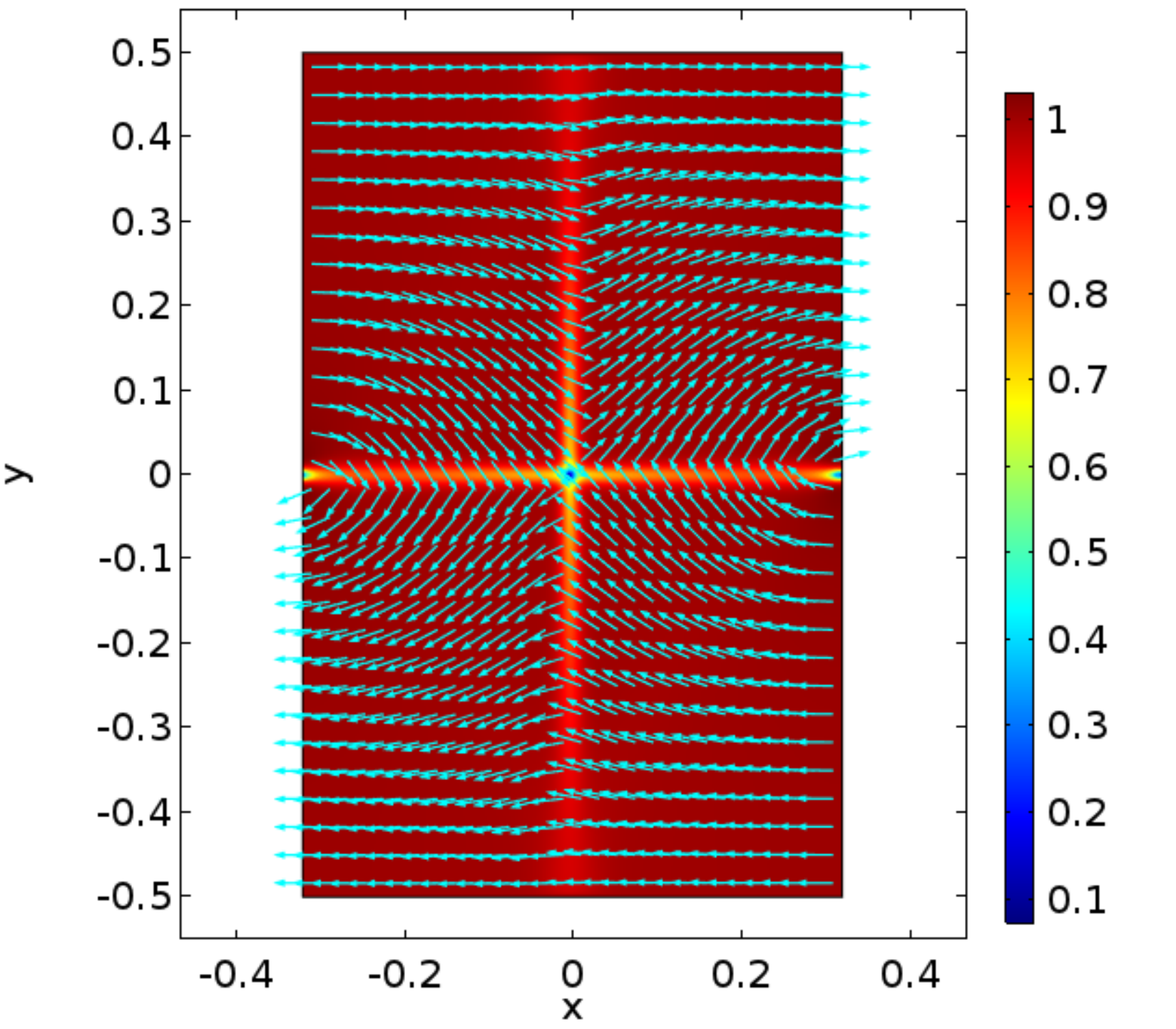}
    \caption{A solution $u$ of the Euler-Lagrange equation associated with the energy functional \eqref{energy} in the rectangle $[-T,T]\times[-1/2,1/2]$ subject to periodic boundary conditions on $\left\{-T,T\right\}\times[-1/2,1/2]$ and assuming that $u(\cdot,\pm 1/2)=(\pm 1,0)$. Here $L=1/2$ and $T=\tilde T(1)/2 \approx 0.3$. Both $u$ and $|u|$ are shown.}
  \label{f:ct-1}
\end{figure}
\begin{figure}[H]
\centering
    \includegraphics[width=3in]
                    {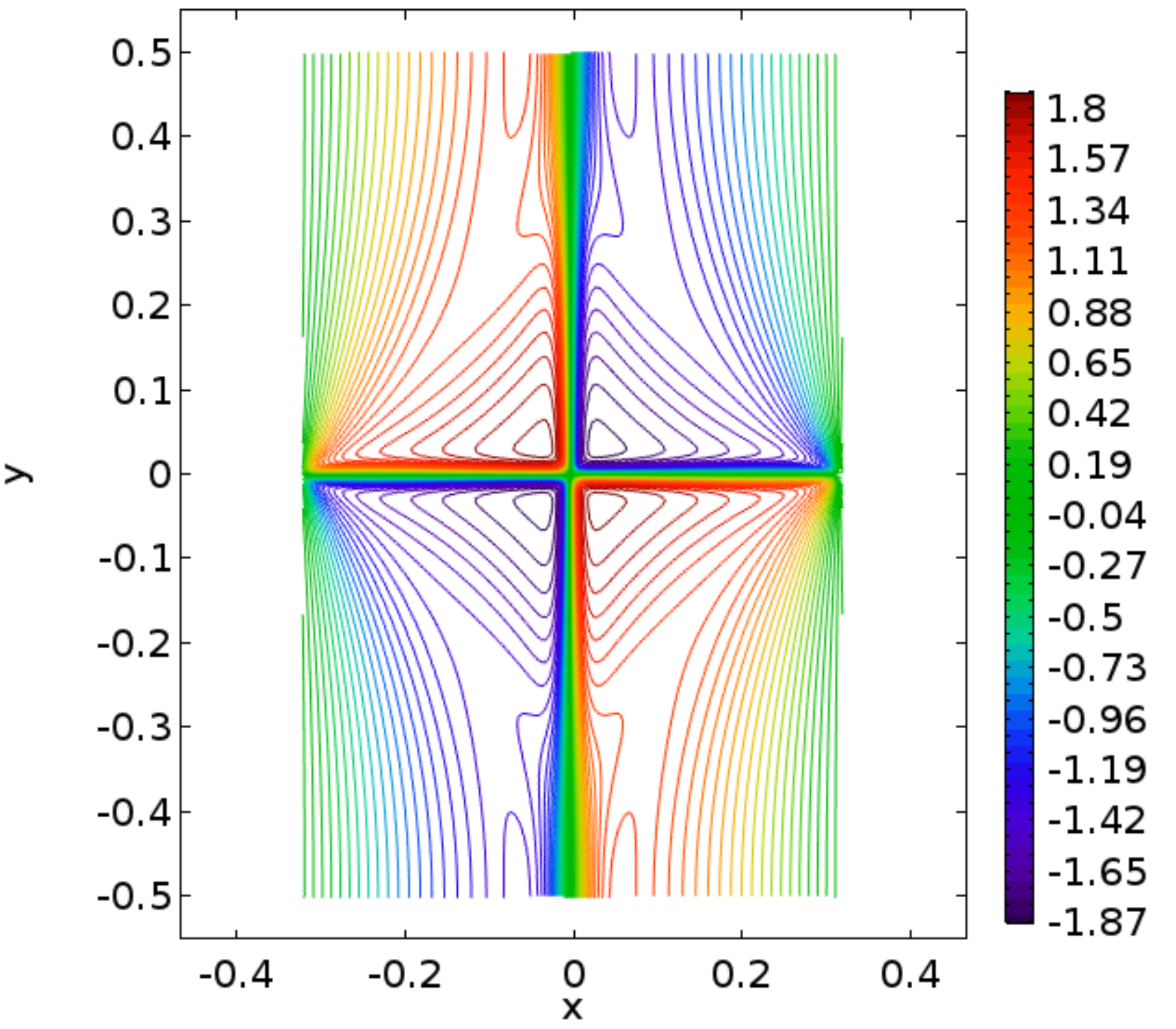}
    \caption{Level curves for the divergence of $u$, where $u$ is depicted in Fig.~\ref{f:ct-1}.}
  \label{f:ct-2}
\end{figure}
\begin{figure}[H]
\centering
    \includegraphics[width=3in]
                    {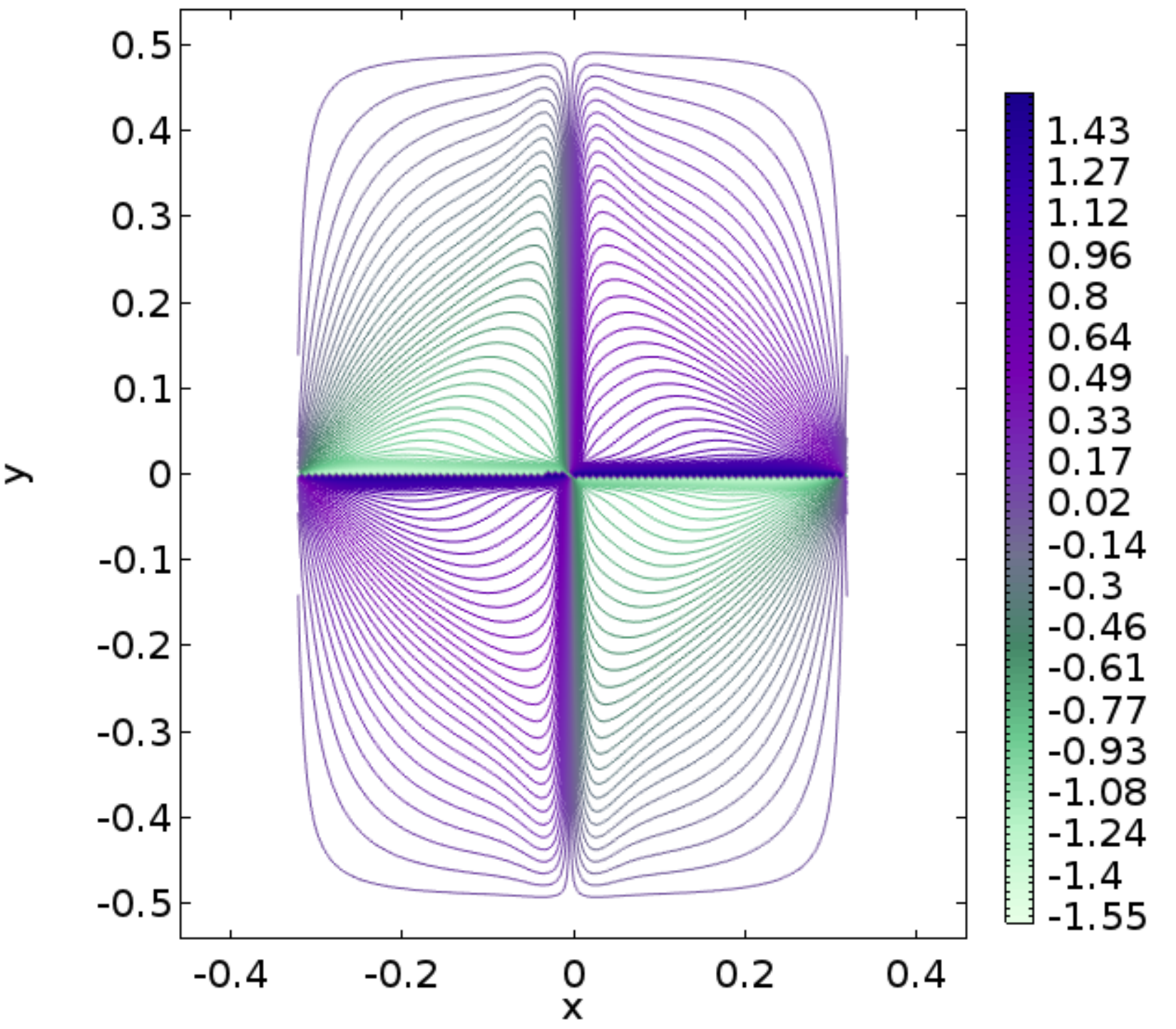}
    \caption{Level curves for the angle $\theta$, where $u=(\cos{\theta}, \sin{\theta})$ is depicted in Fig.~\ref{f:ct-1}.}
  \label{f:ct-3}
\end{figure}
Before proceeding with the analytical construction of a cross-tie configuration resembling Fig. \ref{f:ct-1}, we continue with further remarks about our $E_\e$ numerics for larger values of $L.$ When $L$ is increased further, it appears that $2T=\tilde T(2L)$ as determined by \eqref{eq:TDLN} is no longer the period of the optimal construction as two cross-tie structures appear on the interval $[-T,T]$ in Figs.~\ref{f:ctt-1}-\ref{f:ctt-3}. We call this a type II cross-tie configuration. A close examination of Fig.~\ref{f:ctt-2} shows that the level curves for divergence that originate on the $y$-axis appear to terminate on the line $y=1/2$, as opposed to those in Figs.~\ref{f:ct-1}-\ref{f:ct-3}. Pursuing an analytical construction of this solution is beyond the scope of the present paper. However,  it follows that we can identify at least three families of critical points that may minimize the limiting energy functional $E_0$ for different values of $L$. 
\begin{figure}[H]
\centering
    \includegraphics[width=3in]
                    {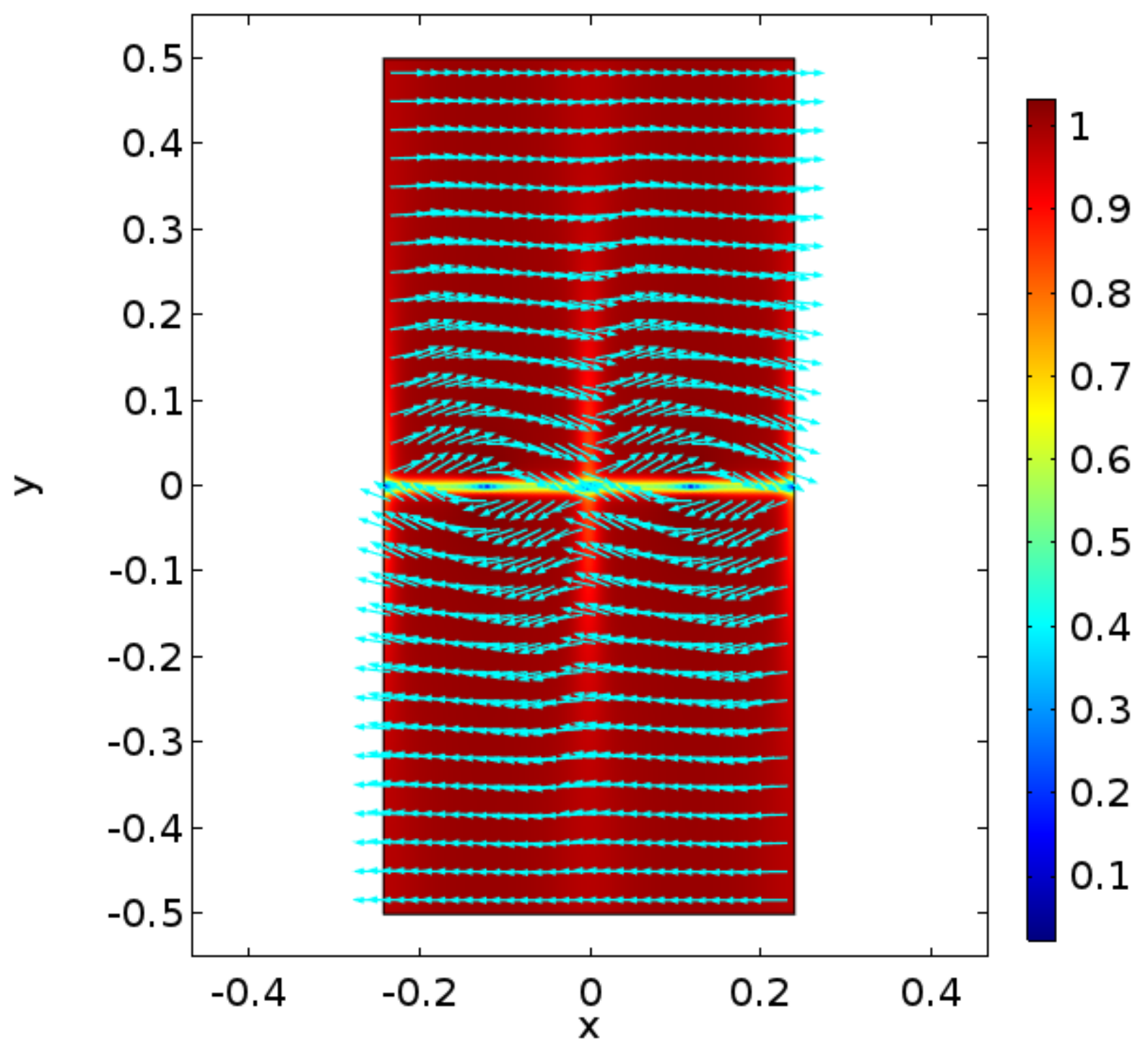}
    \caption{A solution $u$ of the Euler-Lagrange equation associated with the energy functional \eqref{energy} in the rectangle $[-T,T]\times[-1/2,1/2]$ subject to periodic boundary conditions on $\left\{-T,T\right\}\times[-1/2,1/2]$ and assuming that $u(\cdot,\pm 1/2)=(\pm 1,0)$. Here $L=3/2$ and $T=\tilde T(3)/2 \approx 0.25$. Both $u$ and $|u|$ are shown.}
  \label{f:ctt-1}
\end{figure}
\begin{figure}[H]
\centering
    \includegraphics[width=3in]
                    {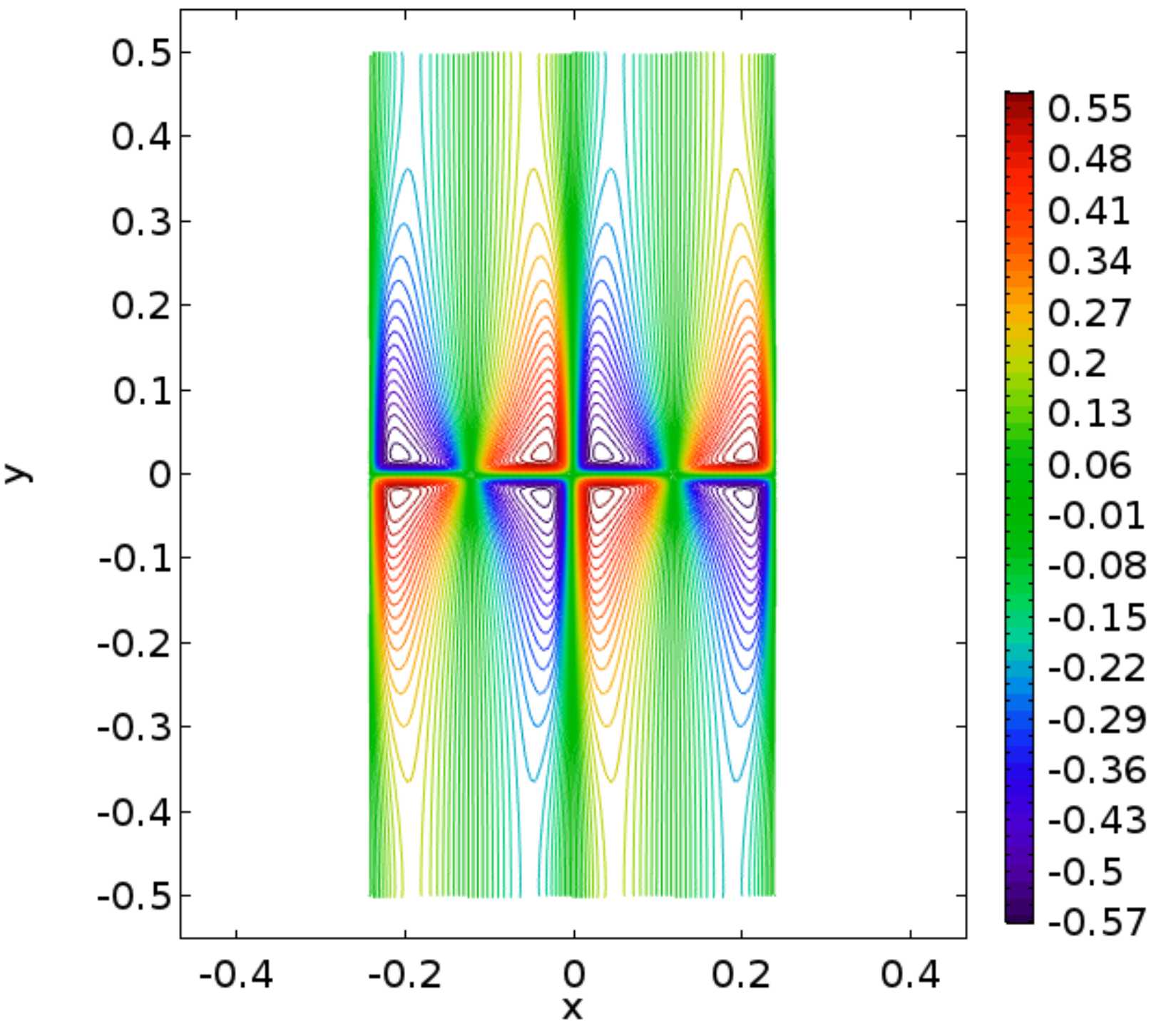}
    \caption{Level curves for the divergence of $u$, where $u$ is a type II cross-tie depicted in Fig.~\ref{f:ctt-1}.}
  \label{f:ctt-2}
\end{figure}
\begin{figure}[H]
\centering
    \includegraphics[width=3in]
                    {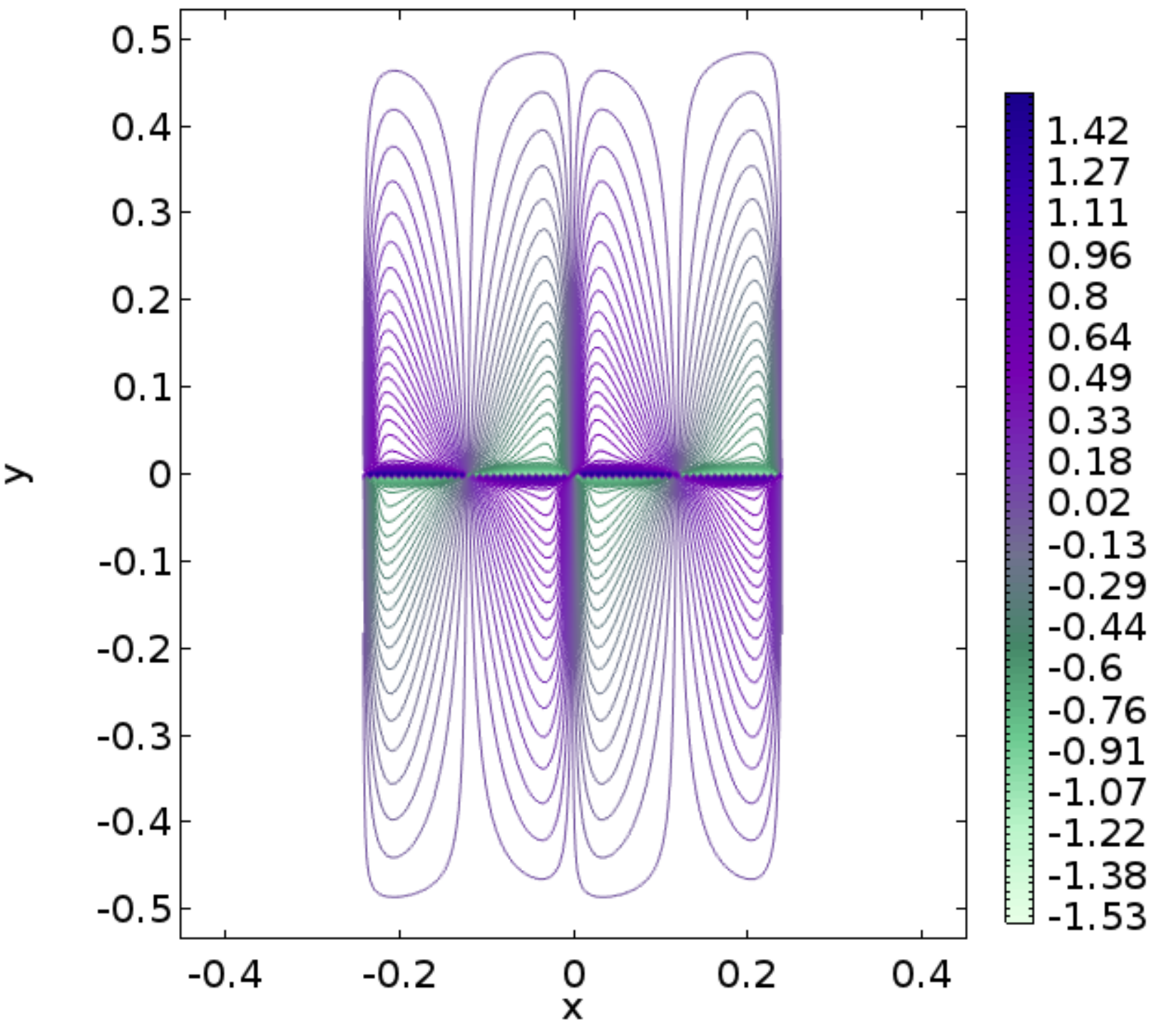}
    \caption{Level curves for the angle $\theta$, where $u=(\cos{\theta}, \sin{\theta})$ is depicted in Fig.~\ref{f:ctt-1}.}
  \label{f:ctt-3}
\end{figure}

\subsection*{Analytical construction of a cross-tie configuration and the proof of Theorem \ref{not1d}}
We now use the observations made concerning the numerics of a single cross-tie to construct a critical point of $E_0.$ Although numerics were carried out fixing $H = 1/2,$ we will carry out our construction for any $H,$ and work on a single period cell $\Omega = (0,2T) \times (-H,H).$ We will further assume $T < H,$ a choice that is consistent with the choice made in the numerics. 

A single period cell of this solution is composed of a dipole, i.e. a pair of $+1$ and $-1$ vortices along with walls connecting them. The observations (i)-(iii) above from the numerics motivates us to construct the critical point $u = (\cos \theta, \sin \theta)$ on a quarter of the period cell, say $\omega := (0,T) \times (0,H)$ and define $u$ elsewhere by appropriate reflections.  
\begin{figure}[H]
\centering
    \includegraphics[width=1.5in]
                    {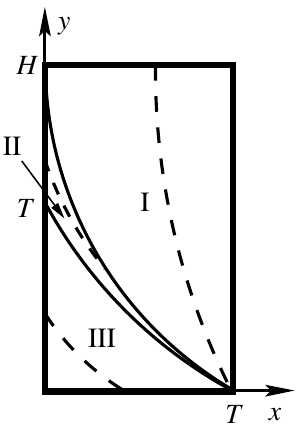}
    \caption{Regions corresponding to different characteristics families. Typical characteristics for each region are indicated by dashed lines.}
  \label{fig:BLAH}
\end{figure}
A quarter period cell is displayed in Figure \ref{fig:BLAH}. By comparison with Fig. \ref{f:ct-1}, the line $x = 0$, $0 \leqslant y \leqslant H$ denotes a vertical wall and the $x-$axis denotes a horizontal wall. Upon reflection and periodic extension, the point $(0,0)$ is to house a degree $-1$ vortex while at the point $(T,0)$ we will have constructed a $+1$ vortex resembling the $\widehat{e}_\theta$ vector. 

We construct solutions to the system of conservation laws given in Corollary \ref{conservation} using the method of characteristics. 
Within the quarter period cell $\omega$ under consideration, we seek $ u = (\cos \theta, \sin \theta)$ with $\theta \in [0, \frac{\pi}{2}].$ We impose Dirichlet boundary conditions $\theta = 0$ along the top and right boundaries of $\omega.$ The condition on the right boundary is a result of the symmetry observation (ii) above. The natural boundary condition \eqref{el:4} is to be satisfied along the left boundary and the $x-$axis since these represent walls. 

Building on observation (iii), the characteristics solution in $\omega$ consists of three families of circular arcs, labeled regions $I$ through $III$ in Fig.~\ref{fig:BLAH} and described in Steps 1-3 below. In each of the regions $I,II$ and $III$ above, we will denote the arc-length and characteristic variables by $s_1,s_2, s_3$  and $t_1,t_2,t_3$ respectively. The dependent variables, $x = x(s,t), y = y(s,t), \theta = \theta(s,t) $ and $v = v(s,t)$ will also be denoted using appropriate subscripts. 

\noindent {\bf 1.} In this step, we construct characteristics foliating Region I in Fig.~\ref{fig:BLAH}. First, starting from the top boundary $\{(s_1,H): 0 \leqslant s_1 \leqslant T\},$ we issue characteristics that meet at the point $(T,0).$ Indeed, along the top boundary, we have the boundary condition $\theta_1 = 0,$ since $a = 0.$  By the characteristic equations, characteristics emanating from $(s_1,H)$ for any $0 \leqslant s_1 \leqslant T$ leave the top boundary orthogonally. For such fixed $s_1,$ there is a unique circle orthogonal to the top boundary at $(s_1,H)$ that passes through the point $(T,0).$ A geometric argument shows that the center of this circle is given by $\left( \frac{T + s_1}{2} + \frac{H^2}{2(T-s_1)} , H\right),$ while the radius is given by 
\begin{align} \label{eq:radius}
R(s_1) = \frac{T-s_1}{2} + \frac{H^2}{2(T-s_1)}. 
\end{align}
It follows that $R(s_1) \geqslant H.$ Integrating the characteristics starting at the top boundary, the circles constructed are characteristics, and along the circle starting at $(s_1,H),$ we have $v_1(s_1,t_1) \equiv v_1(s_1) := \frac{1}{R(s_1)},$ and $\theta_1(s_1,t_1) = v_1(s_1)t_1. $ 

It is clear that the foregoing yields characteristics that only meet at $(T,0)$ and not before. Furthermore, the right boundary of $\omega,$ along which $\theta_1 = 0,$ is itself a characteristic, and belongs to the above family corresponding to infinite radius, as can be observed by setting $s_1 = T$ in \eqref{eq:radius}. Furthermore, it is clear that the divergence is bounded for this family, i.e. $v_1(s_1) \in [0,H].$ 

The characteristic emanating out of $(0,H)$ satisfies $v_1(0) = \frac{1}{R(0)} = \frac{2T}{T^2 + H^2}=:\alpha.$ For later use, we record the equation of this characteristic as being given by 
\begin{align}
x_1(t_1) = - \frac{1}{\alpha} \left( \cos\big( \alpha t_1\big)  - 1\right), \hspace{1cm} 
y_1(t_1) = H-\frac{1}{\alpha} \sin \big(\alpha t_1\big). 
\end{align}
We will refer to this characteristic as the terminal characteristic of the first family and denote it by $\Gamma.$
If we let $t_1^*$ denote the time of arrival of this characteristic at $(T,0),$ then we have the relation 
\begin{align}
\label{eq:relation}
\frac{H}{\sin \big( \alpha t_1^* \big)} = \frac{T}{1 - \cos \big(\alpha t_1^*\big)},
\end{align}
that we can also write as
\begin{equation}
\label{eq:ldog}
\tan{\left(\frac{\alpha t_1^\ast}{2}\right)}=\frac{T}{H}.
\end{equation}

\noindent {\bf 2.} In this step, we construct a family of characteristics that foliate region III of Fig~\ref{fig:BLAH}. This family of characteristics consists of circular arcs emanating off of $(s_3,0)$ and terminating on the vertical wall at $(0,y_3(t_3^*(s_3))),$ for $s_3 \in (0,T).$  The symmetry assumptions from observations (ii)-(iii) along with \eqref{el:4} yield $Lv_3 + \sin 2 \theta_3 = 0$ along both the left and bottom walls. Since the divergence $v_3$ is constant along characteristics, we find that $\sin 2 \theta_3(s_3,0) = \sin 2 \theta_3(s_3,t_3^*(s_3)),$ yielding 
\begin{align}
\theta_3(s_3,t_3^*(s_3)) = \frac{\pi}{2} - \theta_3(s_3,0). \label{eq:tint}
\end{align}
Writing down the condition that $(x_3(s_3,t_3^*(s_3)), y_3(s_3, t_3^*(s_3))$ lies on the left wall, i.e. $x_3(s_3,t_3^*(s_3)) = 0,$ along with \eqref{el:4} along this wall, yields upon some elementary computations that
\begin{align} \label{eq:thetafrombelow}
\sin 2 \theta_3(s_3,0) = \frac{-1 + (1 + 2 \lambda)^{1/2}}{\lambda}, \hspace{1cm} \lambda = \frac{2s_3^2 }{L^2}. 
\end{align}
It can be checked that the right-hand side of the last equation defining $\sin 2 \theta_3(s_3,0)$ indeed belongs to the interval $(0,1).$ Integrating the characteristic equations, we find that the circular arcs of the family foliating the region $III$ are characteristics along which we have $v_3(s_3,t_3) := v_3(s_3) = - \frac{\sin 2 \theta_3(s_3,0)}{L}. $ It can also be easily checked that the particular characteristic of this family originating at $(T,0)$ satisfies 
\begin{align}
y_3(T,t_3^*(T)) = T.
\label{termarrival}
\end{align}
 We will refer to this characteristic as the terminal characteristic of the family foliating the region $III$. 

For reasons that will be clear in the next step, we require that the terminal characteristic of the families foliating the regions $I$ and  $III$, respectively, are tangent at $(T,0).$ This condition can be rewritten, using equation \eqref{eq:relation} and \eqref{eq:thetafrombelow} as 
\begin{align} \label{eq:TDL}
\frac{L^2}{T^2} \left( \sqrt{1 + 4 \frac{T^2}{L^2}} - 1\right) =  \frac{8TH}{T^2 + H^2} \frac{H^2 - T^2}{H^2 + T^2}. 
\end{align} 
Before continuing, we remark about the relation \eqref{eq:TDL}. The left-hand side is a function of $L/T$ while the right-hand side is a function of $H/T$  alone, which we are assuming to be greater than one. We claim that for any $x := H/T > 1,$ there exists a unique $L/T$ such that \eqref{eq:TDL} holds. Indeed, setting $\zeta = \frac{2x}{x^2 + 1 } \frac{x^2 - 1}{x^2 + 1} < 1,$ and $\Lambda = \frac{4T^2}{L^2},$ we are required to solve 
\begin{align*}
\sqrt{1 + \Lambda} = 1 + \zeta \Lambda. 
\end{align*}
We obtain that $\Lambda = \frac{1 - 2 \zeta}{ \zeta^2},$ which is positive provided $\zeta < 1/2,$ or equivalently, provided $2x (x^2 - 1) < \frac{1}{2}(x^2+1)^2.$ This is clear since 
\begin{align*}
\frac{1}{2} (x^2 +1)^2 - 2x (x^2 - 1) = \left( \frac{1}{\sqrt{2}} (x^2 - 1) - \sqrt{2}x \right)^2. 
\end{align*}

Introducing the rescaling $\tilde T=T/H$, we denote by $\tilde T(L/H)$ the unique solution of
\begin{align} \label{eq:TDLN}
L/H\left( \sqrt{(L/H)^2 + 4 \tilde T^2} - L/H\right) - \frac{8\tilde T^3\left(1 - \tilde T^2\right)}{{\left(\tilde T^2 + 1\right)}^2}=0,
\end{align} 
for a given value of $L/H$. In what follows, we set $T=H\tilde T(L/H)$.

We conclude this part with the following observation. The equation \eqref{eq:ldog} can now be written as
\[\tilde T=\tan{\left(\frac{\alpha t_1^\ast}{2}\right)}\]
and testing \eqref{eq:TDLN} with $\alpha t_1^\ast=\frac{\pi}{4}$ and $\alpha t_1^\ast=\frac\pi2,$ we observe that the left hand side is negative and positive, respectively. By the intermediate value theorem, it then follows that $\alpha t_1^\ast\in\left[\frac\pi4,\frac\pi2\right]$ for all $L>0$. With the help of \eqref{eq:tint} we can now conclude that 
\begin{equation}
\label{eq:smth}
\theta_3(T,t_3^\ast(T))\in\left[0,\frac\pi4\right].
\end{equation}

\medskip
\noindent {\bf 3.} We finally foliate region II by characteristics to define our critical point in this region. Since $H > T$ by assumption, it remains to fill the gap between the circles of the first two families. Briefly: we issue secondary characteristics that emanate from the $s = 0$ characteristic $\Gamma$ of the first family tangentially, to meet the left wall. The divergence $v$ has a jump discontinuity along $\Gamma,$ while the tangential departure of the secondary characteristics from $\Gamma$ renders $\theta$ continuous across $\Gamma.$ 

In more detail, we write the initial curve $\Gamma$ using $s$ as the arclength parameter (cf. Step 1) to get 
\begin{align*}
x_0(s_2) = \frac{1}{\alpha} \big( 1 - \cos (\alpha s_2) \big), \hspace{1cm} y_0(s_2) = H - \frac{1}{\alpha} \sin (\alpha s_2),
\end{align*}
and the initial condition on $\theta$ is given by $\theta_0(s_2) = \alpha s_,$ where $s_2\in[0,t_1^\ast]$ and $t_1^\ast$ is as in \eqref{eq:relation}. We do not set an initial condition on the divergence $v_2,$ but instead determine $v_2$ by enforcing \eqref{el:4} at the left wall. Integrating the characteristic equations, we find 
\begin{gather} 
\label{vfol} v_2(s_2,t_2) = v_2(s_2),\quad
\theta_2(s_2,t_2) = \theta_0(s_2) + v_2(s_2)t_2,\\ \label{xfol}
x_2(s_2,t_2) = \frac{1}{\alpha} \big( 1 - \cos(\alpha s_2) \big) + \frac{1}{v_2(s_2)} \left( \cos \theta_2(s_2,t_2) - \cos \alpha s_2)\right), \\  \nonumber y_2(s_2,t_2) = H - \frac{1}{\alpha} \sin (\alpha s_2) + \frac{1}{v_2(s_2)} \left( \sin \theta_2(s_2,t_2) - \sin \alpha s_2\right).
\end{gather}
Again, defining $t_2^\ast(s_2)$ to be the time of arrival of the characteristic originating at $(x_0(s_2),y_0(s_2))$ to the $y-$axis, we find using \eqref{vfol},\eqref{xfol} and \eqref{el:4}, and denoting $\theta_2^\ast(s_2) := \theta_2(s_2,t_2^\ast(s_2)),$  
\begin{gather} \nonumber
 \frac{1}{\alpha} \big( 1 - \cos(\alpha s_2) \big) + \frac{1}{v_2(s_2)} \left( \cos \theta_2^\ast(s_2) - \cos \alpha s_2\right) = 0 \\ \label{lastthing}
L v_2(s_2) + \sin 2 \theta_2^\ast(s_2) = 0.
\end{gather}
If we define a function $f=f(\beta,s_2)$ via the formula
\begin{gather} \label{f}
f(\beta,s_2) := (1 - \cos (\alpha s_2)) \sin 2\beta- L \alpha (\cos \beta- \cos (\alpha s_2)),
\end{gather}
then substituting the second of the equations in \eqref{lastthing} into the first, we find that $\theta_2^\ast$ must satisfy the condition $f(\theta_2^\ast(s_2),s_2)=0.$
We note that $f(0,s_2)<0$ for any $s_2>0$. Now with an eye towards applying the intermediate value theorem, we define $\beta^\ast=\beta^\ast(s_2)$ via
\[
\sin\beta^\ast=\left\{\begin{matrix} \frac{L\alpha}{2\big(1-\cos(\alpha s_2)\big)}&\;\mbox{if}\;L\alpha\leq 2\big(1-\cos(\alpha s_2)\big)\\
1&\;\mbox{otherwise}
\end{matrix}\right.
\]
One easily checks that $f(\beta^\ast(s_2),s_2)>0$ for $s_2>0.$ Hence, the desired terminal angle $\theta_2^\ast(s_2)$ exists for all $s_2$. 

Furthermore, differentiating \eqref{f}, setting $\beta = \theta_2^\ast(s_2)$, with respect to $s_2,$ we find that 
\begin{align*}
\frac{d \theta_2^\ast}{ds_2} = \frac{( L\alpha-\sin 2 \theta_2^\ast) \alpha \sin (\alpha s_2)}{2 (1 - \cos (\alpha s_2)) \cos 2 \theta_2^\ast + L \alpha \sin \theta_2^\ast}.
\end{align*}
Our goal is to show that $\frac{d \theta_2^\ast}{ds_2} > 0.$ We start first, by showing that the denominator of the fraction defining this derivative is positive. 
For each fixed $s_2,$ the function $D$ prescribed by 
\begin{align*}
D(\sin \beta,s_2) := 2( 1 - \cos (\alpha s_2))(1 - 2 \sin^2 \beta) + L \alpha \sin \beta
\end{align*}
defines a downward facing quadratic in $\sin \beta.$ We note that $D(0,s_2) = 2(1 - \cos(\alpha s_2)) > 0,$ and an easy calculation shows that $D(\sin \beta^\ast,s_2) > 0.$ It follows easily that $D(\sin \theta^\ast(s_2),s_2) > 0,$ which is precisely the denominator of the fraction defining  $\frac{d \theta_2^\ast}{ds_2}.$ 

We must show that the numerator of this fraction is also positive. This is immediate when $L \alpha \geqslant 1,$ and therefore we must provide an argument for when $L\alpha < 1.$ Define the number $\beta_- \in [0,\pi/4]$ using the formula $\sin 2 \beta_- = L \alpha.$ Then, we note that $f(\beta_-,s_2) > 0.$ Therefore, when $L \alpha \leqslant 1,$ we have that in fact $\theta_2^\ast(s_2) < \min ( \beta_-, \beta^\ast) \leqslant \pi/4.$ Consequently, for such $L\alpha$ values, we have that $L \alpha - \sin 2 \theta_2^\ast > L \alpha - \sin 2 \beta_- = 0.$ 

This completes the proof of the claim that $\theta^\ast$ is increasing as a function of $s.$ Combining this fact with the constraint \eqref{eq:smth}, we have
\[\theta_2^\ast(T)=\theta_3(T,t_3^\ast(T))\in\left[0,\frac\pi4\right],\] and hence 
\[\theta_2^\ast(s_2)\in\left[0,\frac\pi4\right]\mbox{ for all }s_2\in[0,T].\] The equation \eqref{lastthing} can now be used to show that that $v_2$ is both negative and decreasing. The proof that the characteristics foliate region II then proceeds exactly as in Lemma \ref{l:inc}, completing the construction of our cross-tie critical point.

Having completed the construction of the critical point $u$ of $E_0$ on all of $\Omega$ by appropriate reflections, towards proving Theorem \ref{not1d}, it remains to compute $E_0(u)$ and compare it with that of the one dimensional minimizer from Theorem \ref{1dmain}. The energies per period for the different competitors are compared in Fig.~\ref{f:cten}. Recall that, by Remark \ref{scaleinvariance}, the energy density per period is a function of the scaled elastic constant $L/H$. The solid and dashed lines in Fig.~\ref{f:cten} represent the energies of one-dimensional and the two-dimensional characteristics cross-tie constructions, respectively. Here the energy of a one-dimensional competitor is given in the statement of Theorem \ref{1dmain} and the energy of the two-dimensional construction is obtained by computing an appropriate Jacobian and numerically integrating in MATLAB \cite{MATLAB} (or by using the co-area formula). Comparing these energies for $L/H\in(L_0,L_1)$, Theorem \ref{not1d} now follows. 

\begin{figure}[H]
\centering
    \includegraphics[width=2.5in]
                    {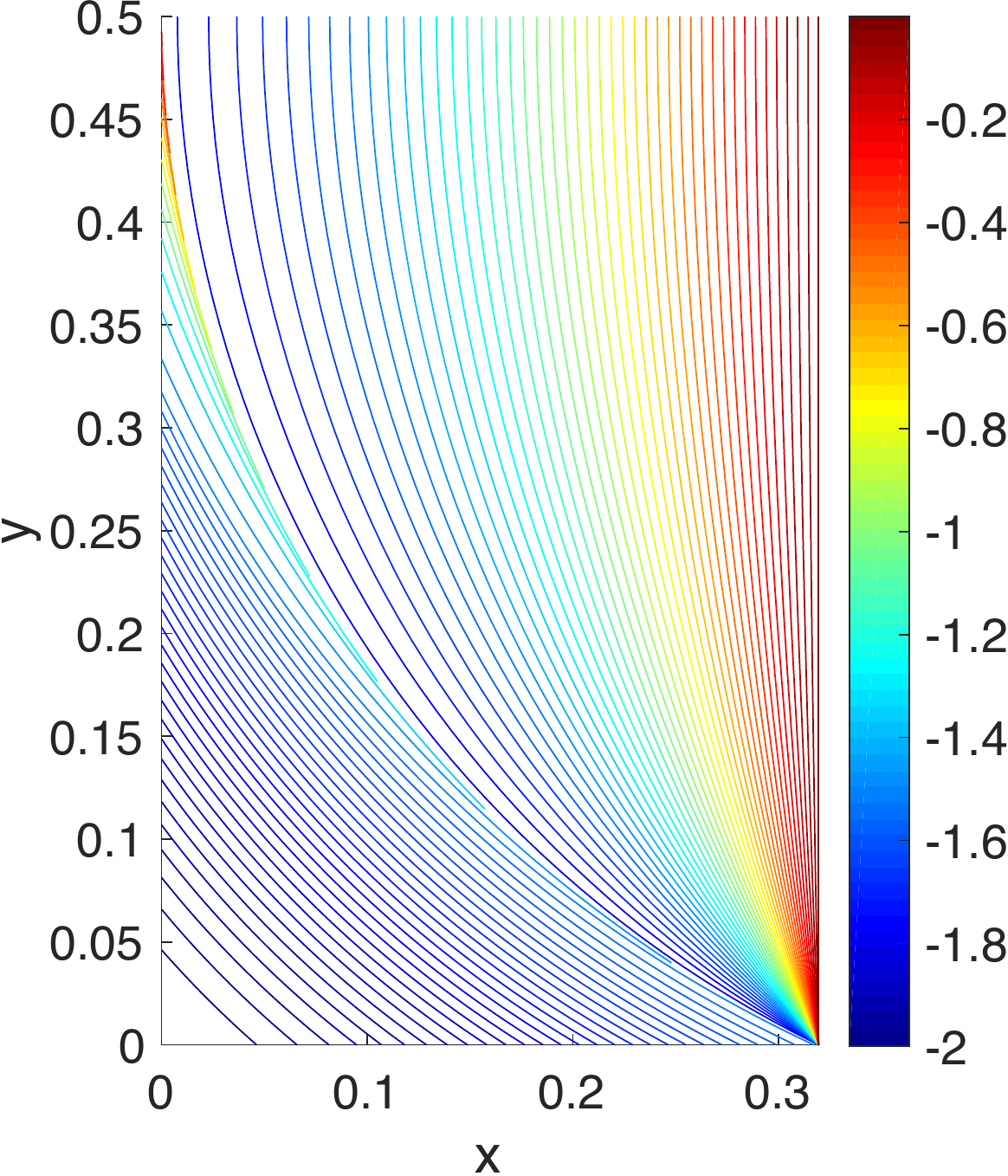}
    \caption{Level curves for the divergence of $u$, where $u$ is a solution obtained using characteristics. Here $L=1$.}
  \label{f:ctm-1}
\end{figure}
\begin{figure}[H]
\centering
    \includegraphics[width=2.5in]
                    {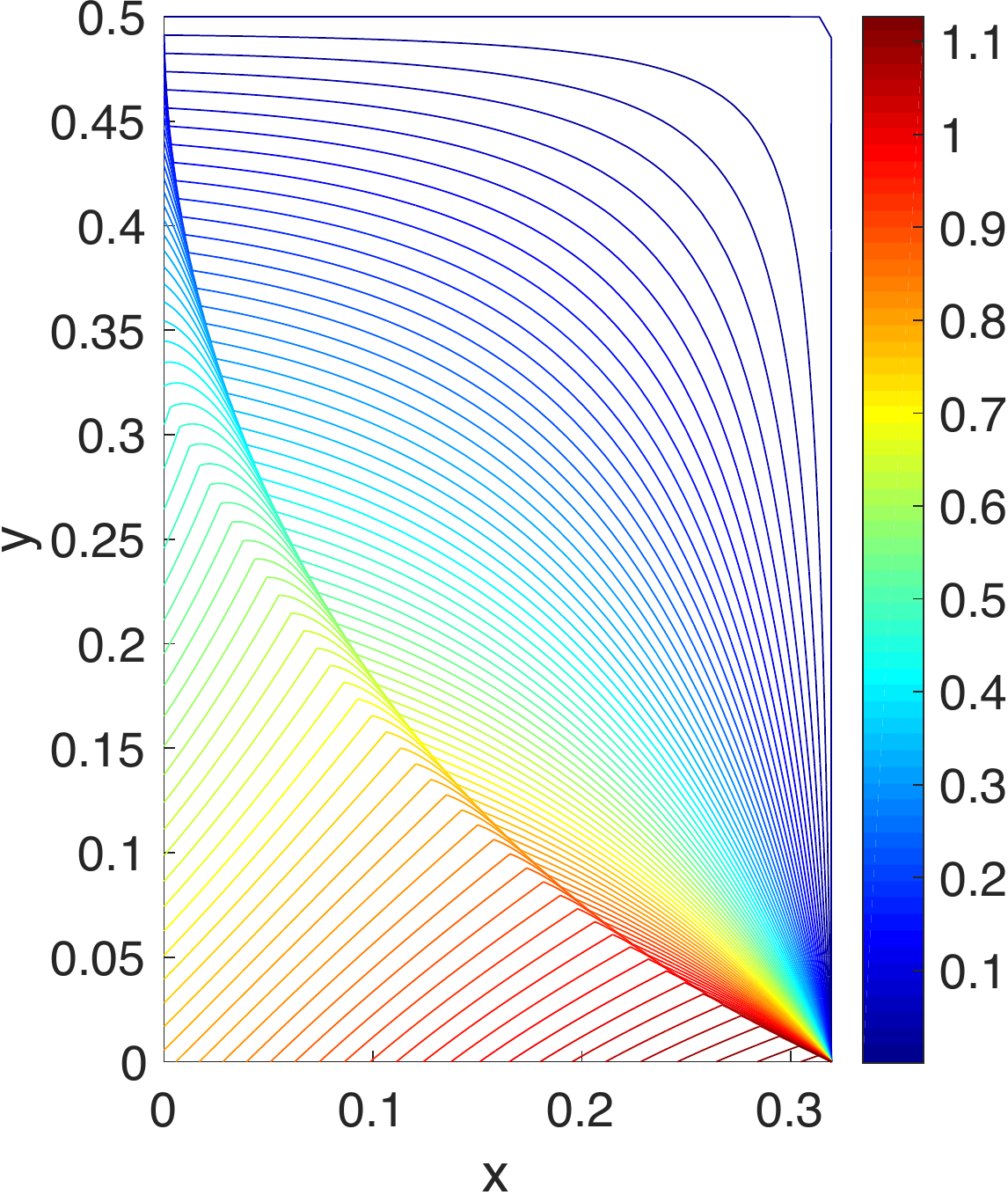}
    \caption{Level curves for the angle $\theta$, where $u=(\cos{\theta}, \sin{\theta})$ is a solution obtained using characteristics. Here $L=1$.}
  \label{f:ctm-2}
\end{figure}
\begin{figure}[H]
\centering
    \includegraphics[width=3.5in]
                    {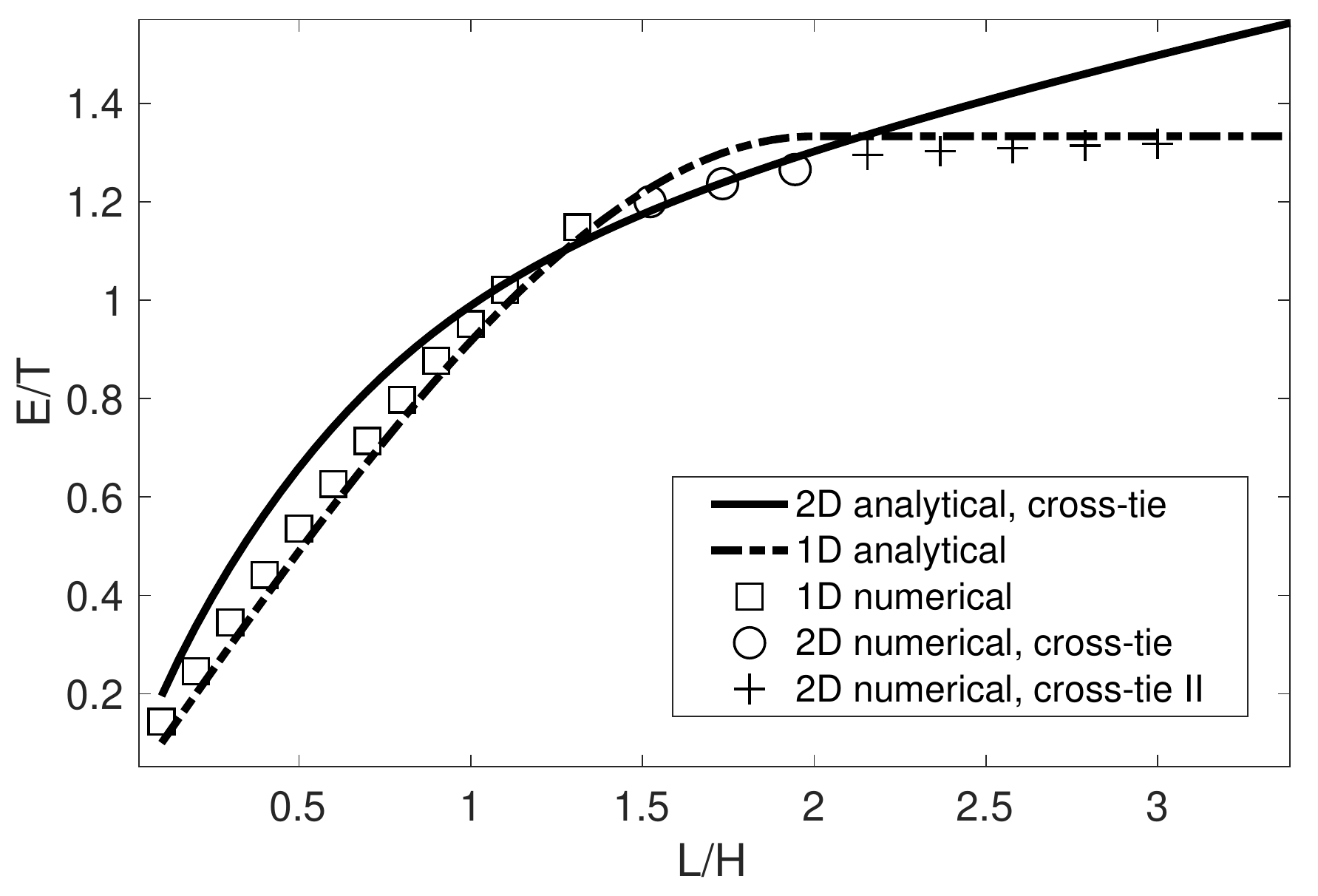}
    \caption{Energy per unit length.}
  \label{f:cten}
\end{figure}

Further numerical observations are in order. When the solution resulting from the characteristics construction is plotted (Figs.~\ref{f:ctm-1}-\ref{f:ctm-2}), it closely resembles those in Figs.~\ref{f:ct-2}-\ref{f:ct-3}. The markers in Fig.~\ref{f:cten} represent the energies of the numerically computed solutions to the Euler-Lagrange equations for $E_\varepsilon$, where the shape of the marker distinguishes the type of the energy-minimizing solution obtained in the simulations. We can observe a close correspondence between the numerics and analytical solutions as the squares and circles track well the one- and two-dimensional constructions, respectively. While the two-dimensional cross-tie construction discussed above has a smaller energy (both theoretically and numerically) on a short interval of $L$ values, it is then superseded by the two-dimensional cross-tie type II configurations of Figs.~\ref{f:ctt-1}-\ref{f:ctt-3}. Indeed, this configuration still has a smaller energy than the one-dimensional construction. The difference between the energies of the one-dimensional and the two-dimensional cross-tie type II constructions is small, however, and appears to decrease with an increasing $L$. 

We conclude with a few conjectures suggested by numerics. 

\medskip
\noindent\textbf{Conjecture 1:} For $0 < L/H < L_0,$ the one-dimensional minimizer from Theorem \ref{1dmain} is a unique minimizer of $E_0$ among all two-dimensional competitors.

\medskip
\noindent\textbf{Conjecture 2:} For $L/H \in (L_0,L_1),$ the critical point constructed in the proof of Theorem \ref{not1d} is a minimizer of $E_0.$ 

\medskip
\noindent\textbf{Conjecture 3:} For $L/H \geqslant L_1,$ there exist a two-dimensional minimizer $u_L$ with $E_0[u_L]$ that is lower than the minimum energy achieved over one-dimensional competitors. The difference in energies, however, vanishes in the $L \to \infty$ limit. Moreover, the unique cluster point of $u_L$ in $H_{\mathrm{div}}(\Omega;\mathbb{S}^1) \cap BV(\Omega;\mathbb{S}^1)$ is given by the piecewise constant vector field which equals $(1,0)$ for $y > 0$ and equals $(-1,0)$ for $y < 0.$ 

\section{Acknowledgements}
PS and RV acknowledge the support from NSF DMS-1101290 and NSF DMS-1362879. RV also acknowledges the support from an Indiana University College of Arts and Sciences Dissertation Year Fellowship. DG acknowledges the support from NSF DMS-1729538.

 
\bibliographystyle{acm} 

\bibliography{GSV} 

\begin{thebibliography}{10}

\bibitem{COMSOL}
{COMSOL} {Multiphysics\textregistered} {v. 5.3}.
\newblock http://www.comsol.com/.
\newblock {COMSOL} {AB}, {Stockholm}, {Sweden}.

\bibitem{MATLAB}
{MATLAB 9.3}, {R}elease name {R2017b}, {September 2017}.
\newblock The MathWorks, Inc.
\newblock Natick, Massachusetts, United States.

\bibitem{ARS}
{\sc Alouges, F., Rivi\`ere, T., and Serfaty, S.}
\newblock N\'eel and cross-tie wall energies for planar micromagnetic
  configurations.
\newblock {\em ESAIM Control Optim. Calc. Var. 8\/} (2002), 31--68.
\newblock A tribute to J. L. Lions.

\bibitem{ADM}
{\sc Ambrosio, L., De~Lellis, C., and Mantegazza, C.}
\newblock Line energies for gradient vector fields in the plane.
\newblock {\em Calc. Var. Partial Differential Equations 9}, 4 (1999),
  327--255.

\bibitem{AFP}
{\sc Ambrosio, L., Fusco, N., and Pallara, D.}
\newblock {\em Functions of bounded variation and free discontinuity problems}.
\newblock Oxford Mathematical Monographs. The Clarendon Press, Oxford
  University Press, New York, 2000.

\bibitem{AG}
{\sc Aviles, P., and Giga, Y.}
\newblock On lower semicontinuity of a defect energy obtained by a singular
  limit of the {G}inzburg-{L}andau type energy for gradient fields.
\newblock {\em Proc. Roy. Soc. Edinburgh Sect. A 129}, 1 (1999), 1--17.

\bibitem{BF}
{\sc Barroso, A.~C., and Fonseca, I.}
\newblock Anisotropic singular perturbations---the vectorial case.
\newblock {\em Proc. Roy. Soc. Edinburgh Sect. A 124}, 3 (1994), 527--571.

\bibitem{BBH}
{\sc Bethuel, F., Brezis, H., and H\'elein, F.}
\newblock {\em Ginzburg-{L}andau vortices}, vol.~13 of {\em Progress in
  Nonlinear Differential Equations and their Applications}.
\newblock Birkh\"auser Boston, Inc., Boston, MA, 1994.

\bibitem{CD}
{\sc Conti, S., and De~Lellis, C.}
\newblock Sharp upper bounds for a variational problem with singular
  perturbation.
\newblock {\em Math. Ann. 338}, 1 (2007), 119--146.

\bibitem{DO}
{\sc De~Lellis, C., and Otto, F.}
\newblock Structure of entropy solutions to the eikonal equation.
\newblock {\em J. Eur. Math. Soc. (JEMS) 5}, 2 (2003), 107--145.

\bibitem{DA}
{\sc DeBenedictis, A., and Atherton, T.~J.}
\newblock Shape minimisation problems in liquid crystals.
\newblock {\em Liquid Crystals 43}, 13-15 (2016), 2352--2362.

\bibitem{DKMO2}
{\sc Desimone, A., Kohn, R.~V., M\"uller, S., and Otto, F.}
\newblock Repulsive interaction of {N}\'eel walls, and the internal length
  scale of the cross-tie wall.
\newblock {\em Multiscale Model. Simul. 1}, 1 (2003), 57--104.

\bibitem{DKMO}
{\sc DeSimone, A., M\"uller, S., Kohn, R.~V., and Otto, F.}
\newblock A compactness result in the gradient theory of phase transitions.
\newblock {\em Proc. Roy. Soc. Edinburgh Sect. A 131}, 4 (2001), 833--844.

\bibitem{E}
{\sc Ericksen, J.~L.}
\newblock Liquid crystals with variable degree of orientation.
\newblock {\em Arch. Rational Mech. Anal. 113}, 2 (1990), 97--120.

\bibitem{Go}
{\sc Golovaty, D.}
\newblock On a {$\Gamma$}-limit of a family of anisotropic singular
  perturbations.
\newblock {\em Manuscripta Math. 92}, 4 (1997), 515--524.

\bibitem{Helein}
{\sc H\'elein, F.}
\newblock Minima de la fonctionnelle \'energie libre des cristaux liquides.
\newblock {\em C. R. Acad. Sci. Paris S\'er. I Math. 305}, 12 (1987), 565--568.

\bibitem{Ignat}
{\sc Ignat, R.}
\newblock Singularities of divergence-free vector fields with values into
  {$\mathbb S^1$} or {$\mathbb S^2$}. {A}pplications to micromagnetics.
\newblock {\em Confluentes Math. 4}, 3 (2012), 1230001, 80.

\bibitem{JOP}
{\sc Jabin, P.-E., Otto, F., and Perthame, B.}
\newblock Line-energy {G}inzburg-{L}andau models: zero-energy states.
\newblock {\em Ann. Sc. Norm. Super. Pisa Cl. Sci. (5) 1}, 1 (2002), 187--202.

\bibitem{JinKohn}
{\sc Jin, W., and Kohn, R.~V.}
\newblock Singular perturbation and the energy of folds.
\newblock {\em J. Nonlinear Sci. 10}, 3 (2000), 355--390.

\bibitem{KohnICM}
{\sc Kohn, R.~V.}
\newblock Energy-driven pattern formation.
\newblock In {\em International {C}ongress of {M}athematicians. {V}ol. {I}}.
  Eur. Math. Soc., Z\"urich, 2007, pp.~359--383.

\bibitem{LO}
{\sc Lamy, X., and Otto, F.}
\newblock On the regularity of weak solutions to {B}urgers' equation with
  finite entropy production.
\newblock {\em Calc. Var. Partial Differential Equations 57}, 4 (2018), Art.
  94, 19.

\bibitem{Lorent}
{\sc Lorent, A.}
\newblock A simple proof of the characterization of functions of low {A}viles
  {G}iga energy on a ball {\it via} regularity.
\newblock {\em ESAIM Control Optim. Calc. Var. 18}, 2 (2012), 383--400.

\bibitem{NewLorent}
{\sc Lorent, A.}
\newblock A quantitative characterisation of functions of low {A}viles {G}iga
  energy in convex domains.
\newblock {\em Ann. Sc. Norm. Super. Pisa Cl. Sci. (5) 13}, 1 (2014), 1--66.

\bibitem{MN}
{\sc Mottram, N.~J., and Newton, C.~J.}
\newblock Introduction to {Q}-tensor theory.
\newblock {\em arXiv preprint arXiv:1409.3542\/} (2014).

\bibitem{RS}
{\sc Rivi\`ere, T., and Serfaty, S.}
\newblock Compactness, kinetic formulation, and entropies for a problem related
  to micromagnetics.
\newblock {\em Comm. Partial Differential Equations 28}, 1-2 (2003), 249--269.

\bibitem{Te}
{\sc Temam, R.}
\newblock {\em Navier-{S}tokes equations. {T}heory and numerical analysis}.
\newblock North-Holland Publishing Co., Amsterdam-New York-Oxford, 1977.
\newblock Studies in Mathematics and its Applications, Vol. 2.

\bibitem{V}
{\sc Virga, E.~G.}
\newblock {\em Variational theories for liquid crystals}, vol.~8 of {\em
  Applied Mathematics and Mathematical Computation}.
\newblock Chapman \& Hall, London, 1994.

\bibitem{ZN}
{\sc Zhou, S., Nastishin, Y.~A., Omelchenko, M.~M., Tortora, L., Nazarenko,
  V.~G., Boiko, O.~P., Ostapenko, T., Hu, T., Almasan, C.~C., Sprunt, S.~N.,
  Gleeson, J.~T., and Lavrentovich, O.~D.}
\newblock Elasticity of lyotropic chromonic liquid crystals probed by director
  reorientation in a magnetic field.
\newblock {\em Phys. Rev. Lett. 109\/} (Jul 2012), 037801.

\end{thebibliography}

\end{document}